%% file: ridgerunner.tex
\documentclass[12pt,letterpaper]{amsart}
\usepackage{times,graphicx,amsthm}
\usepackage[percent]{overpic}
\usepackage{varioref}
\pdfoutput=1
\usepackage{placeins}
\usepackage[ruled,noline]{algorithm2e}
\usepackage[matrix,arrow]{xy}
\usepackage{array}
\usepackage{longtable}
\usepackage{dcolumn}
\usepackage{booktabs}
\usepackage{fullpage}
\usepackage{epic,eepic}
\usepackage{fp}
\usepackage{units}
\usepackage{url}

\renewcommand{\subsection}[1]{\medskip\noindent\textbf{#1.} }
\def\marginpar#1{}   

\let\lbl=\label
\def\label#1{\lbl{#1}\ifinner\else\marginpar{\ref{#1} #1}\ignorespaces\fi}

\newcommand\R{\mathbb{R}}
\newcommand\<{\langle}
\newcommand\cross{\times}
\newcommand\dcsd{\operatorname{dcsd}}
\newcommand\Thi{\operatorname{Thi}}
\newcommand\Rop{\operatorname{Rop}}
\newcommand\PThi{\operatorname{Thi}_p}
\newcommand\PRop{\operatorname{Rop}_p}

\newcommand\MinRad{\operatorname{MinRad}}
\newcommand\Eq{\operatorname{Eq}}
\newcommand\grad{\nabla}

\providecommand{\abs}[1]{\lvert#1\rvert} 
\providecommand{\norm}[1]{\lVert#1\rVert}

\renewcommand{\phi}{\varphi}

\newcommand{\Len}{\operatorname{Len}}

\newcommand{\VB}{\operatorname{VB}}

\newcommand{\V}{\mathcal{V}}
\newcommand{\IM}{\operatorname{I}}
\newcommand{\RM}{\operatorname{R}}

\providecommand{\abs}[1]{\lvert#1\rvert} 
\providecommand{\norm}[1]{\lVert#1\rVert}

\newtheorem{theorem}{Theorem}[section]
\newtheorem{lemma}[theorem]{Lemma}
\newtheorem{corollary}[theorem]{Corollary}
\newtheorem{proposition}[theorem]{Proposition}

\theoremstyle{definition}

\newtheorem{definition}[theorem]{Definition}

\newcommand{\vb}{\operatorname{vb}}

\newcommand{\co}{\operatorname{:}}
\newcommand{\CThi}{\operatorname{CThi}}

\newcommand{\Strut}{\operatorname{Strut}}
\newcommand{\Kink}{\operatorname{Kink}}

\input{latex_knot_prev.tex}

\input{latex_knot_prevcite.tex}

\input{KnotTags.tex}

\newcommand{\SimpleChainRRUB}{41.7086588}

\input{fileCounts.tex}

\graphicspath{{./figs/}}

{\makeatletter
 \gdef\xxxmark{%
   \expandafter\ifx\csname @mpargs\endcsname\relax 
     \expandafter\ifx\csname @captype\endcsname\relax 
       \marginpar{xxx}
     \else
       xxx 
     \fi
   \else
     xxx 
   \fi}
 \gdef\xxx{\@ifnextchar[\xxx@lab\xxx@nolab}
 \long\gdef\xxx@lab[#1]#2{{\bf [\xxxmark #2 ---{\sc #1}]}}
 \long\gdef\xxx@nolab#1{{\bf [\xxxmark #1]}}
}

\newcommand{\TAUCS}{\texttt{TAUCS}}
\newcommand{\tsnnls}{{\texttt{tsnnls}}}
\newcommand{\tlsqr}{\texttt{tlsqr}}
\newcommand{\ridgerunner}{{\texttt{ridgerunner}}}
\newcommand{\octrope}{{\texttt{octrope}}}

\begin{document}
\bibliographystyle{plain}

\title{Knot tightening by constrained gradient descent}
\date{June 12, 2004; Revised: \today}

\author{Ted Ashton}

\author{Jason Cantarella}

\author{Michael Piatek}

\author{Eric Rawdon}

\begin{abstract}
  We present new computations of approximately length-minimizing polygons with fixed
  thickness. These curves model the centerlines of ``tight'' knotted tubes with minimal length and fixed circular cross-section. Our curves approximately minimize the ropelength (or quotient of length and thickness) for polygons in their knot types. While previous authors have minimized ropelength for polygons using simulated annealing, the
  new idea in our code is to minimize length over the set of polygons
  of thickness at least one using a version of constrained gradient
  descent.

  We rewrite the problem in terms of minimizing the length of the
  polygon subject to an infinite family of differentiable constraint
  functions. We prove that the polyhedral cone of variations of a
  polygon of thickness one which do not decrease thickness to first
  order is finitely generated, and give an explicit set of generators.
  Using this cone we give a first-order minimization procedure and a
  Karush-Kuhn-Tucker criterion for polygonal ropelength criticality.

  Our main numerical contribution is a set of \FileCount\  almost-critical knots and links, including all prime knots with ten and fewer crossings and all prime links with nine and fewer crossings. For links, these are the first published ropelength figures, and for knots they improve on existing figures. We give new maps of the self-contacts of these knots and links, and discover some highly symmetric tight knots with particularly simple looking self-contact maps.
\end{abstract}

\keywords{ropelength, tight knots, ideal knots, constrained gradient
  descent, sparse non-negative least squares problem (snnls),
  knot-tightening}

\maketitle

\section{Introduction}

\subsection{Overview}
Knots tied in rope are flexible machines which organize tensions and
contact forces to bind tightly and resist unravelling. As a
technology, knots have proved remarkably effective. For this reason
there is a vast body of knowledge about their practical uses. Yet in
many ways, the design of these machines remains mysterious. As early
as 1987 Maddocks and Keller were able to study different types of
hitches and predict their holding power by an analysis of their
equilibrium shapes~\cite{MR89g:73037}. But these shapes were rather
simple, and there was no way to infer the structures of more
complicated knots from these examples. It was obvious that what was
needed was data, and by the end of the century a series of numerical
experiments in knot-tightening were
underway~\cite{rawdonthesis,MR1702021,MR1702022,MR1953340}. This paper
describes a new computational approach to knot-tightening which yields
improved numerical results (a preliminary report on some of our findings appeared in the conference proceedings~\cite{cprvis}). To build our method, we derive some new
results in the theory of ropelength for polygonal knots.

\subsection{Defining the problem}
Given any space curve $\gamma$, we can define the \emph{thickness}
$\Thi(\gamma)$ of $\gamma$ to be the supremal $\epsilon$ for which any
point in an $\epsilon$-neighborhood of $\gamma$ has a unique nearest
neighbor on the curve\footnote{Federer referred to this number as the
\emph{reach} of $\gamma$~\cite{MR22:961}.}. Any curve with nonzero
thickness is $C^{1,1}$ (that is, is $C^1$ with a Lipschitz first
derivative)~\cite{MR22:961,MR2003h:58014}. Given this, it has been
shown that
\begin{proposition}[\cite{MR99k:57025}]
  If $\gamma$ is a $C^1$ curve, then the thickness $\Thi(\gamma)$ is
  given by the supremal radius of all embedded tubes formed by
  taking the union of disks of uniform radius centered on $\gamma(s)$ in the planes
  normal to $\gamma'(s)$.
\end{proposition} 
This idea of thickness was first proposed by Kr\"otenheerdt and Veit
in 1976~\cite{MR55:9069,MR2197931} and was rediscovered in the 1990's by
Nabutovsky~\cite{MR96d:58025} and Buck and Orloff~\cite{MR95k:58037}.
The thickness can be used to define a scale-invariant quantity
called~\emph{ropelength}:
\begin{definition}
  \label{def:ropelength}
  The \emph{ropelength} of a curve $\gamma$ is defined by
  \begin{equation*}
    \Rop(\gamma) = \frac{\Len(\gamma)}{\Thi(\gamma)},
  \end{equation*}
  where $\Len(\gamma)$ is the length of $\gamma$.  The \emph{minimal
  ropelength} of a knot or link type $L$, $\Rop(L)$, is the minimal
  ropelength of all curves in that knot or link type.
\end{definition}

The knot tightening problem is to find and describe the minimal
ropelength curves in a given knot type. It is known that such curves
exist, but their exact shapes are currently the subject of active
mathematical
research~(c.f.\,\cite{MR2003j:57010,MR2002m:74035,MR2003h:58014}).
Once found (or computed to sufficient accuracy), these configurations
have been used to predict the relative speed of DNA knots under gel
electrophoresis~\cite{MR99b:92032}, the pitch of double helical DNA~\cite{mbms}, the average values of different spatial measurements
of random knots~\cite{MR1981020}, and the breaking
points of knots~\cite{1367-2630-3-1-310}. They also provide a model for the
structure of a class of subatomic particles known as
glueballs~\cite{kepglue}.

\subsection{Another form of the problem}
Let $\gamma \co S^1 \rightarrow \R^3$ now be a $C^2$ parametrized curve,
and define the self-distance function $d \co S^1 \cross S^1
\rightarrow \R$ of $\gamma$ by $d(s,t) := \norm{\gamma(s) -
\gamma(t)}$. As usual, let $\kappa(s)$ denote the curvature of
$\gamma$. We then define the set $\dcsd(\gamma)$ of \emph{doubly-critical
  self-distances} to be the set of critical points of $d$ with $s \neq
  t$. Taking the partial derivatives of $d$, we see that $(s,t) \in
  \dcsd(\gamma)$ if and only if
\begin{equation*}
  \< \gamma(s) - \gamma(t), \gamma\,'(s) \rangle = 0 \text{ and } 
  \< \gamma(s) - \gamma(t), \gamma\,'(t) \rangle = 0.
\end{equation*}

A key idea in~\cite{MR99k:57025} is that for any $\tau <
\Thi(\gamma)$, the surface of the tube of radius $\tau$ around
$\gamma$ has no self-intersections and is $C^2$ smooth. But when $\tau
= \Thi(\gamma)$, the tube is pinched or has a tangential
self-intersection.  This leads to an alternate characterization of
thickness:
\begin{theorem}[\cite{MR99k:57025}]
  \label{thm:lsdr}
  The thickness of $\gamma$ is the minimum of 
 \begin{equation*}
   \min_s \frac{1}{\kappa(s)} \text{ and } \min_{(s,t) \in
   \dcsd(\gamma)} \frac{d(s,t)}{2}.
 \end{equation*}
\end{theorem}
\noindent Figure~\ref{fig:thick} shows curves where the first and second of these terms control the thickness.
\begin{figure}
  \includegraphics{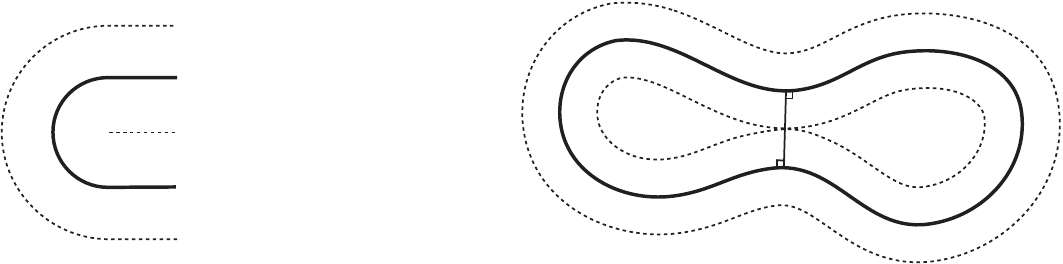}
  \caption{The thickness of a smooth curve $\gamma$ is controlled by
    curvature (as in the left picture), and the length of chords in
    $\dcsd(\gamma)$ (as in the right picture).}
  \label{fig:thick}
\end{figure}

Since length and thickness scale together, minimizing ropelength is
the same as minimizing length over the set of curves with thickness at
least one. Since thickness is a min-function, the condition
$\Thi(\gamma) \geq 1$ can be viewed as an infinite family of
inequality constraints on~$\gamma$. These constraints are active at
places where the tube around~$\gamma$ forms kinks (where $1/\kappa$ is in control of the minimum in Theorem~\ref{thm:lsdr}) or has self-contacts (where the self-distance~$\nicefrac{d(s,t)}{2}$ is in control of the minimum).

\subsection{Numerical approaches to the knot-tightening problem}
Previous authors have defined discretized versions of thickness
for polygons or spline curves and viewed the
problem as one of minimizing the nonsmooth quotient of length and thickness. The advantage of this approach is that it is a very simple and robust way to obtain approximately ropelength-minimizing curves. The disadvantage is that it is very difficult to take advantage of the
fact that thickness (as given in Thm~\ref{thm:lsdr}) is a min-function.

Our approach is to
define a discrete version of thickness as a min-function and think of
the problem as one of minimizing a differentiable function $\Len(\V)$
subject to a family of differentiable constraints $\PThi(\V) \geq 1$.
While our approach will not quite fit into the standard framework of
constrained optimization (our family of constraints is infinite), we
will be able to define a version of constrained gradient descent which
minimizes polygonal ropelength effectively.

\subsection{Theoretical framework} 
For an equilateral space polygon $\V$ we first prove that our function
$\PThi(\V)$ can be written as a min over a fixed compact family of
differential functions. From here we use Clark's theorem to show that
$\PThi$ has a one-sided derivative in the direction of any variation
$W$ of $\V$. For a polygon with $\PThi(\V) = 1$ we use these
derivatives to define a cone of infinitesimal variations $\IM(\V)$ which
do not decrease $\PThi$ to first order and the dual cone of ``resolvable''
variations $\RM(\V)$. Our next main theorem is that $\RM(\V)$ is a
finitely generated polyhedral cone whose generators are the gradients
of the lengths of certain chords of the polygon (called
\emph{struts}) and of a function of certain turning angles of the
polygon (called \emph{kinks}). We give explicit formulae for these
gradients in terms of the vertex positions. We then compute the
gradient of~$\Len(\V)$ and define the constrained gradient of length
to be the projection of~$\Len(\V)$ onto the polyhedral cone~$\IM(\V)$. At this point we give the expected result that a polygon is
critical for polygonal ropelength if and only if the constrained
gradient of length is zero. Equivalently, a polygon is critical for
polygonal ropelength if there is a set of positive Lagrange
multipliers on the struts and kinks which combine to equal the
negative of the length gradient. The theory section ends with a
discussion of how to compute the constrained gradient numerically.

\subsection{Numerical methods} 
Sections~\ref{sec:bridge} and~\ref{sec:programdesign} describe the design of our polygonal ropelength minimizing software. Our algorithm essentially consists of computing the
constrained gradient of length and taking small steps in this
direction until the constrained gradient is sufficiently
small. However, the details of the process are not quite so
simple. Since the constraint functions are nonlinear, even steps
that are in the direction of the constrained gradient violate some
constraints to second order. Further, newly active constraints are
discovered throughout the run as previously distant sections of tube
come into contact with one another. As a result, we must choose
stepsizes carefully and correct errors periodically. It is also
important to run efficiently, as the size of our problem (about one
thousand variables and a similar number of active constraints) is
fairly large. We have solved these technical and engineering problems and used our software to minimize all prime knots with ten or fewer crossings and all prime links with nine or fewer crossings, for a total of \FileCount\  different knot and link types. We intend to address the ropelength of composite knots and links in a future publication.

\newcommand{\FourOneImpr}{-}
\FPeval{\FourOneImpr}{100.0*((\FourOnePrev - \FourOneRRUB)/(\FourOnePrev))}
\FPround{\FourOneImpr}{\FourOneImpr}{2}

\newcommand{\NineTwentyImpr}{-}
\FPeval{\NineTwentyImpr}{100.0*((\NineTwentyPrev - \NineTwentyRRUB)/(\NineTwentyPrev))}
\FPround{\NineTwentyImpr}{\NineTwentyImpr}{2}

\FPeval{sixthreetworesult}{0.0017}
\FPeval{simplechainresult}{100*((\SimpleChainRRUB - 41.6991)/(41.6991))}
\FPround{\simplechainresult}{\simplechainresult}{2} 
\subsection{New ropelength bounds} 
We check our figures against previous computations of the minimum
ropelength of knots and links and against some of the few known
theoretical results for the lengths of tight links. Our results
improve on all previously published computational results except for the
trefoil knot. For example, we improve the best known upper bound for
the ropelength of the well-studied figure-eight knot $4_1$ by
$\FourOneImpr$ to $\FourOneRRUB$ (as compared to the bound
of~\cite{\FourOnePrevCite}) and improve the best known upper bound for
the ropelength of the $9_{20}$ knot by $\NineTwentyImpr\%$ to
$\NineTwentyRRUB$ (compared to the bound
of~\cite{\NineTwentyPrevCite}). To get a sense of the difference between the configurations produced by our method and the configurations produced by the simulated annealer of~\cite{\NineTwentyPrevCite} we show both configurations in Figure~\ref{NineTwentyFig}. For links, our figures are the first
computational results to appear in print, but compare well to known theoretical results.  For example, the upper bound provided by our computation of the Borromean rings link $6^3_2$ is $\SixThreeTwoRRUB$ --- within
$\sixthreetworesult\%$ of the exact value around $58.0060$ suggested
by~\cite{MR2284052}, while our computation of the tight shape of the
``simple chain'' link is $\SimpleChainRRUB$ --- within
$\simplechainresult\%$ of the correct value of of $6 \pi +
2$~\cite{MR2003h:58014}. 

We also compared our results to those of Gilbert~\cite{gilbert}, which are unpublished but available on Bar-Natan's \emph{Knot Atlas} wiki. Gilbert provides Fourier cofficients and instructions for reconstructing the vertices of his configurations from this data. We followed his instructions, but our software did not verify his claimed ropelength numbers\footnote{Our measurement of curvature by $\MinRad$ is sensitive to edgelength and seems to come out much larger than his ropelengths would indicate. This is probably a discretization effect, and it is certainly very possible that the Fourier knots defined by Gilbert's data have ropelengths corresponding to Gilbert's claimed numbers.}. According to our measurement of the ropelength of Gilbert's configurations, our knots are tighter in all cases but $2^2_1$ by an average of $\AverageWin\%$, with some outliers, such as our $9^2_{37}$ link, which is $71\%$ shorter. If we compare our results to Gilbert's claimed ropelengths, our knots and links are tighter in $\WonClaimCases$ cases and less tight in $\LostClaimCases$. Overall, our knots and links are (on average) $\AverageClaimWin\%$ tighter than the bounds claimed by Gilbert with our $9^2_{28}$ link about $4\%$ shorter than Gilbert's claim.

\begin{figure}
\includegraphics[width=2.2in]{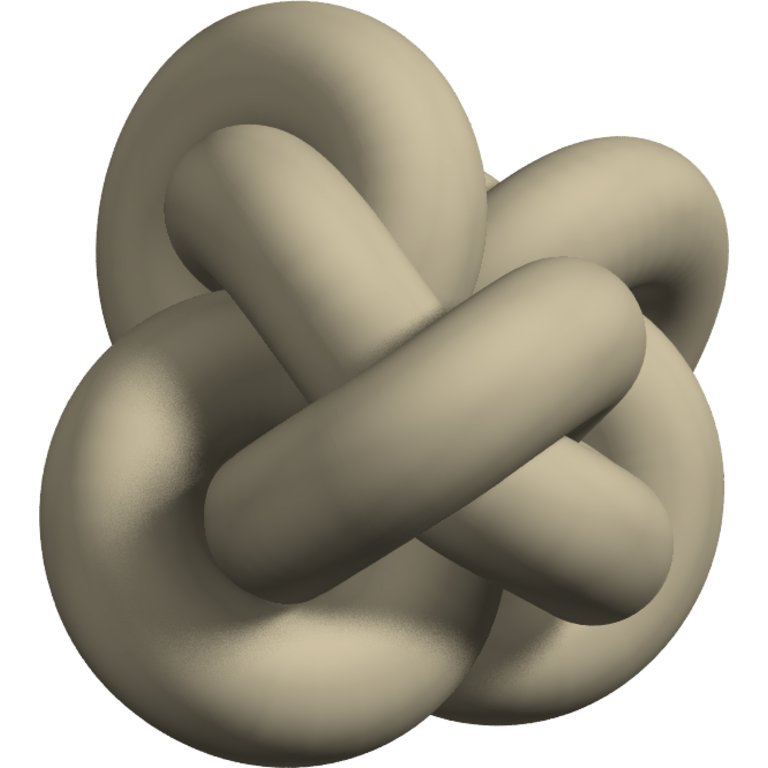} \hspace{0.4in}
\includegraphics[width=3in,trim = 0 1.0in 0 0]{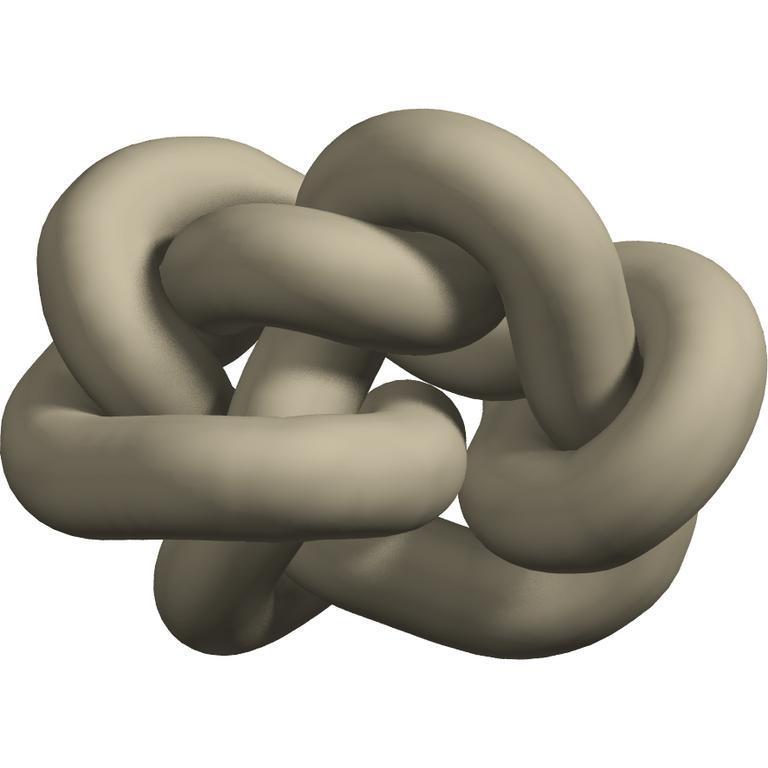}
\caption{These two images of the $9_{20}$ knot show the tightest configurations obtained by our algorithm (left) and by the TOROS algorithm described in~\cite{\NineTwentyPrevCite} (right). It is clear that our algorithm performs better once there are many self-contacts in the knot. In fact, the ropelength of the left-hand configuration is bounded by $\NineTwentyRRUB$, while the configuration on the right has ropelength bounded by $\NineTwentyPrev$.}
\label{NineTwentyFig}
\end{figure}

\subsection{Self-contact maps}
Two sets of authors (von der Mosel et\,al.\,\cite{MR2033143} and
Cantarella et\,al.\,\cite{MR2284052}) have given versions of a ropelength
criticality criterion for knots without kinks which state roughly that
a knot $\gamma$ is ropelength-critical when the elastic force given by
the gradient of the length of the curve is balanced by a system of
Lagrange multipliers on the self-contacts of the tube around
$\gamma$. The latter authors used their condition to derive a
ropelength critical configuration of the Borromean rings
and a surprising ropelength critical configuration of a clasp formed by two tubes stretched across each other.

In each of these examples, the most difficult part of the result was
the deduction of the structure of the set of self-contacts for the
tight configuration. Since these contact maps are very sensitive to
small perturbations of the centerline, it has been difficult to
resolve them using previous numerical methods\footnote{The notable exception
to this rule has been the ``biarc'' spline-annealing method
of~\cite{mglob}, which has produced well-resolved contact maps for
the $3_1$ and $4_1$ knots.}. These contacts and the system of
Lagrange multipliers on them are explicitly computed by our algorithm,
allowing us to give medium-quality contact maps for a large number of
knots and links. The contact maps offer some support for the
hypothesis that a relatively small number of structures may reappear
often in tight knots and links.

\subsection{Previous work}
This is not the first time gradient-like methods have been attempted
for the knot-tightening problem. Our work has been inspired by Piotr
Pieranski's SONO algorithm~\cite{MR1702021}, which follows a version
of the length gradient, but does not include an explicit resolution of
this vector against the active constraints.  Our thinking is also
informed by John Sullivan's ``energy-ropelength method''~\cite{MR1953340}, which
optimizes thickness instead of length, estimating the maximum diameter
of a uniform embedded tube around the core curve by an $L^p$ average
of the radii of embedded cross-sectional disks and minimizing the
resulting smooth functional using the conjugate-gradient
implementation in Brakke's \texttt{evolver}~\cite{MR1203871}.

\section{A discretization for the ropelength problem}
\subsection{Polygonal thickness}
Consider a closed space polygon $\V$ with vertices $v_1, \dots, v_{V}$
and edges $e_1, \dots, e_{V}$.  We will think of $\V$ as the vector
$(v_1,\dots,v_{V})$ in $(\R^3)^V = \R^{3V}$, and assume that all subscripts on vertices and edges are taken mod $V$. The unit tangent vector~$T_i$ to each edge of a polygon is well-defined on the interior of the
edge. At the vertex $v_i$ joining edges $e_{i-1}$ and $e_i$, there are
two tangent vectors $T_{i-1}$ and $T_i$. The curvature of $\V$ at
$v_i$ is usually thought of as a delta function whose mass is given by
the turning angle $\theta_i$ from $T_{i-1}$ to $T_i$. We will use a
somewhat different definition of curvature for polygons:
\begin{definition}
  \label{def:minrad}
  The minimum radius of curvature (or $\MinRad$) of $\V$ at
  $v_i$ is given by the radius of the unique circle that is tangent to the two edges
  meeting at $v_i$ and that touches the midpoint of the shorter one.
\end{definition}

Rawdon has shown \cite{rawdonthesis} that if $\theta_i$ is the turning angle of
$\V$ at $v_i$, then we can give $\MinRad(v_i)$ (and define
$\MinRad^\pm(v_i)$) by the expressions:
\begin{equation}
  \label{eq:minrad}
  \frac{\min\{\abs{e_{i-1}},\abs{e_i}\}}{2\,\tan(\nicefrac{\theta_i}{2})} =
  \min\left\{ \frac{\abs{e_{i-1}}}{2\,\tan(\nicefrac{\theta_i}{2})},
  \frac{\abs{e_{i}}}{2\,\tan(\nicefrac{\theta_i}{2})} \right\} = \min \{
    \MinRad^-(v_i), \MinRad^+(v_i) \}.
\end{equation}
It is clear that while $\MinRad v_i$ is not neccesarily a
differentiable function, the two functions $\MinRad^{\pm} v_i$ are
differentiable when they are defined.  The motivation for this
definition is that we can round off all the corners of $\V$ by
splicing in these circle arcs, generating a $C^{1,1}$ curve with radii of
curvature equal to the $\MinRad(v_i)$. We could have defined
$\PThi(\V)$ to be the thickness of this curve. It turns out, however,
that there is no closed form computation for that number (though it
can be computed approximately, as we will see in
Section~\ref{sec:RRUB}).

We now define a set corresponding to $\dcsd$ for polygons:
\begin{definition}
  Let $\dcsd(\V)$ be the set of $(p,q)$ on $\V$ with $p \neq q$ which
  are local minima of the self-distance function on $\V$.
\end{definition}
There are several possible cases for $(p,q)$ in $\dcsd(\V)$,
since the polygon might have a vertex at one or both of the endpoints
of the chord. These are shown in Figure~\ref{fig:pdcsd}.
\begin{figure}[ht]
  \includegraphics{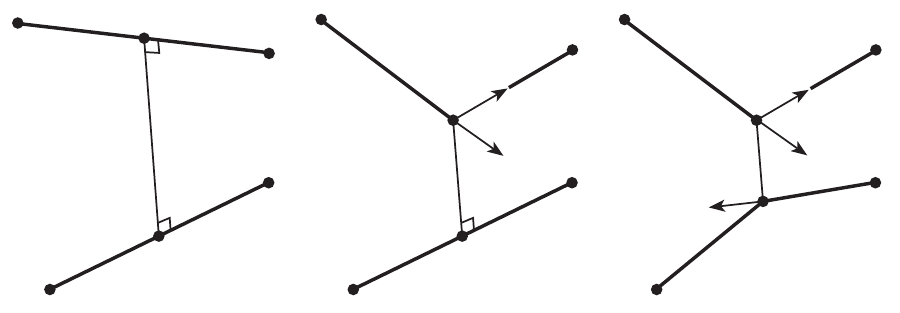}
  \caption[Classifying pairs in $\dcsd(\V)$] {We see three types of
    local minima of the self-distance function on a space polygon $\V$ in the three-dimensional drawings above. From
    left to right, these are an \emph{edge-edge} pair, a
    \emph{vertex-edge} pair, and a \emph{vertex-vertex} pair.}
  \label{fig:pdcsd}
\end{figure}

We can then define Rawdon's polygonal thickness:

\begin{definition}
  \label{def:polythi}
  The polygonal thickness $\PThi(\V)$ of a space polygon $\V$ without
  self-intersections is given by the minimum of
  \begin{equation*}
    \PThi(\V) := \min \left\{ \min_i \MinRad(v_i), \min_{(p,q) \in
      \dcsd(\V)} \frac{d(p,q)}{2} \right\}.
  \end{equation*}
\end{definition}
We have carefully constructed this definition so that when polygons~$\V_n$ with increasing numbers of
edges are inscribed in a space curve~$\gamma$ under some mild geometric hypotheses, $\PThi(\V_n)
\rightarrow \Thi(\gamma)$~\cite{rawdonthesis,MR1702029,MR2034393}.

\subsection{The problem with $\PThi$}
Definition~\ref{def:polythi} allows us to define the set of polygons
with $\PThi(\V) \geq 1$ as the polygons obeying a family of
constraints in the form $\MinRad(v_i) \geq 1$ and $d(p,q) \geq 2$ for
$(p,q) \in \dcsd(\V)$. This is almost the standard form for constrained
optimization problems:
\begin{equation}
  \label{eq:nlo}
  \min_{\V \in \R^{3V}} f(\V) \text{ subject to }
  g_i(\V) \geq 0,
\end{equation}
where $f$ and the $g_i$ are differentiable. The problem is that the
set of constraint functions $d(p,q)$ for $(p,q) \in \dcsd(\V)$ depends on the polygon. We will
need a common set of constraint functions for all polygons in a
neighborhood of a solution. 

\subsection{Constraint thickness} 
To solve this problem, we will define a new thickness measure for
polygons called \emph{constraint thickness} which is given
in the form above. We will then prove that for equilateral polygons, the new constraint
thickness defines the same set of polygons as the old polygonal
thickness.

We first define a subset of the pairs of points on a polygon
\begin{definition} 
  \label{def:vb}
  For a given positive $\tau$ and $\ell$, let $\theta(\tau,\ell)$ be
  the turning angle of a pair of edges of length $\ell$ with $\MinRad =
  \tau$. We set
  \begin{equation*}
    \VB(\tau,\ell) = \left\{ (p,q) \in \V \cross \V : \vb(p,q) \geq
    \frac{\pi}{\theta(\tau,\ell)} \right\},
  \end{equation*}
  where $\vb(p,q)$ is the smaller number of vertices between
  points $p$ and $q$ (counting $p$ and/or $q$ if they are vertices) if
  they are on the same connected component of $\V$ and
  $\infty$ otherwise. 
\end{definition}
We note that an easy computation shows that $\theta(\tau,\ell) = 2
\arctan(\nicefrac{\ell}{2 \tau})$.  We can now define our new thickness measure
\begin{definition}
  \label{def:CThi} 
  The ($\tau$,$\ell$)-constraint thickness $\CThi(\tau,\ell,\V)$ of a
  polygon $\V$ is given by
  \begin{equation}
    \CThi(\tau,\ell,\V) = \min \left\{ \min \frac{\MinRad(v_i)}{\tau}, 
      \min_{(p,q) \in \VB(\tau,\ell)} \frac{d(p,q)}{2} \right\}.
  \end{equation}
\end{definition}
We note that $\V$ need not be equilateral or have edgelength $\ell$ to define the constraint thickness to defined the constraint thickness of $\V$. We can view $\tau$ as the ``stiffness'' of the rope~(c.f. the definition of $\lambda$-thickness in~\cite{cfksw2} and~\cite{buckrawdonstiffness}), as it provides a lower bound on the radius of curvature of a tube of unit radius. Though our theory (and our code) should work for any $\tau \geq 1$, we have not experimented with values for $\tau$ other than~$1$ and so will write the $(1,\ell)$-constraint thickness $\CThi(1,\ell,\V)$ as $\CThi(\ell,\V)$.

We can now prove that $\CThi(\ell,\V)$ is an equivalent thickness to~$\PThi$ for equilateral polygons of edgelength $\ell$.  
\begin{theorem}
  \label{thm:vb}
  If $\V$ is an equilateral polygon of edgelength~$\ell$, 
  $\PThi(\V) \geq 1 \iff \CThi(\ell,\V) \geq 1$.
\end{theorem}

To prove the theorem we will need a lemma (c.f. Lemma~13 of~\cite{MR2001a:57017}):

\begin{lemma} 
  \label{lem:dcsd vs vb}
  If $\V$ is an equilateral polygon of edgelength $\ell$ and $\MinRad
  \geq \tau$, then $\dcsd(\V) \subset \VB(\tau,\ell)$.
\end{lemma}

\begin{proof}
  The proof has two parts --- in the first, we show that the shorter arc
  between any $(p,q) \not\in \VB(\tau,\ell)$ has total curvature~$t$ less than
  $\pi$, while in the second we will show that any pair joined by such
  an arc cannot be in $\dcsd(\V)$. So suppose that $t \geq \pi$. We
  will prove that $(p,q) \in \VB(\tau,\ell)$.

  Since $\MinRad(\V) \geq \tau$, we know that each turning angle of
  $\V$ is less than $\theta(\tau,\ell)$. If the total curvature of
  the arc joining $p$ and $q$ is at least $\pi$, then $\vb(p,q) \cdot
  \theta(\tau,\ell) \geq \pi$, so
  \begin{equation}
    \vb(p,q) \geq \frac{\pi}{\theta(\tau,\ell)} 
  \end{equation}
  and $(p,q) \in \VB(\tau,\ell)$, proving the claim.

  Now suppose that $(p,q) \in \dcsd(\V)$. We claim that the total
  curvature $t$ of each arc joining $p$ and $q$ is at least $\pi$, and
  hence that $(p,q) \in \VB(\tau,\ell)$. Suppose not. The arc of $\V$
  joining $p$ and $q$ together with the chord from $p$ to $q$ form a
  closed space polygon $\mathcal{V'}$. The total curvature of this
  polygon is equal to $t$ plus the turning angles at $p$ and $q$. 
  By Fenchel's Theorem~\cite{MR0394451}, that total curvature is at least $2\pi$.  So the
     angle at $p$ and the angle at $q$ must sum to more than $\pi$.  Thus
     either the angle at $p$ or the angle at $q$ must exceed
     $\nicefrac{\pi}{2}$.  But in that case, we could reduce $d(p,q)$ to first order by
     moving $p$ or $q$ along an edge from the arc which connects $p$ and $q$, contradicting our assumption that $(p,q) \in \dcsd(\V)$.
 \end{proof}

We are now ready to prove Theorem~\ref{thm:vb}:

\begin{proof}
  Suppose that $\CThi(\ell,\V) \geq 1$. This implies that $\min_i
  \MinRad(v_i) \geq 1$ by the definition of $\CThi$. Lemma~\ref{lem:dcsd vs vb} tells
  us that $\dcsd(\V) \subset \VB(1,\ell)$, so we know that
  \begin{equation}
    \label{eq:dcsd-vs-vb}
    \min_{(p,q) \in \dcsd(\V)} d(p,q) \geq \min_{(p,q) \in
    \VB(1,\ell)} d(p,q).
  \end{equation}
  Together, these facts imply that $\PThi(\V) \geq 1$, proving one
  direction of the theorem.

  Suppose that $\PThi(\V) \geq 1$. As above, this means that $\min_i
  \MinRad(v_i) \geq 1$, so Lemma~\ref{lem:dcsd vs vb} applies and \eqref{eq:dcsd-vs-vb}
  holds. If the minimum in the right-hand side of
  \eqref{eq:dcsd-vs-vb} is achieved on the interior of $\VB(1,\ell)$,
  then it is a local minimum of $d(p,q)$ where $p \neq q$ and so is in
  $\dcsd(\V)$. In this case, \eqref{eq:dcsd-vs-vb} is an equality and
  $\CThi(\ell,\V) \geq 1$, completing the proof.

  We are left with the case where the minimum of $d(p,q)$ over
  $\VB(1,\ell)$ is realized by some $(p,q)$ on the boundary of
  $\VB(1,\ell)$. We claim that $\nicefrac{d(p,q)}{2} \geq 1$. This
  will complete the proof that $\CThi(\ell,\V) \geq 1$. 

  By definition, $(p,q)$ is on the boundary of $\VB(1,\ell)$ only if
  $\vb(p,q) = \lceil \nicefrac{\pi}{\theta(1,\ell)} \rceil$. And since
  $\vb(p,q)$ is constant on the interiors of edges, one of $p$ and $q$
  (without loss of generality, $q$) must be a vertex. Since each
  turning angle of the arc of $\V$ between $p$ and $q$ is bounded by
  $\theta(1,\ell)$, Schur's theorem~\cite{MR35:3610} implies that $d(p,q)$ is bounded
  below by the distance between the endpoints of $p'$, $q'$ of a
  planar polygonal arc $\mathcal{P}$ with the same edgelengths and
  each turning angle equal to $\theta(1,\ell)$. We depict the
  situation in Figure~\ref{fig:vb}.

  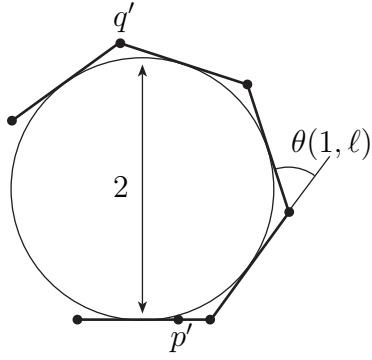
\begin{figure}
    \centering
    \begin{overpic}{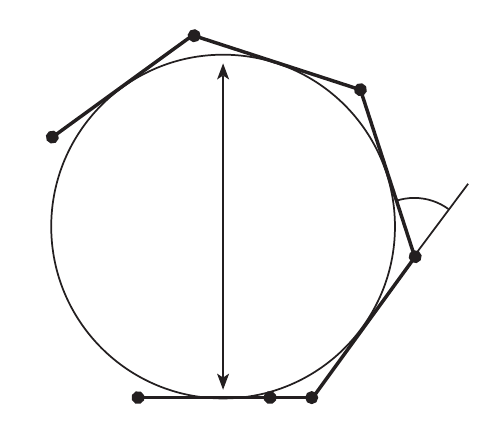}
      \put(52,1){$p'$}
      \put(37,85){$q'$}
      \put(37,40){$2$}
      \put(83.5,52){$\theta(1,\ell)$}
    \end{overpic}
    \caption{The key step in the proof of Theorem~\ref{thm:vb} is the
      proof that points $p'$ and $q'$ on an arc $\mathcal{P}$ are at
      least distance $2$ apart. This arc has equal edgelengths $\ell$,
      each turning angle equal to $\theta(1,\ell) := 2
      \arctan(\nicefrac{\ell}{2})$ and $n := \lceil \pi/\theta(1,\ell)
      \rceil$ edges. We see above that these conditions imply
      that $\mathcal{P}$ has an inscribed circle of unit radius.
      Further, the marked point $q'$ must have a larger $y$-coordinate
      than the top of the circle, providing the required lower bound
      on the distance from $p'$ to $q'$.}
    \label{fig:vb}
  \end{figure}

  We know that $\mathcal{P}$ has $n = \vb(p,q)$ edges and total
  curvature $(n-1) \theta(1,\ell)$. Since $n = \vb(p,q) = \lceil
   \nicefrac{\pi}{\theta(1,\ell)} \rceil$, we have
  \begin{equation}
    n - 1 < \frac{\pi}{\theta(1,\ell)} \leq n \quad
    \text{ so } \quad (n-1) \theta(1,\ell) < \pi \leq n\theta(1,\ell).
  \end{equation}
  Thus if we add an edge to $\mathcal{P}$ at $q'$ with turning angle
  $\theta(1,\ell)$ to form an arc $\mathcal{P^+}$, the total curvature of
  $\mathcal{P}$ is less than $\pi$ while the total curvature of
  $\mathcal{P^+}$ is at least $\pi$. These facts imply that if the
  first edge of $\mathcal{P}$ lies along the $x$-axis, the point $q'$
  has the largest $y$ coordinate on $\mathcal{P^+}$. But our turning
  angle and edgelength conditions imply that $\mathcal{P^+}$ has an
  inscribed circle of unit radius, so the $y$-coordinate of $q'$ is at
  least two. This implies that $d(p',q') \geq 2$, completing the
  proof.
\end{proof}

These proofs imply an obvious corollary which will be useful in practice:
\begin{corollary}
\label{cor:cthi vs pthi}
If $\dcsd(\V) \subset \VB(\tau,\ell)$ and the distance between any two vertices on the boundary of $\VB$ is strictly greater than $\PThi(\V)$, then $\CThi = \PThi$ for polygons in a neighborhood of $\V$ (regardless of whether or not $\V$ is equilateral with edgelength $\ell$).
\end{corollary}

\begin{proof} 
The argument is the same as that of Theorem~\ref{thm:vb}, using the hypotheses instead of Lemma~\ref{lem:dcsd vs vb} and the argument about turning angles.
\end{proof}

\subsection{Struts and Kinks}
In our definition of $\PThi$, we saw that pairs of points in $\dcsd$ and
vertices with minimum $\MinRad$ were in control of thickness. We now
want to develop similar sets of ``controlling'' pairs of points and
vertices for $\CThi$. This will require a bit of care.

Given any two line segments $e_1$ and $e_2$ in space, a calculation
reveals that the minimum distance between them is attained at a single
point unless $e_1$ and $e_2$ are parallel. In that case, the minimum
is attained at an interval of corresponding pairs (as in
Figure~\ref{fig:parallel}). The endpoints of these intervals are
self-distances measured from an endpoint of one segment to a point on
the other.  Following this line of argument we see that for any space
polygon the local minima of the self-distance function $d(p,q)$ are
isolated unless there are pairs of parallel edges, in which case there
may be families of local minima as above. Using these observations we
define
\begin{definition}
  \label{def:strutset} 
  The \emph{strut set} $\Strut(\V)$ is the set
  of pairs $(p,q)$ in $\VB(1,\ell)$ with $\nicefrac{d(p,q)}{2} = 1$ and either
  \begin{itemize}
  \item $(p,q)$ is an isolated local minimum of $d(p,q)$, or
  \item $(p,q)$ is an \emph{endpoint} of a family of local minima of $d(p,q)$.
  \end{itemize}
  In the second case, $(p,q)$ must be a vertex-edge pair joining two
  parallel edges of $\V$. 
\end{definition}
We note that $\Strut(\V)$ is a finite subset of $\dcsd(\V)$ (which may
be infinite if two edges are parallel). It is much easier to define
\begin{definition}
  \label{def:kinkset}
  The \emph{kink set} $\Kink(\V)$ is the set of vertices $v_i$ and signs
  $\pm$ with $\MinRad^{\pm} v_i = 1$.
\end{definition}
The strut and kink sets are both empty if $\CThi(\ell,\V) > 1$.

\subsection{Polygon space and variations of $\CThi$}
We now want to describe the space of variations of a polygon which
preserve or increase $\CThi$ to first order. Given a polygon $\V \in
\R^{3V}$ we can define a variation of $\V$ by any $W = (w_1, \dots,
w_{V}) \in \R^{3V}$. This variation generates a family of polygons
\begin{equation}
  \V_t = \V + t W = (v_1 + t w_1, \dots, v_{V} + t w_{V}).
\end{equation}
We now want to prove that $\CThi(\ell,\V)$ has a one-sided derivative as we
vary $\V$ according to any variation $W$ and to give a finite
procedure for computing that variation. This will require some setup.

\begin{proposition}
  \label{prop:fv}
  Suppose that $\CThi(\ell,\V) = 1$. Then viewing every pair of points
  $(p,q)$ on $\V$ and every $\MinRad^{\pm} v_i$ as functions of $t$, the
  forward time derivative below exists and satisfies
  \begin{multline}
    \label{eq:fv}
    D_W \CThi(\ell,\V) = \left. \frac{\mathrm{d}}{\mathrm{d}t^+}
    \CThi(\V_t) \right|_{t=0} \\ = \min \left\{ 
    \min_{(v_i, \pm) \in \Kink}
    \left. \frac{\mathrm{d}}{\mathrm{d}t^+} (\MinRad^{\pm} v_i)(t)
    \right|_{t=0},
    \min_{\Strut(\V)}
    \left. \frac{\mathrm{d}}{\mathrm{d}t^+}
    \frac{d(p(t),q(t))}{2}\right|_{t=0} \right\}.
  \end{multline}
\end{proposition}

\begin{proof} 
  We begin by ignoring any $\MinRad v_i$ functions which are not
  defined (which happens when $v_{i-1}$, $v_i$ and $v_{i+1}$ are colinear). Since $\CThi(\ell,\V)$ is equal to $1$, the $\MinRad$ of these
  vertices will not affect $\CThi(\V + tW)$ for small enough $t$.  The
  function $\CThi$ is then the minimum of a set of differentiable
  functions $\MinRad^{\pm} v_i$ and $\nicefrac{d(p,q)}{2}$ indexed by
  the (compact) disjoint union of compact sets $\{v_1,
  \pm\} \sqcup \dots \sqcup \{v_{V},\pm\} \sqcup \VB(1,\ell) $ (where we assume that any
  $v_i$ with $\MinRad v_i$ undefined are missing). Clark's theorem for
  min-functions~\cite{MR51:3373} tells us immediately that the
  derivative in \eqref{eq:fv} exists.

  However, Clark's theorem tells us that 
  \begin{equation*}
    D_W \CThi(\ell,\V) = \min \left\{ 
    \min_{\substack{(v_i,\pm) \\
    \MinRad^\pm v_i = 1}} \!\!\!\!\!\! \left. \frac{\mathrm{d}}{\mathrm{d}t^+} \right|_{t=0}
    \!\!\!\! (\MinRad^\pm v_i)(t),\!
    \min_{\substack{(p,q) \in \VB(1,\ell) \\
    \nicefrac{d(p,q)}{2} = 1}} \left. \frac{\mathrm{d}}{\mathrm{d}t^+} \right|_{t=0}
    \!\!\!\! \frac{d(p(t),q(t))}{2} \right\}.
  \end{equation*}
  The first set $\{ (i,\pm)\, | \,\MinRad^\pm v_i = 1\}$ is the kink
  set, which matches~\eqref{eq:fv}. But if a pair of edges in $\V$ are
  parallel and at distance $2$ from one another, then $\Strut(\V)$ is
  only a subset of $\{(p,q) \in \VB(1,\ell) \,|\, \nicefrac{d(p,q)}{2} =
  1\}$. We must prove that
  \begin{equation}
    \label{eq:toprove}
    \min_{\substack{(p,q) \in \VB(1,\ell) \\
    \nicefrac{d(p,q)}{2} = 1}} 
    \left. \frac{\mathrm{d}}{\mathrm{d}t^+}
    \frac{d(p(t),q(t))}{2}\right|_{t=0} = \min_{(p,q) \in \Strut(\V)}
    \left. \frac{\mathrm{d}}{\mathrm{d}t^+}
    \frac{d(p(t),q(t))}{2}\right|_{t=0}\!\!\!\!\!\!\!\!.
  \end{equation}
  For any pair of parallel edges with distance 2, we may assume that
  the situation is as in Figure~\ref{fig:parallel}.
  \begin{figure}
    \centering
    \begin{overpic}{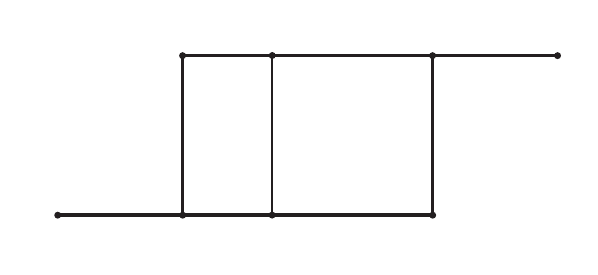}
      \put(5,4){$v_{i-1}$}
      \put(65,4){$v_i = q$}
      \put(28,4){$p$}
      \put(17,40.5){$v_{j-1} = r$}
      \put(72,40.5){$s$}
      \put(92,40.5){$v_j$}
    \end{overpic}
    \caption{When the edges $e_i$ and $e_j$ are parallel, many chords
      realize the minimum distance between the segments. In this case,
      we show that the minimum derivative of distance between any of
      these pairs occurs at one end or the other.  We name the
      endpoints of this family of chords $p$ and $q$ on $e_i$ and $r$
      and $s$ on $e_j$. One of each of these pairs must be an
      endpoint --- in this case it is $q = v_i$ and $r = v_{j-1}$ that
      are endpoints. }
    \label{fig:parallel}
  \end{figure}
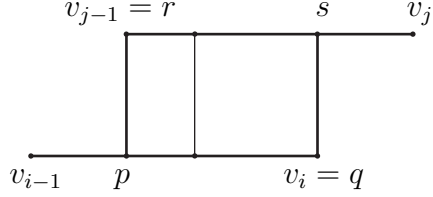
  We label points $p$, $q$, $r$ and $s$ as in the Figure, and
  parametrize the line segments between $p$ and $q$ and between $r$
  and $s$ by $\eta \in [0,1]$. The pairs with $\eta = 0$ and $\eta =
  1$ are in the strut set of $\V$, but the pairs given by all other
  values of $\eta$ are not. To prove~\eqref{eq:toprove} we must find
  \begin{equation*}
    \min_{\eta \in [0,1]} \frac{\mathrm{d}}{\mathrm{d}t^+} \, \left\Vert 
    \eta p + (1-\eta) q - \eta r - (1-\eta) s \right\Vert,
  \end{equation*}
  and show that it is attained at $\eta = 0$ or $\eta = 1$. If we view
  $p$, $q$, $r$, and $s$ as functions of time, then for any given
  $\eta$, the time derivative of the corresponding length is given by
  \begin{equation*}
    \frac{1}{2} \langle \eta p + (1-\eta) q - \eta r - (1-\eta) s,
    \eta p' + (1-\eta) q' - \eta r' - (1-\eta) s' \rangle,
  \end{equation*}
  where we have used the fact that $\nicefrac{d(e_i,e_j)}{2} = 1$.
  Regrouping, we can rewrite this as
  \begin{equation*}
    \frac{1}{2} \langle \eta (p - r) + (1-\eta) (q-s), \eta
    (p-r)' + (1-\eta) (q-s)' \rangle,
  \end{equation*}
  and using the fact that $p-r = q-s$ at time $0$, we can again
  rewrite this as
  \begin{equation*}
    \eta \langle p - r, p' - r' \rangle +
    (1-\eta) \langle q - s, q' - s' \rangle.
  \end{equation*}
  Now as $\eta$ varies between $0$ and $1$, we note that the
  $\eta$ derivative of the above quantity is
  \begin{equation*}
    \langle p - r, p' - r' \rangle - 
    \langle q - s, q' - s'\rangle.
  \end{equation*}
  In particular, this derivative is nonzero for all $\eta \in [0,1]$
  unless $\langle p - r, p' - r' \rangle = \langle q - s, q' - s'
  \rangle$, in which case it vanishes identically. This means that the
  minimum value of this expression is always realized when $\eta =
  0$ or $\eta = 1$. This completes the proof.
\end{proof}

We can use Proposition~\ref{prop:fv} to define two sets of variations
that will be of particular interest to us. The first set consists of
variations that are tangent to the boundary or pointing into the
interior of the set of polygons $\CThi(\ell,\V) \geq 1$. We will allow our
polygons to move in these directions.

\begin{definition}
  Suppose we have a polygon $\V$ and a variation $W$ of $\V$.  If
  $\CThi(\ell,\V) = 1$, we say $W$ is an \emph{infinitesimal motion}
  of $\V$ if the forward directional derivative
  \begin{equation}
    D_{W} \CThi(\ell,\V) \geq 0.
  \end{equation}
  If $\CThi(\ell,\V) > 1$, we call every variation $W$ an infinitesimal motion. The set of all infinitesimal motions of $\V$ is denoted
  $I(\V)$.
\end{definition}

The following Corollary follows directly from Proposition~\ref{prop:fv}.
\begin{corollary} 
  The set $I(\V)$ is the dual cone of the set $-\grad
  \nicefrac{d(p,q)}{2}$ for $(p,q) \in \Strut(\V)$ and $-\grad
  \MinRad^\pm v_i$\ \,for $(v_i,\pm) \in \Kink(\V)$.
\end{corollary}

\begin{proof}
  We need only recall that the dual cone $A^+$ to a set of vectors $A$
  is the set of vectors $X$ for which $\left< X, W \right> \leq 0$ for
  all $W \in A$. Since the directional derivatives of
  $\nicefrac{d(p,q)}{2}$ and $\MinRad^\pm v_i$ in the direction $X$
  are the dot products of $X$ with $-\nicefrac{\grad d(p,q)}{2}$ and
  $-\grad \MinRad^\pm v_i$, $X$ is in the dual cone if and only if all
  these directional derivatives are nonnegative. But by the
  Proposition, this implies that $D_X \CThi(\ell,\V)$ is nonnegative
  as well.
\end{proof}

The second set of variations of interest will be the normal cone of
the boundary of the set of polygons with $\CThi(\ell,\V) \geq 1$. We will
forbid our polygons from moving in these directions.
\begin{definition}
  The convex cone of \emph{resolvable motions} $R(\V)$ of $\V$ is the
  cone generated by the set $-\grad \nicefrac{d(p,q)}{2}$ for $(p,q) \in
  \Strut(\V)$ and $-\grad \MinRad^\pm v_i$ for $(v_i,\pm) \in \Kink(\V)$.
  $R(\V)$ is the set of vectors $R \in \R^{3V}$ which can be expressed
  in the form
  \begin{equation}
    \label{eqn:rdef}
    R = \sum_{(p,q) \in \Strut(\V)} -\lambda^2_i \grad \frac{d(p,q)}{2}
    + \sum_{v_j \in \Kink(\V)} -\lambda^2_i \grad \MinRad v_j.
  \end{equation}
  Here the indices $i$ and $j$ just number the elements of the strut
  and kink sets. The constants $\lambda^2_i$ and $\lambda^2_j$ are
  nonnegative numbers, as suggested by the notation.
\end{definition}
It is a standard fact from optimization theory that $R(\V) = I(\V)^+$, since for any set of vectors $\{v\}$ the double dual $\{v\}^{++}$ is the cone generated by $\{v\}$.

\subsection{Theory of constrained optimization}
Given a function $f(\V)$ on the space of polygons $\R^{3V}$, we can
compute the negative gradient $-\grad f$, which is a variation vector
in $\R^{3V}$. We are now interested in understanding how this gradient
is modified by the constraint $\CThi(\ell,\V) \geq 1$. This thickness
constraint models the effect of an embedded tube around the polygon:
it allows some motions of $\V$ and blocks others.

\begin{definition}
  \label{def:cg}
  The \emph{constrained gradient} $(-\grad f)_I$ of $-f$ is the
  closest vector in $I(\V)$ to $-\grad f(\V)$.
\end{definition}

We now recall that any convex cone and its dual cone provide a kind
of orthogonal decomposition of their ambient vector space, as shown in 
Figure~\ref{fig:decomp}.

\begin{figure}
  \begin{center}
    \begin{overpic}[scale=0.9]{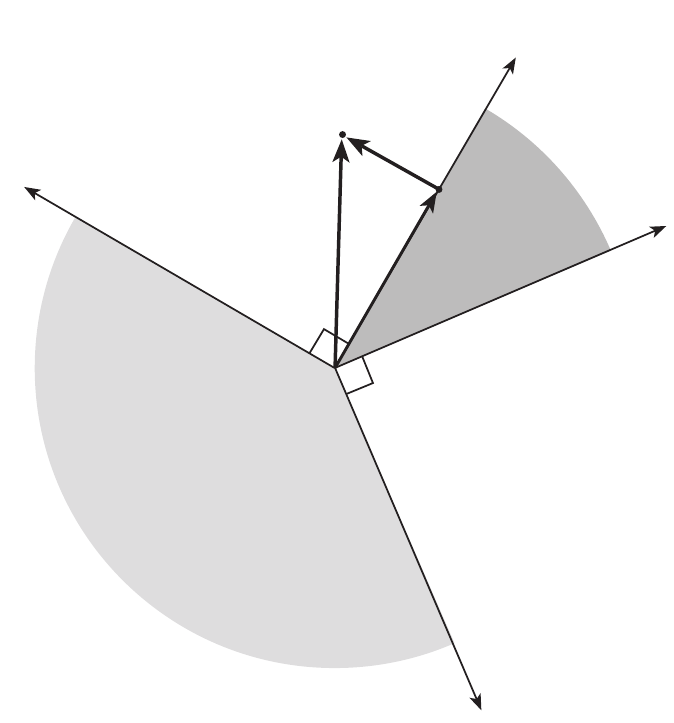}
      \put(65,70){$R(\V)$}
      \put(28,37){$I(\V)$}
      \put(45,83){$W$}
      \put(57,60){$W_R$}
      \put(54,79){$W_I$}
     \end{overpic}
  \end{center}
  \caption{The infinitesimal motions $I(\V)$ and the resolvable
    motions $R(\V)$ of $\V$ form dual convex cones.  Hence, although
    these are not orthogonal subspaces of~$\R^{3V}$, a similar
    decomposition property holds true: any vector $W$ may be written
    uniquely as a sum of a vector $W_I \in I(\V)$ and a vector $W_R
    \in R(\V)$.}
  \label{fig:decomp}
\end{figure}
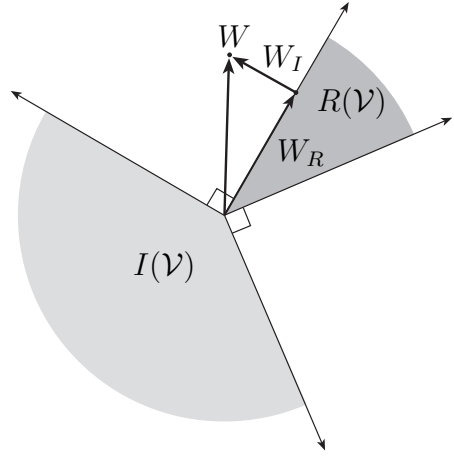

\begin{proposition}[\cite{MR0286498}, Thm. $2.8.7$]
  \label{prop:decomposition}
  Any vector $W \in R^{3V}$ may be uniquely written
  \begin{equation}
    W = W_R + W_I, \quad \text{ where } \langle
    W_R, W_I\, \rangle = 0, 
  \end{equation} 
  $W_R \in R(\V)$ is the closest resolvable motion to $W$, and $W_I
   \in I(\V)$ is the closest infinitesimal motion to $W$.
\end{proposition}

We note that this Proposition shows that the constrained gradient of
$-f$ is well-defined. Further, it is easy to show that the constrained
gradient is the direction of steepest descent for $f$ within
$I(\V)$. This makes us guess that the constrained gradient should
vanish at a critical point for minimizing~$f$. To prove it, we define
critical points more carefully

\begin{definition}
  \label{def:critical}
  We say that $\V$ is \emph{thickness-critical for minimizing
    $f$} if either:
    \begin{itemize}
    \item $D_W f = 0$, or 
    \item  $\CThi(\ell,\V) = 1$ and for any~$W$ 
    with $D_{W} f(\V) < 0$, we have $D_{W} \CThi(\ell,\V) < 0$.
    \end{itemize}
\end{definition}

In the first case,  we are at an unconstrained critical point of the objective function $f$. In the second, we are at a constrained critical point where motion in the direction of the negative gradient of $f$ is blocked by active constraints. We then have a version of the Kuhn-Tucker theorem (restated in our
language from the original form in~\cite{MR2284052}), which gives a
verifiable condition for thickness-criticality. 

 \begin{theorem}
   \label{thm:constrainedcritical}
   The polygon $\V$ is thickness-critical for minimizing $f$ $\iff$
   $-\grad f$ is in $R(\V)$ $\iff$ the constrained gradient $(-\grad
   f)_I$ vanishes.
 \end{theorem}
 
 \begin{proof}
   It suffices to show that the first two statements are equivalent,
   since the second and third are clearly equivalent by
   Proposition~\ref{prop:decomposition}.  

   If $-\grad f$ is not in $R(\V)$, then Farkas' theorem implies that
   there exists some $W$ with $\langle W, \grad f \rangle = D_W f < 0$
   and $\langle W, R \rangle \leq 0$ for all $R \in R(\V)$
   (\cite{panik}, p.\,118). Using the definition of $R(\V)$ and
   Proposition~\ref{prop:fv}, this implies $D_W \CThi(\ell,\V) \geq
   0$. Thus $\V$ is not thickness-critical for minimizing~$f$.

   If $-\grad f$ is in $R(\V)$ we will prove that $\V$ is
   thickness-critical for minimizing $f$. We first observe that the
   dual cone of $-\grad f$ contains the dual cone $R^+(\V)$. Now
   suppose we have some $W$ with $D_{W} f < 0$. Then $\langle W,
   -\grad f \rangle > 0$, so $W \not\in (-\grad f)^+$ and in
   particular $W \not\in R^+(\V)$. But this means that $\langle W, R
   \rangle > 0$ for some $R \in R(\V)$, so $D_W \CThi(\ell,\V) < 0$. Hence
   $\V$ is thickness-critical for minimizing~$f$.
 \end{proof}
   
 We can give a natural interpretation of this Theorem in mathematical
 and physical terms by considering the condition $-\grad f \in
 R(\V)$. By definition, this means that 
 \begin{equation}
   \label{eqn:kt}
   -\grad f + \sum_{(p,q) \in \Strut(\V)} \lambda^2_i \,\grad
    \frac{d(p,q)}{2} + \sum_{v_j \in \Kink(\V)} \lambda^2_j \, \grad \MinRad
    v_i = 0.
 \end{equation}
 Mathematically, the $\lambda^2_i$ and $\lambda^2_j$ are Lagrange
 multipliers. If we think of the thickness constraint as an embedded
 tube around $\V$, we can interpret these scalars as magnitudes of
 compression forces transmitted by tube contacts (for struts) and
 angles where the polygon resists further bending (for kinks).
 
 In general, we cannot expect every local minimum of a constrained
 function to be a constrained critical point in the sense of
 Definition~\ref{def:critical}. If the set of polygons defined by
 $\CThi(\ell,\V)$ had an outward-pointing cusp we might reach a point where
 some $W$ with $D_{W} f < 0$ had $D_{W} \CThi = 0$. For example, the
 constrained system 
 \begin{equation*}
\text{ minimize } f(x,y) = -x, \text{ subject to } g(x,y) = \min
 \{x^3 - y, y \} \geq 0
 \end{equation*}
 has this property at the local minimum
 $(0,0)$ for $W = (1,0)$.  The problem here is simply that $D_{W} g \leq
 0$ for all $W$. This does not happen for thickness-constrained
 polygons, but we will need another idea to prove it:
 \begin{definition}
   We say that $\V$ is \emph{constraint-qualified} (in the
   sense of Mangasarian and Fromovitz~\cite{MR0207448}) if there exists some $W$
   so that $D_{W} \CThi > 0$.
 \end{definition}
 It is then standard to show
 \begin{proposition}[\cite{MR2284052}]
   Any constraint-qualified local minimum of $f$ is a
   thickness-critical point for minimizing $f$.
 \end{proposition}
 In our case, scaling $\V$ provides the desired motion, so we
 have
 \begin{corollary}
   \label{cor:maxes-are-critical}
   If the polygon $\V$ is a local minimum for $f$, then it is
   a thickness-critical point for minimizing $f$.
 \end{corollary}
 We make a final note that in general, our criticality theory works equally well for $\CThi$ and $\PThi$ (even for polygons $\V$ which are not equilateral), as long as they obey the hypotheses of Corollary~\ref{cor:cthi vs pthi}. This is true in practice in all of our numerically computed configurations.

\section{Bridging Theory and Computation}
\label{sec:bridge}

\subsection{Overview of the algorithm}
We have now derived enough theory to describe our algorithm in general
terms. We wish to minimize the function $\Len(\V)$ subject to the
constraint $\CThi(\ell,\V) \geq 1$. We will do so by computing the
constrained gradient $(-\Len \V)_I$ and stepping in this
direction. These steps will reduce $\Len(\V)$ while keeping $\V$ close
to the set $\CThi(\ell,\V) \geq 1$ (since the constraints are nonconvex, we cannot stay
entirely inside this set). When $(-\Len \V)_I$ vanishes, the algorithm
will terminate. By Theorem~\ref{thm:constrainedcritical} if the constrained gradient was exactly zero, the resulting configuration would be a thickness-critical point for minimizing length. We note that our algorithm will attempt to maintain an approximately equilateral polygon $\V$ but it is not required to: constant edgelength $\ell$ is not a hypothesis of Theorem~\ref{thm:constrainedcritical}. Our only caveat is that we must remember that $\CThi(\V)$ may not be equal to $\PThi(\V)$ if the final configuration fails to obey the hypotheses of Corollary~\ref{cor:cthi vs pthi}. We also note that there is nothing special about choosing $\Len(\V)$ as the function to minimize --- both our theory and our code would work just as well for any other function. 

\subsection{Computing the constrained gradient}
\label{sec:computingi}
 To implement this algorithm, we must be able to compute the
 constrained gradient $(-\grad f)_I$. This is a standard problem in
 linear algebra. By definition, if $-\grad f$ is written as $(-\grad
 f)_R + (-\grad f)_I$ using Proposition~\ref{prop:decomposition}, the
 constrained gradient is equal to $(-\grad f)_I$. We can compute that
 by computing $(-\grad f)_R$, which is easy to do since we know the
 generators of the cone $R(\V)$.
 \begin{definition}
   If $\CThi(\ell,\V) = 1$, the \emph{rigidity matrix} $A$ of $\V$ is
   the matrix whose columns are the gradients $-\grad
   \nicefrac{d(p,q)}{2}$ for $(p,q) \in \Strut(\V)$ and $-\grad
   \MinRad^\pm v_i$ for $(v_i,\pm) \in \Kink(\V)$.
 \end{definition}
 We can construct the rigidity matrix by finding the members of
 $\Strut(\V)$ and $\Kink(\V)$. It follows from the definition that
 $R(\V)$ is the image of the positive orthant under the matrix $A$.
 By Proposition~\ref{prop:decomposition}, $(-\grad f)_R$ is the
 closest vector in that image to $-\grad f$. So if we solve the
 non-negative least-squares (NNLS) problem
 \begin{equation}
   \label{eq:nnls-problem}
   \min_{\Lambda \geq 0} \norm{A \Lambda - (-\grad{f})},
 \end{equation}
 then $(-\grad f)_R = A \Lambda$ and $(-\grad f)_I = -\grad f - A
 \Lambda$.  This least-squares problem is a special kind of quadratic
 programming problem which has been well-studied in numerical linear
 algebra (see~\cite{MR1386889}). In our case, the problem is much
 easier because~$A$ is extremely sparse --- the gradients of the
 $\nicefrac{d(p,q)}{2}$ involve no more than $4$ vertices (and so $12$
 variables), while the gradients of the $\MinRad^\pm$ involve only $3$
 vertices (and $9$ variables). So each column of $A$, which is
 typically $1000$ or more entries long, contains at most $12$ nonzero
 entries.

\subsection{The gradient of Length}
We can now compute $(-\grad \Len)_I$ if we can compute $-\grad \Len$,
build the rigidity matrix $A$ from the strut and kink sets, and solve
the NNLS problem in~\eqref{eq:nnls-problem}. We will take these
problems in order.

Length is a differentiable function of polygons $\V \in \R^{3V}$,
whose gradient is given by a straightforward calculation:
\begin{proposition}
  \label{prop:gradlen}
  The gradient of length of a polygon $\V_n$ is given by the
  collection of $n$ vectors
  \begin{equation}
    \nabla \Len(\V)_k = \frac{v_{k-1} - v_k}{\norm{v_{k-1} -
	v_k}} + \frac{v_{k+1} - v_k}{\norm{v_{k+1} - v_k}}.
  \end{equation}
\end{proposition}

\subsection{The gradient of $\nicefrac{d(p,q)}{2}$}
Given a pair of points $(p,q)$ on $\V$, the gradient of the distance
between them is a set of four vectors located at the endpoints of the
edges on which $p$ and $q$ lie. These vectors are given by a
calculation:

\begin{proposition}
  \label{prop:dpqgrad}
  Suppose that $(p,q) \in \Strut(\V)$. If $p = \alpha v_i + (1-\alpha)
  v_{i+1}$ and $q = \beta v_j + (1-\beta) v_{j+1}$ then
  \begin{equation*}
    \grad \frac{d(p,q)}{2} = \frac{1}{2 d(p,q)} \, \left\{ \alpha (p-q), (1 - \alpha)
    (p-q), \beta (q-p), (1-\beta)(q-p) \right \}.
  \end{equation*}
  where these three vectors are applied to $v_i$, $v_{i+1}$, $v_j$ and
  $v_{j+1}$ in order.
\end{proposition}

\subsection{The gradient of $\MinRad^\pm$}
As we noted above, the $\MinRad^\pm$ are differentiable where they are
defined. We now compute the gradient on $\MinRad^+$, noting that the
gradient of $\MinRad^-$ is similar.

\begin{proposition}
  \label{prop:mrgrad}
  Given a vertex $i$ on $\V_n$ with finite $\MinRad^\pm(v_i)$,
  we let $n$ denote the oriented normal vector to the plane defined by
  $v_{i-1}, v_i, v_{i+1}$ and define the scalar constant
  \begin{equation*}
    K =  \frac{\norm{v_{i+1} - v_i}}{2 \cos \theta - 2}
  \end{equation*}
  and the vector constants
  \begin{equation*}
    V = \frac{v_{i+1} - v_i}{2\tan(\nicefrac{\theta}{2}) \norm{v_{i+1}
        - v_i}} , \quad 
    W = K \frac{(v_{i-1} - v_i) \cross n}{\norm{v_{i-1} - v_i}^2}, \quad
    X = K \frac{n \cross (v_{i+1} - v_i)}{\norm{v_{i+1} - v_i}^2}. 
  \end{equation*}
  Then if we write the gradient of $\MinRad^+$ as a triple of vectors
  located at $v_{i-1}$, $v_i$, and $v_{i+1}$ we have
  \begin{equation*}
    \grad \MinRad^+(v_i) = \{ W, -W - X - V, X + V \}.
  \end{equation*}
\end{proposition}

\begin{proof}
  The proof is a lengthy calculation. We want to compute the gradient
  of $\MinRad^+(v_i) =
  \frac{\norm{v_{i+1}-v_i}}{2\tan(\nicefrac{\theta}{2})}$, where
  $\theta$ is the turning angle at vertex $v_i$. We start with a
  change of variables. Let $A = v_{i-1} - v_i$ and $B = v_{i+1} -
  v_i$. We can rewrite $\MinRad^+$ in terms of these variables and
  compute its gradient as follows:
  \begin{equation}
    \label{eq:ABgrad}
    \grad \frac{\norm{B}}{2\,\tan(\nicefrac{\theta}{2})} =
    \frac{1}{2\tan(\nicefrac{\theta}{2})} \left( 0, \frac{B}{\norm{B}}
    \right) - \frac{1}{2} \left[ \frac{\norm{B}}{\tan^2
    (\nicefrac{\theta}{2})} \cdot \frac{d}{d\theta}
    \tan(\nicefrac{\theta}{2}) \right] \grad \theta.
  \end{equation}
  Now
  \begin{equation}
    \label{eq:tan-identity}
    \frac{d}{d\theta} \tan(\nicefrac{\theta}{2}) = \frac{1}{2 \cos^2
    (\nicefrac{\theta}{2})} =
    \frac{1}{2 \frac{1 + \cos \theta}{2}} = \frac{1}{1 + \cos \theta},
    \quad \text{ and } \quad \tan^2 (\nicefrac{\theta}{2}) = \frac{1 -
    \cos \theta}{1 + \cos \theta}. 
  \end{equation} 
  So we can rewrite~\eqref{eq:ABgrad} as
  \begin{equation*}
    \grad \frac{\norm{B}}{2\,\tan(\nicefrac{\theta}{2})} =
    \frac{1}{2\tan(\nicefrac{\theta}{2})} \left( 0, \frac{B}{\norm{B}}
    \right) -
    \frac{\norm{B}}{2 - 2 \cos \theta} \, \grad \theta = (0,V) + K
    \grad \theta.
  \end{equation*}
  Keeping track of the sign of the exterior angle, we see that if $n$
  is the oriented unit normal to the plane containing $A$ and $B$, we
  have
  \begin{equation*}
    \grad \theta = \left(\frac{A \cross n}{\norm{A}^2}, \frac{n \cross
    B}{\norm{B}^2} \right) \quad \text{ so } \quad \grad
    \frac{\norm{B}}{2\,\tan(\nicefrac{\theta}{2})} = (W, X + V).
  \end{equation*} 
  Using the definition of $A$ and $B$ to change back to the original
  variables completes the proof.
\end{proof}
The function $\MinRad(v_i)$ provides a discrete analog to the
radius of curvature for the polygonal curve~$\V$ at~$v_i$. Since this
is a numerical computation of a second derivative, we expect the
function to be quite sensitive to small changes in the positions of
the vertices of $\V$. This sensitivity will limit the accuracy of our
computations, so we record an estimate of the norm of the gradient of
$\MinRad^+(v_i)$.

\begin{corollary}
  \label{cor:normgrad}
  If $\V$ is an equilateral polygon with edgelength $\ell$ and
  $\MinRad v_i = 1$ then
  \begin{equation*}
    \norm {\grad \MinRad^{\pm} v_i} \geq \frac{2}{\ell^2}.
  \end{equation*}
\end{corollary}

\begin{proof}
  Consider 
  \begin{equation*}
    \norm{W} = \frac{\norm{v_{i+1} - v_i}}{\abs{2 \cos \theta - 2}}
    \frac{\norm{(v_{i-1} - v_i) \cross n}}{\norm{v_{i-1} - v_i}^2}.
  \end{equation*}
  Since the polygon is equilateral, and $n$ is a unit vector normal to
  $v_{i-1} - v_i$, this is just $\norm{W} = \nicefrac{1}{\abs{2 \cos
  \theta - 2}}$. If $\MinRad = 1$, then (squaring $\MinRad$ and using
  both half-angle formulae for tangent) we see that $\norm{W} =
  \nicefrac{\abs{2 + 2 \cos \theta}}{\ell^2}$.  Since $W$ appears
  alone in the formula for $\grad \MinRad^+$, this is a lower bound
  for the norm of the entire gradient.
\end{proof}

\section{Program design}
\label{sec:programdesign}

\subsection{Issues of scale}
The design and implementation of our algorithm \ridgerunner\ were
shaped by the scale of the knot-minimizing problems we intended to
solve and the amount of computer power we had on hand to solve them.
To inform the discussion that follows, we will now take a moment to
consider the dimensions of our problems. In a typical run, we started
by minimizing the length of a low-resolution version of our knot or
link with $2$ vertices per unit of ropelength ($80$ to $150$ vertices). 
Once that configuration was
minimized, a medium resolution run at $4$ vertices per unit of ropelength was
performed. A final run followed at $8$ vertices per unit ropelength. Most of
the runtime was spent during the final run, which took $20-40$ CPU hours on a desktop computer. During the final run, the
average edgelength $\ell$ for our curves was approximately \AverageEdgelength, which meant that there were \AverageEdges\ edges. The average size of the strut set
was \AverageStruts\ pairs of points, while the average size of the kink set was \AverageMrlocs\ vertices. The rigidity matrix was then on average a $\MatrixDims$ matrix which was \MatrixSparsity\ sparse (no more than \MatrixNNZ\ of its \MatrixEnts\ entries were nonzero). A typical run contained several hundred thousand steps.

\subsection{The algorithm} Our method is based loosely on the method of
constrained gradient descent. The basic idea is to generate a series
of polygons $\V_i$ which converge to a limit polygon which is
thickness-critical for minimizing a function $f(\V)$ by taking a
series of steps in the form
\begin{equation}
  \V_{k+1} = \V_k + \alpha (-\grad f)_I, \quad \text{ where $\alpha$
    is chosen by a search algorithm.}
\end{equation}
When $\CThi(\ell,\V) > 1$, this is just the method of steepest
descent, since $(-\grad f)_I = -\grad f$. When $\CThi(\ell,\V) = 1$,
these steps are tangent to the boundary of $\CThi(\ell,\V) \geq 1$ and
in principle decrease $\CThi$ by no more than $O(\alpha^2)$. In some
circumstances, such as when two sections of tube touch for the first
time, we can decrease $\CThi$ by $O(\alpha)$ (which is much larger,
since $\alpha << 1$). We control this error by searching for an
$\alpha$ which keeps $\CThi(\ell,\V_k + \alpha (-\grad f)_I)$ within
acceptable bounds.  When $\CThi(\ell,\V_k)$ becomes too small, we
correct the accumulated error using a Newton's method-type solver. The
code terminates when we the constrained gradient is small enough to
convince us that we are near a point which is thickness-critical for
minimizing $f$. This procedure is summarized in Algorithm~\ref{alg:outline}.

\LinesNumbered

\SetKwData{maxstep}{MaxStep}
\SetKwData{maxerr}{MaxErr}
\SetKwData{steperr}{StepErr}
\SetKwData{err}{Err}

\begin{algorithm}
  
  \SetKwInOut{Input}{input} 
  \SetKwInOut{Output}{output} 
  
  \BlankLine
  
  \Indp 
  
  \Input{A polygon $\V_0$ and an error bound error bound $\maxerr$.}

  \Output{A sequence of positions $\V_{k}$ with
     $\CThi(\ell,\V_{k}) \geq 1 - \maxerr$.}
  
  \BlankLine
  
  \Repeat{$\nicefrac{\norm{(-\grad f)_I}}{\norm{-\grad f}}$ is
    sufficiently small }{ 
    \Indp 
    Compute $-\grad f = -\grad \Len(\V_{k}) + - \grad \Eq(\V_{k})$\; 
    Find $\Strut(\V)$ and $\Kink(\V)$ and construct the rigidity matrix $A$\; 
    Compute constrained gradient $(-\grad f)_I$.\;
    
    Search for $\alpha$ so that $\V_k + \alpha(-\grad f)_I$ minimizes ropelength 
    and	is computationally acceptable and set $\V_{k+1} = \V_k + \alpha (-\grad f)_I$\;

    \If{$\CThi(\ell,\V_{k+1}) < 1 - \maxerr$}{       
      \Indp
      
      Correct $\CThi(\ell,\V_{k+1})$ by Newton's method\; }
    
  }
  \caption{The outline of the \ridgerunner\ algorithm.}
  \label{alg:outline}
\end{algorithm}
In the rest of this section, we will comment on each of these steps in
turn. 

\subsection{Step 2. Equilateral polygons, $\CThi$ and $\PThi$}
We have only proved that $\CThi \geq 1 \iff \PThi \geq 1$ for
equilateral polygons. It is therefore important that our $\V_k$ remain
at least approximately equilateral during a run. We enforce this
constraint by defining a penalty function $\Eq(\V)$ which is minimized
when $\V_k$ is equilateral and minimizing the sum $\Len(\V) +
\Eq(\V)$. This is quite effective (a typical run recorded an average
error in edgelength of about $0.385\%$) in practice. We note that
while $\CThi$ and $\PThi$ might not be equal for nonequilateral
polygons, we avoid any problems that might result by performing all of
our final ropelength calculations with respect to the original $\PThi$
thickness.

\subsection{Step 3. Finding $\Strut(\V)$ and $\Kink(\V)$}
In principle, the strut and kink sets could be found by direct
inspection of all pairs of edges and all vertices of $\V$. But since
there are usually $10^6$ such pairs, this naive method consumes too
much runtime. So to find the strut and kink sets, we used the
clustering code \texttt{octrope} of Ashton and Cantarella described
in~\cite{MR2197947}. This was fast enough that over $30$ seconds of a
typical\footnote{a $400$ edge $5.1$ knot with about $600$ struts}
run about $10\%$ of runtime was spent finding $\Strut(\V)$ and
$\Kink(\V)$.  The algorithm in \octrope\ does not take advantage of
the fact that it is called successively on data which vary little
between calls, so a much faster customized strut-finding code could be
written into \ridgerunner. However, these figures show that this
project would have little impact on overall performance.

\subsection{Step 4. Finding the constrained gradient}
\label{sec:tsnnls}
Once we have $\Strut(\V)$ and $\Kink(\V)$ we can use the gradient
formulae given in Propositions~\ref{prop:dpqgrad} and~\ref{prop:mrgrad}
to construct the rigidity matrix $A$. We must then solve the sparse non-negative least squares (SNNLS) problem $\min_{\Lambda \geq 0} \norm{A \Lambda - (-\grad{f})}$, which we recall as Equation~\vref{eq:nnls-problem}.
  
We use the freely available \tsnnls\ library of Cantarella,
Piatek, and Rawdon~\cite{tsnnls}, which is an implementation of the block-pivoting
algorithm of Portugal, Judice and Vicente~\cite{MR95a:90059}.  The PJV
algorithm solves a sequence of unconstrained least-squares problems to
find a partition of the variables of~$\Lambda$ into complementary sets
$F$ and $G$ representing variables which will be nonzero and zero in
the solution to~\eqref{eq:nnls-problem}. It is very important to take
advantage of the sparsity of $A$ in order to solve these (rather
large) problems in an acceptable amount of time, as this step makes
the dominant contribution to our overall runtime in most cases. To
this end, \tsnnls\ solves the least-squares problem $Ax = b$ by
solving the ``normal equations'' $A^T A x = A^T b$. Since $A^T A$ is
symmetric, we can solve this system using a Cholesky
factorization. This is done very quickly using the multifrontal
supernodal sparse Cholesky code \texttt{TAUCS} of Toledo
et\,al.\,~\cite{taucs}.

We have sacrificed some accuracy in favor of speed, since the
condition number of $A^T A$ is the square of the condition number
of~$A$.  A standard ``rule of thumb'' in such situations is that the
error in the solution is on the order of machine epsilon ($10^{-16}$)
multiplied by condition number. To verify that this was small in
practice, we used the \texttt{rcond} function in \texttt{LAPACK} to estimate the
condition number of the rigidity matrices of all of our final
configurations. The average condition number was on the order of
$10^{4}$ with none being worse than $8 \times 10^{5}$. Thus we expect to have
an average error on the order of $10^{-8}$ and a worst-case error of $10^{-6}$ 
in our final computations of the constrained gradient.

It is also worth noting that the \texttt{TAUCS} code will fail if the rigidity matrix is singular, which will occur when there is more than one way to balance gradient force. This is expected for very complicated knots, but seems to be rare among knots in our dataset. A more advanced version of \tsnnls\ would calculate a minimum-norm solution to the least-squares problem in this case.

\subsection{Step 5. Choosing a stepsize} 
When $\CThi(\V) > 1$ our code sets a small maximum stepsize of
$10^{-2}$ and proceeds by Euler integration\footnote{We could improve
  the accuracy and speed of this portion of the computation by using a
  smarter ODE solving method. But these steps have no linear algebra
  involved, so they are already orders of magnitude faster than the ones
  to come. In practice, this portion of the run consumes $<1\%$ of the
  total runtime.}. Once $\CThi(\V) = 1$, thickness typically decreases by
a small amount on each step. We choose $\alpha$ by a line search algorithm, 
finding the minimum ropelength of configurations in the given direction 
using Brent's method with a relatively low precision. 

However, we do not always accept the ropelength-minimizing $\alpha$. Instead, we apply a collection of ad hoc conditions which we describe as $\alpha$ being ``computationally acceptable''. These include an upper bound on stepsize of $10^{-2}$, a lower bound of $10^{-6}$, and the requirement that the linear algebra solver of Step 4 can compute a new direction $-\grad f_I$ at the new location. These are motivated by several practical considerations. If the stepsize is permitted to be too large, loose configurations will often form large kinked regions before the tube contacts itself. Kinks reduce stepsizes by orders of magnitude-- in practice, this means that such a run takes an unacceptably long time to converge. If the stepsize is permitted to be too small, the solver can stall just before discovering a new self-contact. In these cases it has proved better to take the risk of a slight increase in ropelength in order to improve the strut set. Finally, even when the stepsize is less than $10^{-2}$, if an arc of the knot suddenly contacts another arc, introducing too many new struts into the rigidity matrix, the matrix can become numerically singular, defeating the \tsnnls\ solver of Step~4. Thus, we must look ahead and make sure the next position will be acceptable to \tsnnls\ before locking in a stepsize.
%
%

\subsection{Step $7$. Error correction}
\label{sec:error}
When the error bound $\maxerr = 10^{-4}$ is reached, we use Newton's
method to return $\V_k$ to a configuration with larger thickness. For
any given variation $W$ of $\V$ we can estimate the change in the
$\nicefrac{d(p,q)}{2}$ for $(p,q) \in \Strut(\V)$ and in
$\MinRad^{\pm} v_i$ for $(v_i,\pm) \in \Kink(\V)$ by $A^T W$, where
$A$ is the rigidity matrix we have already computed.

We use this observation in a straightforward way. We construct a
vector $C$ of desired corrections which is equal to $(1 -
\nicefrac{\maxerr}{2}) - \nicefrac{d(p,q)}{2}$ for $(p,q) \in
\Strut(\V)$ and $(1 - \nicefrac{\maxerr}{2}) - \MinRad^{\pm} v_i$ for
$(v_i,\pm) \in \Kink(\V)$.  Having done so, we find a minimum-norm
solution to $A^T W = C$. We then step according to $W$, using a search algorithm to decide the stepsize,  rebuild the
rigidity matrix in case we have changed the strut or kink set in the
correction step, and iterate.

We note that we do not attempt to correct all of the error in $\CThi(\V)$
during this procedure. If we did so, we would risk losing struts and
kinks when we rebuild the rigidity matrix. In that case, the next
Newton step, ignoring those pairs or vertices, might rediscover them as
struts and kinks. In principle, this cycling behavior could delay or
prevent convergence of the Newton procedure, as noted by
Fletcher~\cite{MR1867781}. Our method does not eliminate this
possibility entirely (in the current version of the code, we have observed occasional failures of the Newton solver) but in practice the Newton solver almost always converges in
only a few iterations.

The main problem with the Newton solver is that it is slow for large
problems.  The matrix $A^T$ is mapping from a high-dimensional space
of variations to a relatively low-dimensional space of struts and
kinks, so it has a large kernel. Hence the matrix $A A^T$ is not
positive definite, and so we cannot solve $A^T W = C$ using the method
of normal equations and the fast Cholesky decomposition of
\texttt{TAUCS}. Instead, we must use the older \texttt{lsqr} code of
Paige and Saunders~\cite{355989} to find a minimum-norm solution to
the problem.  This can be very slow. For instance, in a $640$ edge
trefoil with $975$ struts and $10$ kinks, correction steps
consumed anywhere between $3$ and $25$ seconds of runtime. Normal
steps completed in less than a second. We always have the option of sidestepping Newton correction by simply scaling the knot (as in Pieranski's SONO algorithm). This preserves ropelength but destroys the strut set completely, requiring us to rebuild the strut set during subsequent steps. Our experience has been that this can improve performance during the middle stages of a run, when a fairly large number of struts and kinks have formed but the knot is still far from tight, but it is better to use Newton correction in the final stages of a run when one is trying to adjust a converged strut set to improve the final results.

At the moment, the speed of \texttt{lsqr} controls the overall performance of our code. We hope to find an improved error-correction procedure in future versions of the software.

\subsection{Modified versions of the algorithm} We have also modified our algorithm to handle some special cases, such as open curves with fixed endpoints or endpoints constrained to lie in planes. In these cases, the gradients of the endpoint constraints are added to the rigidity matrix and the gradient of length is resolved against them in Step 4. In addition, a specialized error-correction algorithm enforces the constraints after each step to prevent numerical error from causing the endpoints to drift away from their positions over time. The general Newton's method algorithm for error-correction is also modified in these cases to take endpoint constraints into account. 

In addition, we have found that curves whose final tight positions have long segments with no struts or kinks as well as tightly curved regions with many struts and kinks often take a very large number of steps to tighten completely. Sections of the curve with no struts or kinks simply minimize length with no constraints and must therefore end up as straight lines. But as they approach this position, the gradient of length approaches zero, while regions where the gradient of length is balanced by struts and kinks have comparatively large length gradients. Since the step size is controlled by the tightly curved regions, it may take a very long time for the strut and kink-free regions to finish straightening. We have had some success in these cases with a modified version of our algorithm which detects sections of curve with no struts or kinks and scales up the length gradient on those portions of the curve alone.

\section{Results of Computations}

 We now present the main results of our computations. To summarize, we
 have significantly extended the range and quality of existing
 computations of tight knots and links.  The new data support some
 interesting conjectures about the geometric structure of these
 configurations.

 \subsection{Validation of \ridgerunner\ computations}
 To verify that the system works, we checked the results of
 \ridgerunner\ against some theoretical results. The results of the
 comparison appear in Table~\ref{tab:validation}.  As we can see from
 the Table, the relative error in these ropelength computations is
 as small as $0.0017\%$.

 The paper~\cite{MR2284052} also gives an explicit strut set for the
 Borromean rings. To compare the numerically computed strut set to the
 theoretical one, we plot them together in
 Figure~\ref{fig:borrostruts}. The Figure shows that the numerically
 computed strut set is quite close to the actual one.
 Figure~\ref{fig:claspstruts} shows a similar comparison between theoretical
 results and a \ridgerunner\ computation for the strut set of the
 ``simple clasp'' formed by two strands looped over one another. The theoretical results in~\cite{MR2284052} for this clasp assume that the curvature of the clasp is not bounded, so we compare with the results of a run of our software which did not enforce curvature constraints.
 
 \FPeval{twooneresult}{100*((25.1388 - 25.132741228718)/(25.132741228718))}
 \FPround{\twooneresult}{\twooneresult}{2} 
 
 \FPeval{sixthreetworesult}{100*((\SixThreeTwoRRUB - 58.006)/(58.006))}
 \FPround{\sixthreetworesult}{\sixthreetworesult}{4} 
 
 \FPeval{simplechainresult}{100*((\SimpleChainRRUB - 41.6991)/(41.6991))}
 \FPround{\simplechainresult}{\simplechainresult}{2} 
 
 \begin{table}[ht]
   \begin{tabular*}{5.7in}{p{1in}p{0.8in}p{1in}p{0.8in}p{1.4in}} \toprule
     & \hspace{-0.12in} \includegraphics[width=0.6in]{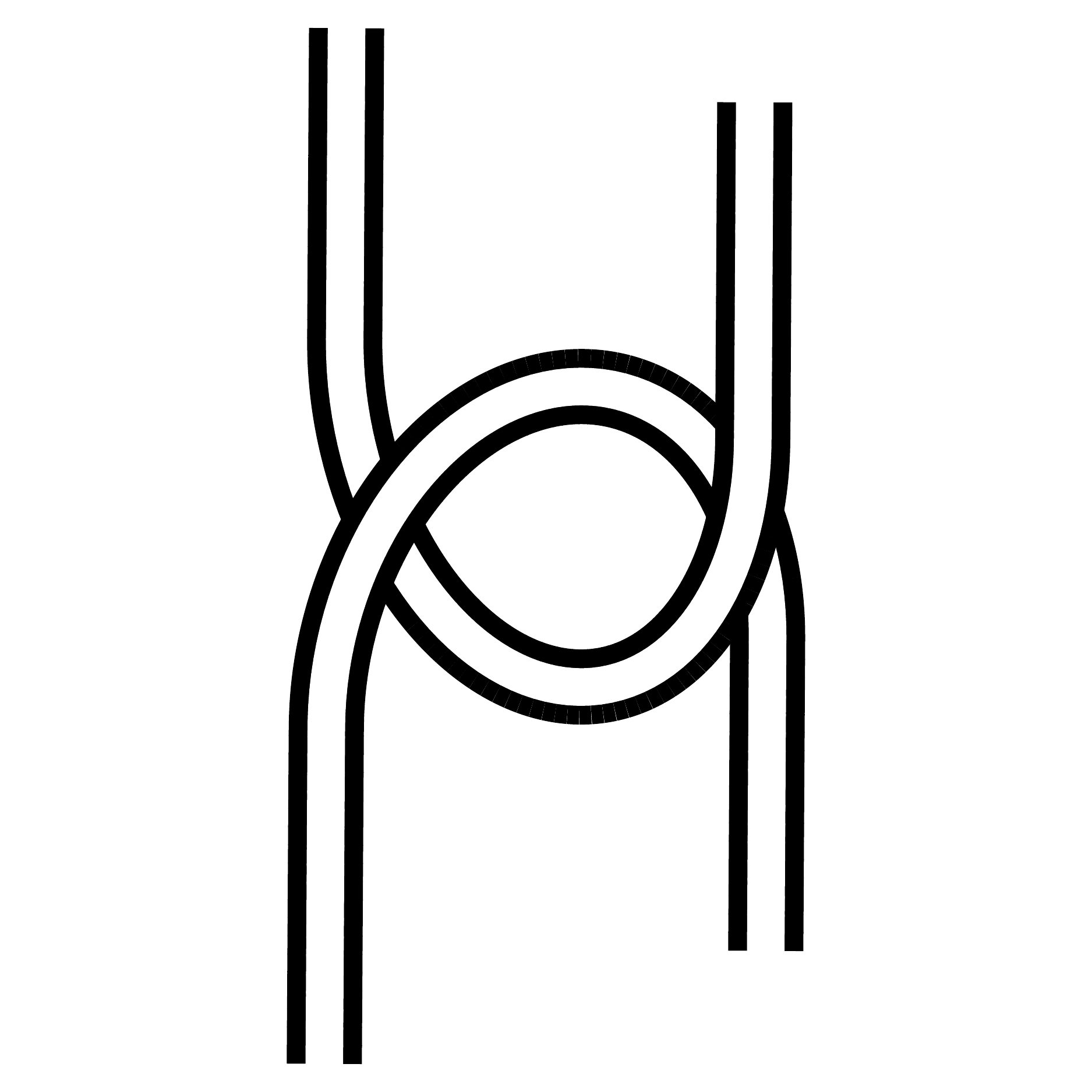} &  \includegraphics[width=0.5in]{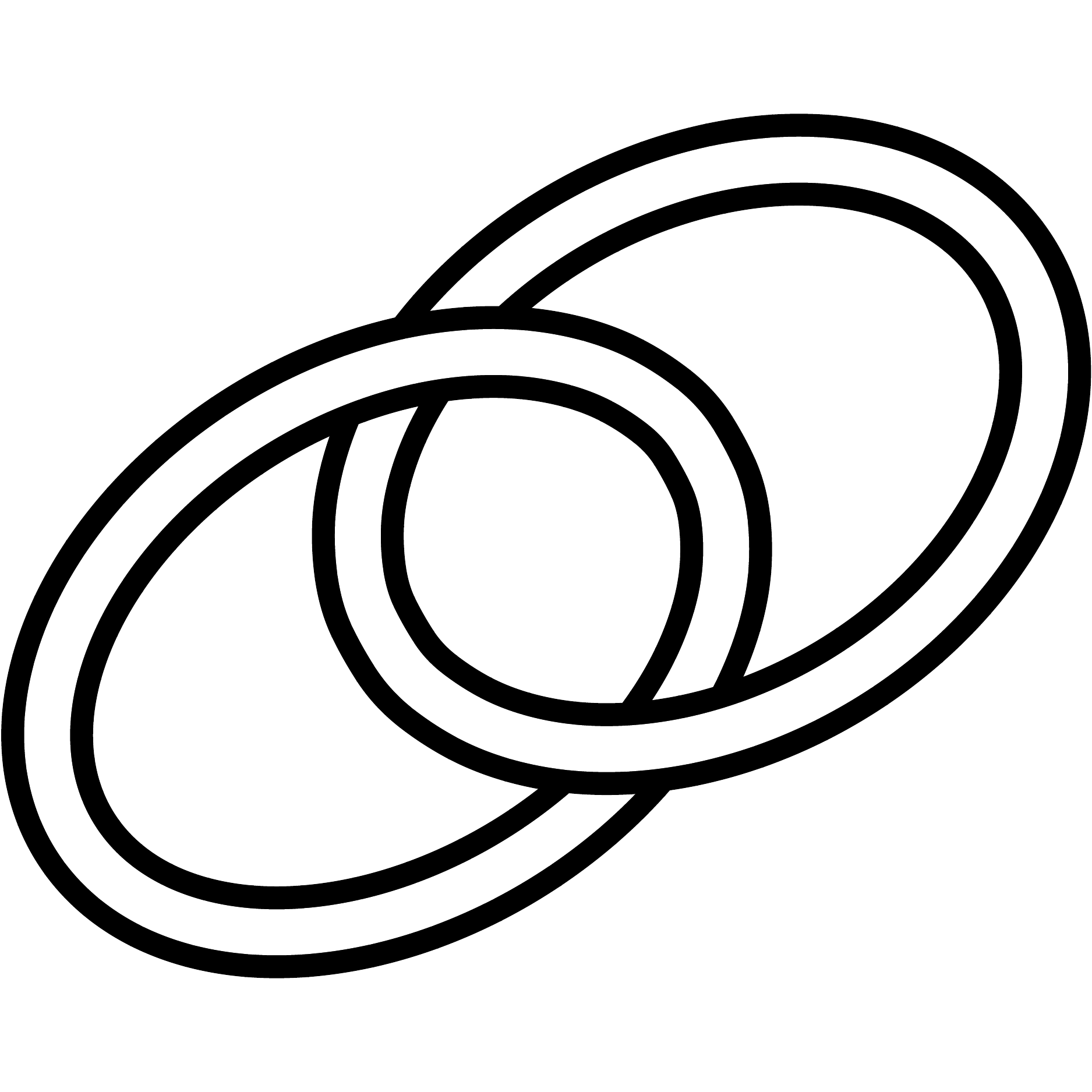} &  \includegraphics[width=0.6in]{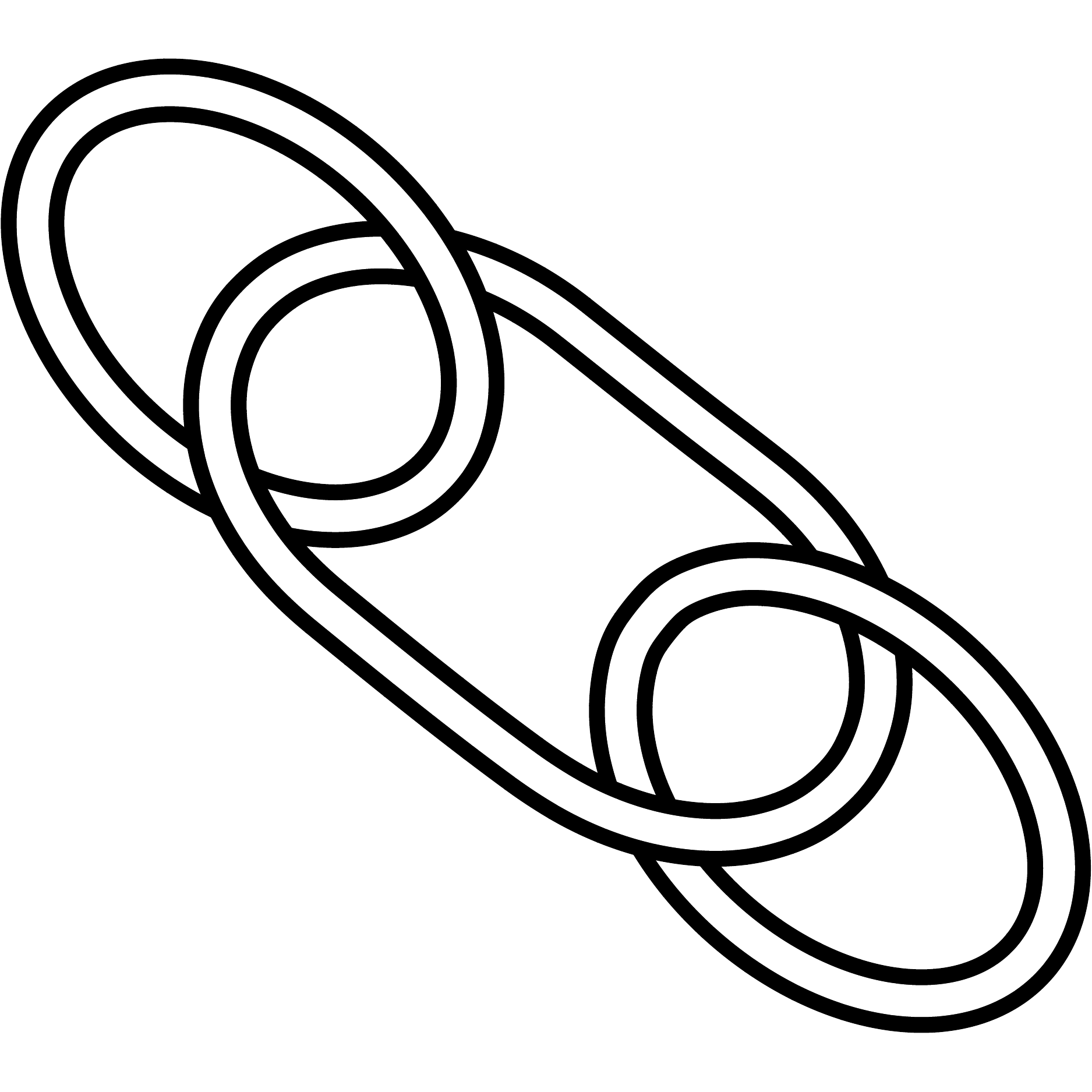} &  \includegraphics[width=0.6in]{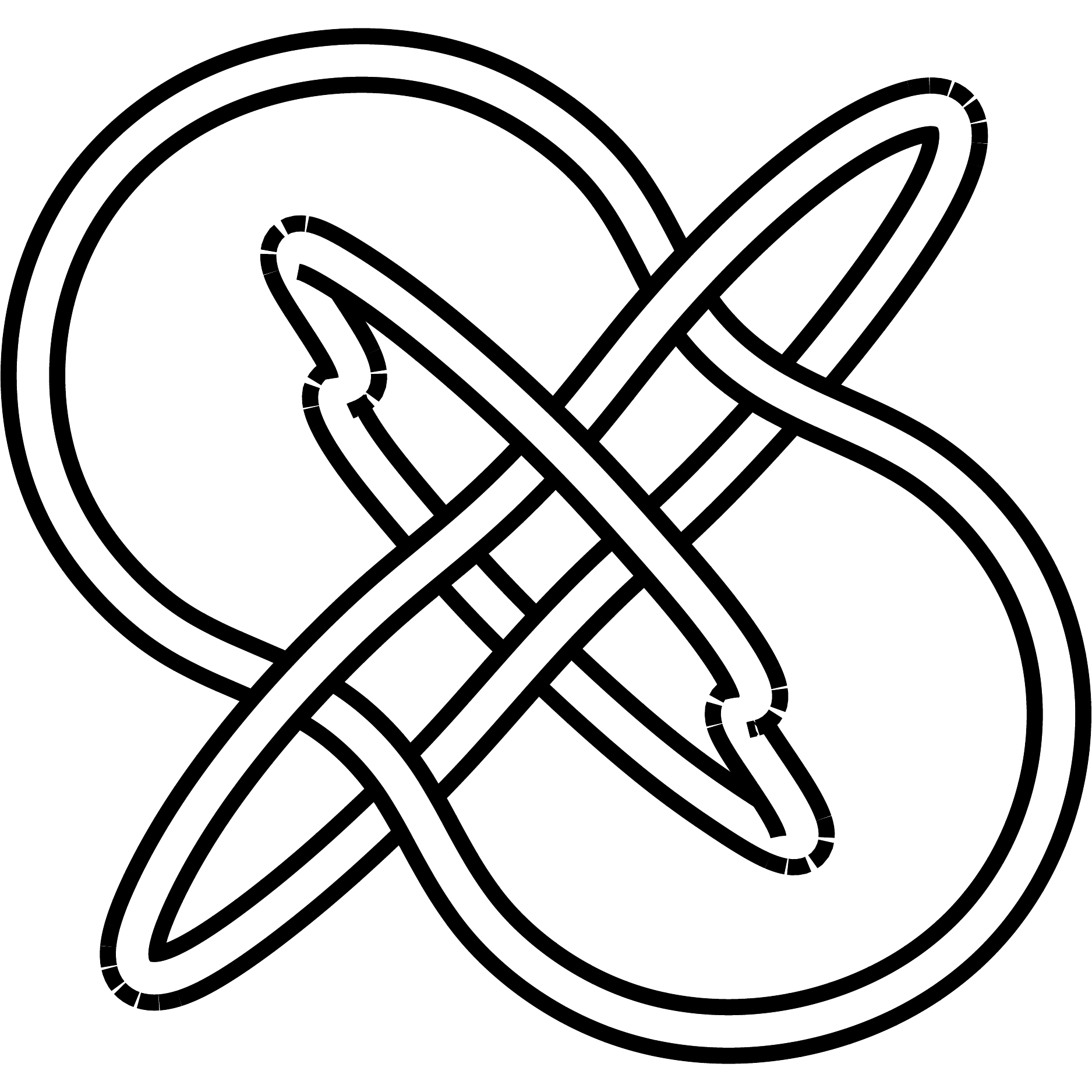} \\ \midrule  
     Link name 	& Clasp 		& Hopf link ($2^2_1$) 	& $2^2_1 \# 2^2_1$ & Borromean rings ($6^3_2$) \\
     Vertices 		& $332$		& $216$		& $384$		& $930$	\\
     $\PRop$ bound    & $4.2841$	& $25.1406$	& $41.7131$	& $58.0192$\\ 
     $\Rop$ bound       & $4.2837$  & \TwoTwoOneRRUB   & \SimpleChainRRUB & \SixThreeTwoRRUB  \\ 
     Smooth length	& $4.2629$\cite{MR2284052} 	& $8\pi$\cite{MR2003h:58014} 
                        & $12\pi+4$\cite{MR2003h:58014}	& $58.0060$\cite{MR2284052}\\
     Relative error	& $0.4\%$ 	& $\twooneresult\%$	& $\simplechainresult\%$	& $\sixthreetworesult\%$ \\ \bottomrule
   \end{tabular*}
   \bigskip

   \caption{Numerical results from \ridgerunner\ compared to the
   minimum ropelength values from~\cite{MR2003h:58014} and~\cite{MR2284052}. The relative
   errors in the computations are quite small.}

   \label{tab:validation}
 \end{table}

\begin{figure}
\begin{center}
  \begin{overpic}[width=4.5in,viewport= 140 140 480 530]{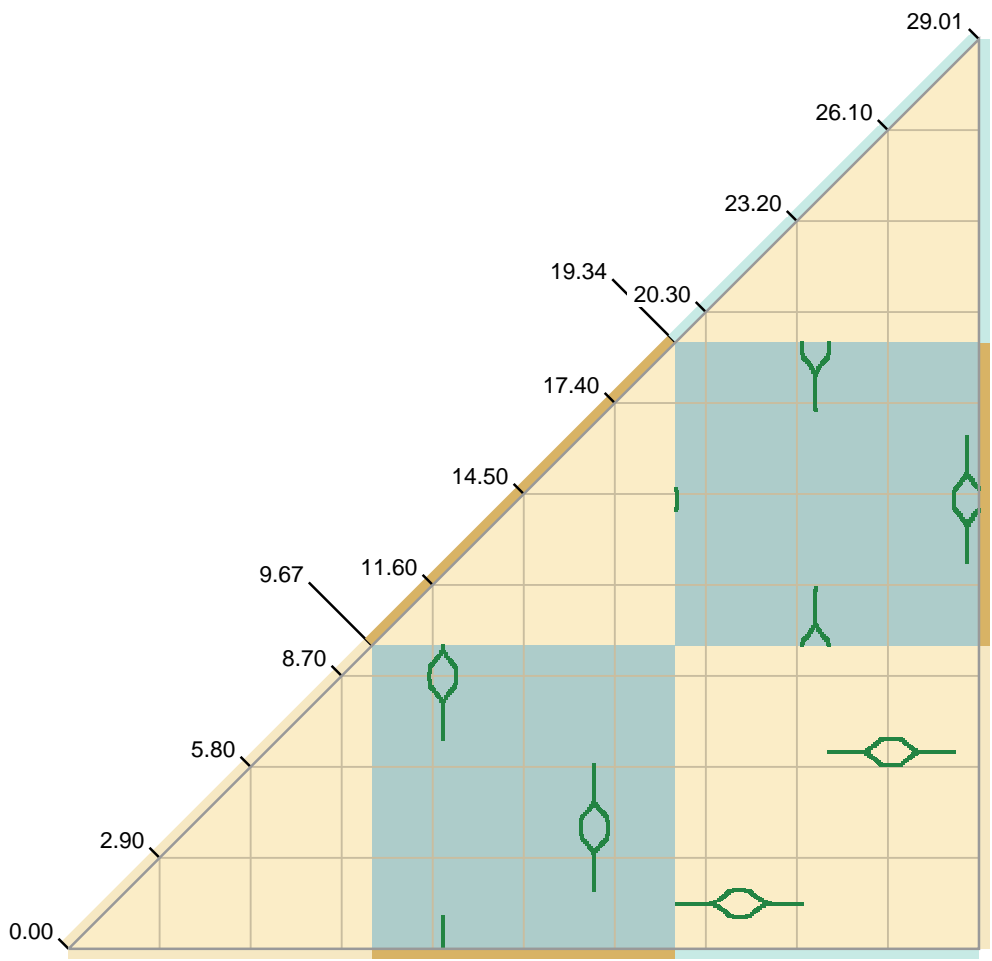}
  \end{overpic}\\
  \vspace{0.25in}
  \begin{overpic}[width=4.25in,viewport= 40 0 320 80]{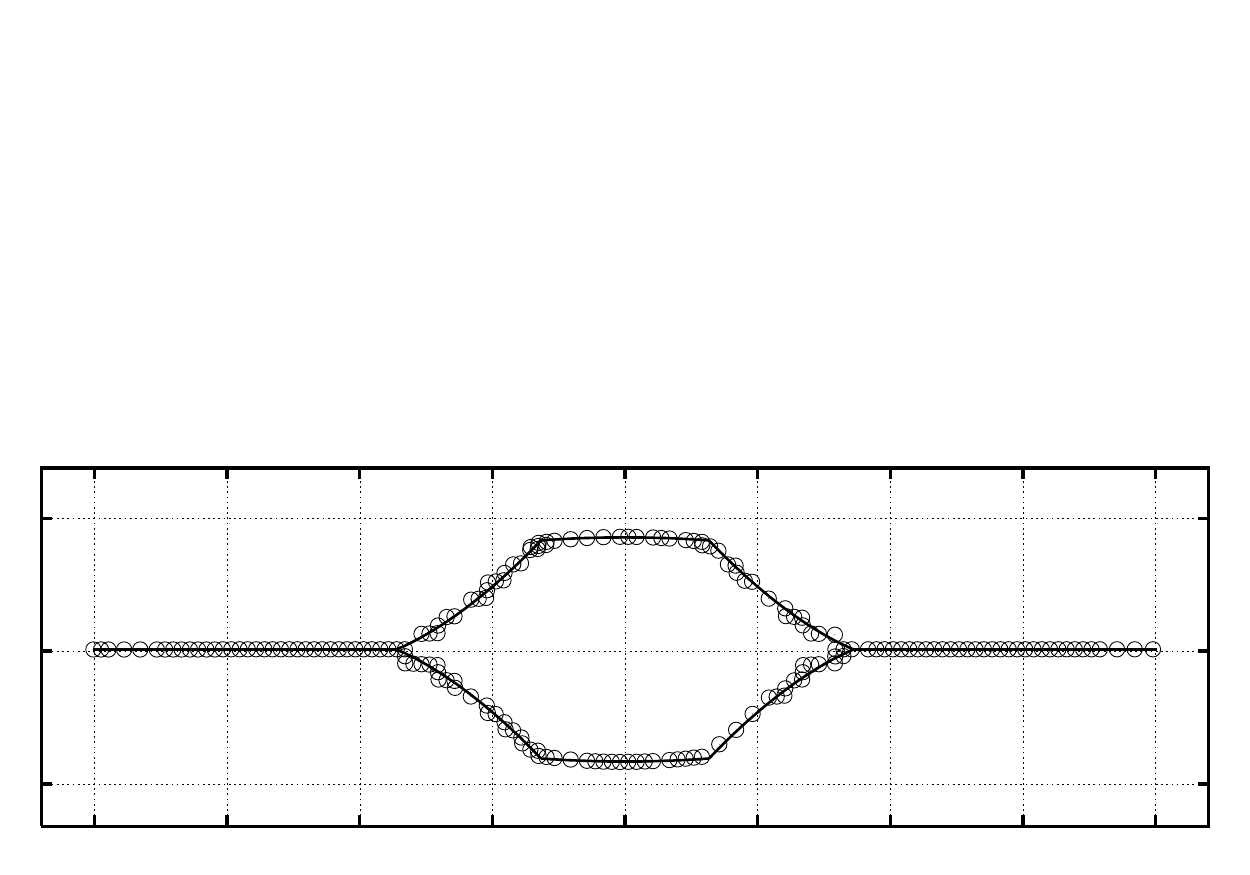}
  \end{overpic}
\end{center}

\caption{The diagonal above is labeled with arclength values along the three components of the Borromean rings link, which is numbered $6^3_2$ in Rolfsen's table. Every pair $(s,t) \in \Strut(\V)$ is represented by a dark green square centered on $(s,t)$.  As we see from the top plot, no tube around a component of the link is in contact with itself (so the three triangles near the diagonal are
 empty). But each of the components makes contact with the other two,
 as shown by the boxes plotted in the rectangles forming the remainder of the plot. We can see that the
 contacts break up naturally into ``lantern-shaped'' structures.
 In the bottom plot, we compare one ``lantern'' to the self-contact set predicted by~\cite{MR2284052}, which
 is represented by a black line.}
\label{fig:borrostruts}
\end{figure}

 \begin{figure}
   \begin{center}
     \begin{overpic}[width=1.5in]{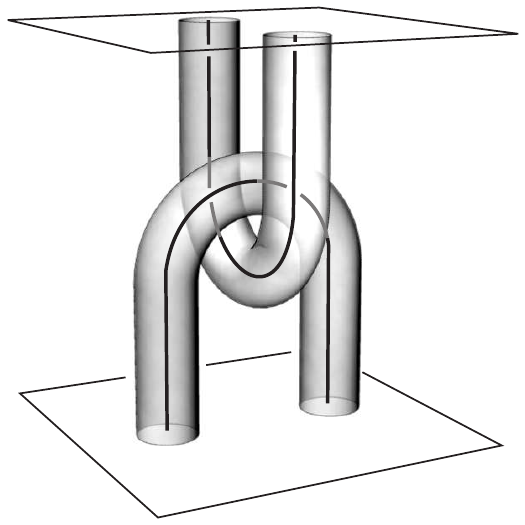}
     \end{overpic}
     \hspace{0.5in}
     \begin{overpic}[width=1.5in]{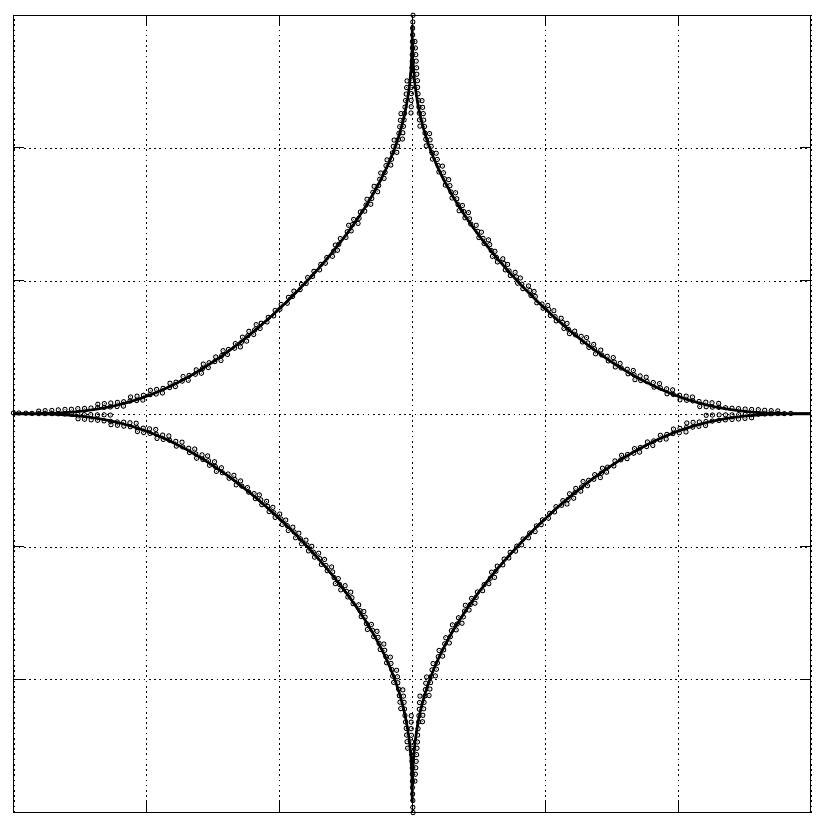}
     \end{overpic}
     \hspace{0.5in}
     \begin{overpic}[width=1.5in]{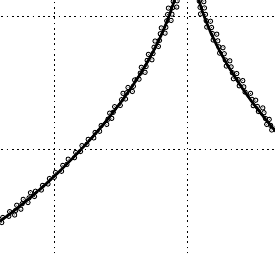}
     \end{overpic}
   \end{center}
   \caption[Computational and theoretical strut sets for the clasp]
	   {The left-hand picture shows a (loose) configuration of the
	   ``simple clasp'' --- a simple two-strand tangle which serves
	   as an interesting model for the interaction between two
	   ropes passing over each other at right angles. A
	   ropelength-critical configuration of this tangle has been
	   derived and studied extensively in~\cite{MR2284052} and~\cite{cfksw2}. Since
	   this derivation included an explicit strut set, it is
	   natural to compare \ridgerunner's results to this
	   theoretical picture.	     
	   This comparison is shown in the two plots center and right,
	   which plot the positions of struts in arclength coordinates
	   with the origin located where each curve first begins to
	   turn. The enlarged plot (right) shows the agreement between
	   theoretical and computational results. The data shown is from a 332 edge polygonal clasp.}
	   \label{fig:claspstruts}
 \end{figure}

\subsection{Computing polygonal ropelength minimizers for many knots and links}
We minimized polygonal ropelength for all prime knots of 10 and fewer crossing and all prime links of 9 and fewer crossings (a total of \FileCount\ knot and link types) at
resolutions of at least $8$ vertices per unit of ropelength (several hundred vertices in total). For a few knots and links of special interest, we computed high resolution runs with 16, 32, or 74 vertices per unit ropelength. The largest runs in our dataset contain about 2400 vertices. 

The computations were performed on clusters at the University of St.\,Thomas, the University of Georgia, and the ACCRE cluster at Vanderbilt University. We began our computations with an initial
low-resolution (200 vertices or fewer) polygon, which we ran until
the residual was sufficiently low. We then increased resolution by
a minrad-preserving version of spline interpolation and minimized again from the resulting new starting configurations. Our inital goal was a residual less than $0.01$, which we achieved for \DoubleOOneCount\ of the \FileCount\ knots and links in our data set. We were able to reach a residual of $0.001$ for \TripleOOneCount\ of the knots and links in our data set, proving that our knots are close to being critical for the $\CThi$ thickness. While our knots are not quite equilateral, they all satisfy the hypotheses of Corollary~\ref{cor:cthi vs pthi} and are hence also close to critical for the original $\PThi$ thickness. Because of this corollary, we know that both thicknesses are equal for our configurations, so we have computed and reported the $\PThi$ thickness and ropelength below.

We started each knot from at least five initial configurations, including the configurations from KnotPlot~\cite{knotplot} (similar to the configurations in Rolfsen's table), the TOROS simulated annealer~\cite{MR2034393}, Gilbert's minimized configurations from the online~\emph{Knot Atlas}~\cite{gilbert}, hand-drawn configurations from Kawauchi's \emph{A Survey of Knot Theory}~\cite{MR1417494}, and positions generated from KnotPlot's \texttt{diagram} command. The results shown describe the lowest ropelength we achieved from any of these starting configurations.

The polygonal ropelengths for our curves appear in the column $\PRop$ of Tables~\ref{ByKnotTableone}-\ref{ByKnotTablethree} of Appendix~\ref{ropdata}, while a plot of the ropelengths organized by crossing number appears in Figure~\ref{fig:data}.

\subsection{Generating upper bounds for smooth ropelength}
\label{sec:RRUB}
Our computations yielded a large set of approximate minimizers of
$\nicefrac{\Len(\V)}{\PThi(\V)}$. From these, we wanted to generate
upper bounds on the minimum (smooth) ropelength of these knots and
links. Rawdon has given general bounds~\cite{MR1702029, MR2034393} on
the rate at which $\PThi \rightarrow \Thi$ which we could have used
for this purpose.  But we were interested in small improvements in
ropelength, so we used a more careful approach.

Our procedure for constructing smooth ropelength bounds from polygonal
data is as follows. Beginning with $\V$, we splice circle arcs of
radius $\MinRad(v_i)$ into the corners at vertices $v_i$ as shown on
the left-hand side of Figure~\ref{fig:splice-and-distbound} to create
a piecewise $C^2$ curve $V(s)$.  The minimal radius of curvature for
this curve is equal to $\MinRad(\V)$. But the self-distances of $V(s)$
may be different from those of the polygon $\V$ if they involve the
new circle arcs.

\begin{figure}[h]
  \begin{center}
    \begin{overpic}{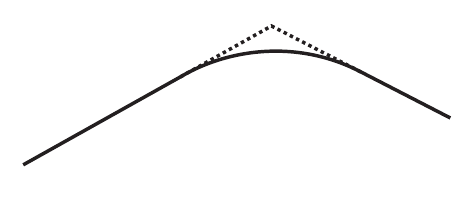} 
      \put(1,2){$v_{i-1}$}
      \put(54,40){$v_i$}
      \put(95,12){$v_{i+1}$}
    \end{overpic}
    \hspace{1in}
    \begin{overpic}{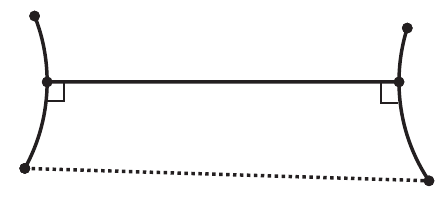}
      \put(-6,25){$c(0)$}
      \put(95,25){$d(0)$}
      \put(51,30){$x$}
      \put(-1,-2){$c(s_0)$}
      \put(100.5,1){$d(t_0)$}
      \put(30,9){$x + (1 + K)\epsilon^2$} 
     \end{overpic}
  \end{center}
  \caption[Splicing a circle into a polygon]{On the left, we see the
    curve constructed from splicing a circular arc of radius
    $\MinRad(v_i)$ into $v_{i-1} v_i v_{i+1}$. This curve is $C^1$, but
    not $C^2$ at the splice points. On the right, we see the setup for
    Proposition~\ref{prop:approx}. On the left and right are arcs
    $c$ and $d$ with curvature $\leq K$ and length $\leq
    \epsilon$. The minimum distance $x$ between them occurs at
    $c(0)$, $d(0)$. We prove that the distance between any other
    pair of points $c(s_0)$ and $d(t_0)$ is bounded above by $x +
    ( 1 + K)\epsilon^2$.}
  \label{fig:splice-and-distbound}
 \end{figure}
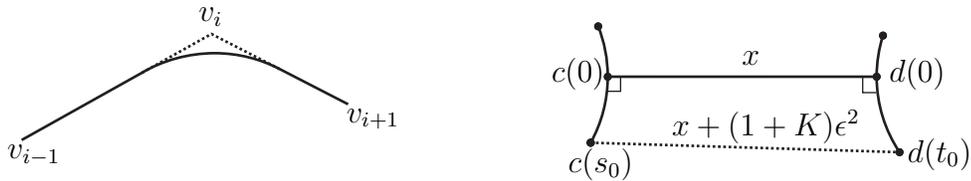

 We must therefore compute the self-distances of $V(s)$. This poses a
 problem: $\V(s)$ is composed of arcs of circles and line segments and
 Neff has shown that there is no simple formula for the distance
 between two arbitrary circle arcs in 3-space \cite{MR1084085}. So we
 estimate the self-distances of the smooth curve $V(s)$ by taking
 distances between a finite number of sample points on the curve
 separated from one another by some $\epsilon$. We bound the error in
 our computation in terms of $\epsilon$ using the following
 Proposition.

 \begin{proposition} 
   Suppose that $c(s)$ and $d(t)$ are each unit-speed piecewise $C^2$
   arcs with curvature bounded above by $K$. Further, suppose that
   $\norm{c(0) - d(0)} > \nicefrac{1}{2}$ is the minimum distance
   between $c$ and $d$. Then for any $0 \leq s_0, t_0 \leq \epsilon$
   \begin{equation*}
     \norm{c(s_0) - d(t_0)} \leq \norm{c(0) - d(0)} + 
     \left( 1 + K \right)\epsilon^2.
   \end{equation*}
   \label{prop:approx}
 \end{proposition}

 \begin{proof} 
   Since $\norm{c(s) - d(t)}$ has a local min at $(0,0)$, we know that
   \begin{equation*}
     \left< c'(0), c(0) - d(0) \right> = 0, \quad \text{
     and } \quad \left< d'(0), c(0) - d(0) \right> = 0.
   \end{equation*}
   Further, the curvature bound tells us that $\norm{c''},
   \norm{d''} < K$. We will use these facts to estimate
   $\norm{c(s_0) - d(t_0)}^2$. If we let $C(s_0) =
   \int_0^{s_0} c'(s) \, \mathrm{d}s$ and $D(t_0) =
   \int_0^{t_0} d'(t) \, \mathrm{d}t$ then we have $c(s_0) =
   C(s_0) + c(0)$ and $d(t_0) = D(t_0) + d(0)$, so
   \begin{equation}
     \label{eq:squared}
     \norm{ c(s_0) - d(t_0) }^2 = \norm{C(s_0) -
     D(t_0)}^2 - 2 \left< C(s_0) - D(t_0), c(0) -
     d(0) \right> + \norm{ c(0) - d(0)}^2.
   \end{equation}
   Since $c(s)$ and $d(t)$ are unit-speed curves, and $0 \leq s_0, t_0
   \leq \epsilon$ we know that $\norm{C(s_0)}, \norm{D(t_0)} <
   \epsilon$ and so the first term is bounded above by $4\epsilon^2$.

   The middle term is more interesting. As before, we can let $CC(s) =
   \int_0^s c''(x)\, \mathrm{d}x$ and $DD(t) = \int_0^t
   d''(y)\,\mathrm{d}y$, so $c'(s) = CC(s) + c'(0)$ and $d'(t) = DD(t) +
   d'(0)$. Since $c'(0)$ and $d'(0)$ are normal to $c(0) - d(0)$, we
   can then write this middle term as
   \begin{equation*}
     -\left< C(s_0) - D(t_0),c(0) - d(0) \right> = -\left<
      \int_0^{s_0} CC(s) \, \mathrm{d}s - \int_0^{t_0} DD(t) \,
      \mathrm{d}t, c(0) - d(0) \right>
   \end{equation*}
   Since $\norm{c''}, \norm{d''} < K$, we know $\norm{CC(s)} < Ks$,
   $\norm{DD(t)} < Kt$. Thus (remembering that $s_0$, $t_0 <
   \epsilon$) the norms of the integrals on the right above are each
   bounded above by $K \nicefrac{\epsilon^2}{2}$ and the entire dot product is
   bounded above by $K \epsilon^2 \norm{c(0) - d(0)}$.

   Thus the right hand side of~\eqref{eq:squared} is bounded by
   $\norm{c(0) - d(0)}^2 + 4\epsilon^2 + 2 K \epsilon^2 \norm{c(0) -
   d(0)}$. Since $\nicefrac{1}{2} < \norm{c(0) - d(0)}$, $4 \epsilon^2 < 2
   \epsilon^2 \norm{c(0) - d(0)}$. Using this, we see that 
   \begin{multline*}
      4\epsilon^2 + 2 K \epsilon^2 \norm{c(0) -
       d(0)}  + \norm{c(0) - d(0)}^2 < \norm{c(0) - d(0)}^2 + (2 + 2 K) \epsilon^2 \norm{c(0) -
       d(0)} \\ < \norm{c(0) - d(0)}^2 + (2 + 2 K) \epsilon^2 \norm{c(0) -
       d(0)} + (1 + K)^2 \epsilon^4 \\
       = \left(\norm{c(0) - d(0)} + (1 + K) \epsilon^2\right)^2.
   \end{multline*}
   This completes the proof.
 \end{proof}

 Our code, named \texttt{roundout\_rl}\footnote{freely available as
 part of the \octrope\ library}, establishes a coarse net of points on
 $V(s) \cross V(s) \simeq [0,1] \cross [0,1]$ and then eliminates
 subsquares of this square from consideration using
 Proposition~\ref{prop:approx}. The remaining squares are then
 subdivided and searched in turn. The process terminates once we have
 computed the local minima of $d(p,q)$ on the square with whatever
 accuracy we require.

 Using \texttt{roundout\_rl} in double-precision machine arithmetic we found upper bounds for the ropelengths of our \FileCount\ minimized configurations. These figures appear in column $\Rop$ of Tables~\ref{ByKnotTableone}-\ref{ByKnotTablethree} of Appendix~\ref{ropdata}. These figures constitute the best known dataset on the lengths of tight knots and links. The data is summarized in Figure~\ref{fig:data} and Table~\ref{tab:bestworst}.
 
 \newcommand{\Cr}{\operatorname{Cr}}
 
 \begin{figure}[ht]
 
 \begin{center}
 \begin{overpic}[width=5in,viewport=0 0 370 160]{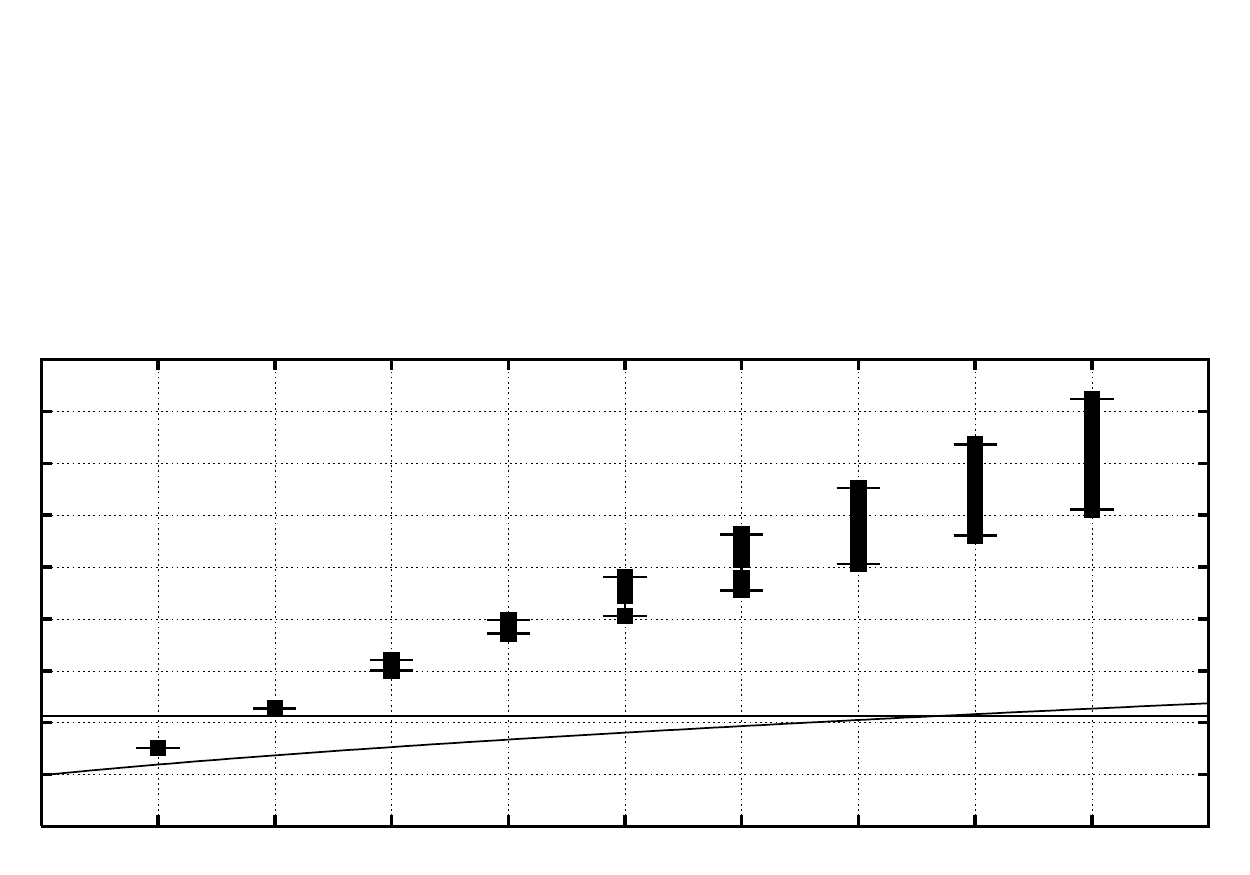} 
 \put(11.5,0){2}
 \put(20.5,0){3}	
 \put(29.5,0){4}	
 \put(39,0){5}	
 \put(48.0,0){6}	
 \put(57.0,0){7}
 \put(66.0,0){8}
 \put(75.0,0){9}
 \put(83.0,0){10}
 
 \put(-1.5,11){30}
 \put(-1.5,19){50}
 \put(-1.5,27){70}
 \put(-1.5,35){90}
 
 \end{overpic}
 \end{center}
 
 \caption{This graph shows the relationship between ropelength (y-axis) and crossing number (x-axis) for knots and links in our data set. The bottom lines show the bound of Denne et al.~\cite{MR2207788} for ropelength of a nontrivial knot (horizontal line) and Diao's bound~\cite{MR1953620} for ropelength in terms of crossing number (curve). We can see that there is a substantial overlap of ropelength values between different crossing numbers. This is reflected in Tables~\ref{ByRopTableone}-\ref{ByRopTabletwo} of Appendix~\ref{ropdata}, which show the knots in ropelength order. Table~\ref{tab:bestworst} shows the links of least and most ropelength for each crossing number.}
\label{fig:data}
\end{figure}
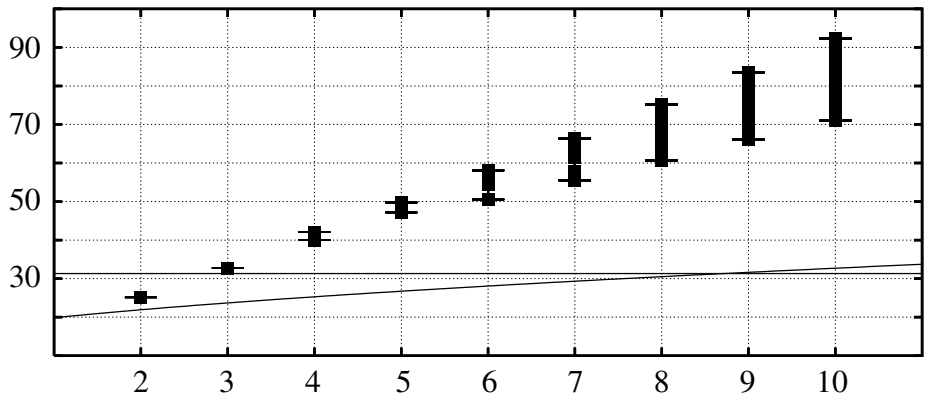

\begin{table}[ht]

\begin{center}
\input{BestWorstTable.tex}
\end{center}

\medskip

\caption{This table shows the links of smallest and largest minimum ropelength for each crossing number (according to our data). Recall that we did not minimize ten-crossing links, so it is likely that some ten-crossing link has more or less ropelength than the $10_{123}$ and $10_{124}$ knots.}
\label{tab:bestworst}
\end{table}


 To test how accurate these final results are likely to be, we
 computed the relative residual $\nicefrac{\norm{(-\grad
 f)_I}}{\norm{-\grad f}}$ for all these knots and links. The average residual of knots in our tabulation is about $\AverageResidual$. We have achieved residuals as low as $2.54 \times 10^{-5}$ for knots and links of special interest, such as $8_{18}$, $10_{123}$, the trefoil, and the  Borromean rings. A table of these residuals appears in Appendix~\ref{ropdata}. Four knots and links in our calculation turned out to be particularly difficult for \ridgerunner: $10_{61}$, $8^3_{10}$, $8^4_3$ and $9^3_{17}$.  
 
\subsection{Generation of tightening animations, pictures, and strut sets}
 We have saved the minimization runs for each
 of these knots and links as an animation showing the tightening
 knot. These animations are posted on the web at
 \url{http://www.jasoncantarella.com/movs/}.
 
We have also generated images of the polygonal strut sets and approximately tight configurations for each of the \FileCount\ knots and links in our data set. Space considerations prevent us from including all of this data in this paper, so they are enclosed in the associated \emph{Atlas of Tight Links}~\cite{tightlinkatlas}. Figure~\ref{atlas} shows a typical page from the \emph{Atlas}. All of our tight knot and link data, including coordinates for the tight configurations, is publicly available with the publication of this paper. We note that for technical reasons, our minimized configurations have thickness close to $1/2$ (rather than $1$, as in the discussion above), and hence their maximum curvature is $2$.

\begin{figure}

\hphantom{.}
\hspace{0.3in}
\begin{minipage}[t]{3in}
\newpage
\label{SevenOnePage}
\begin{tabular}{ll} \includegraphics[width=0.9in]{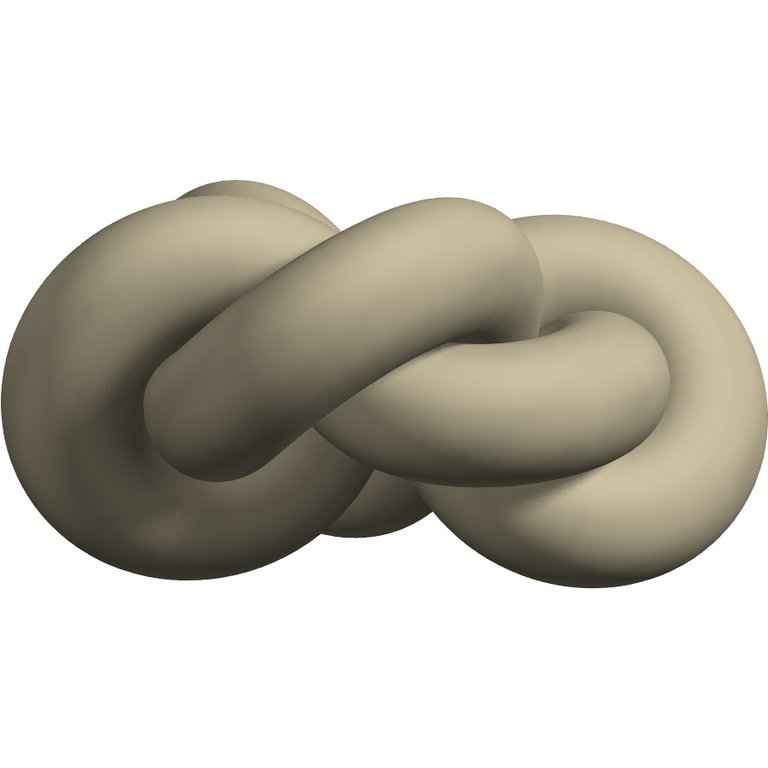} & \includegraphics[width=0.9in]{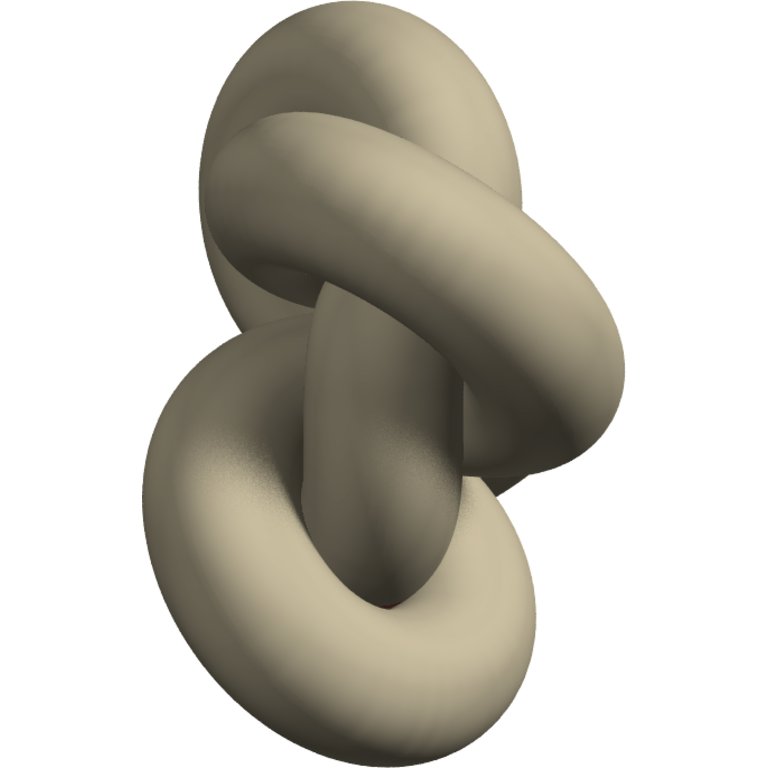} \\
\includegraphics[width=0.9in]{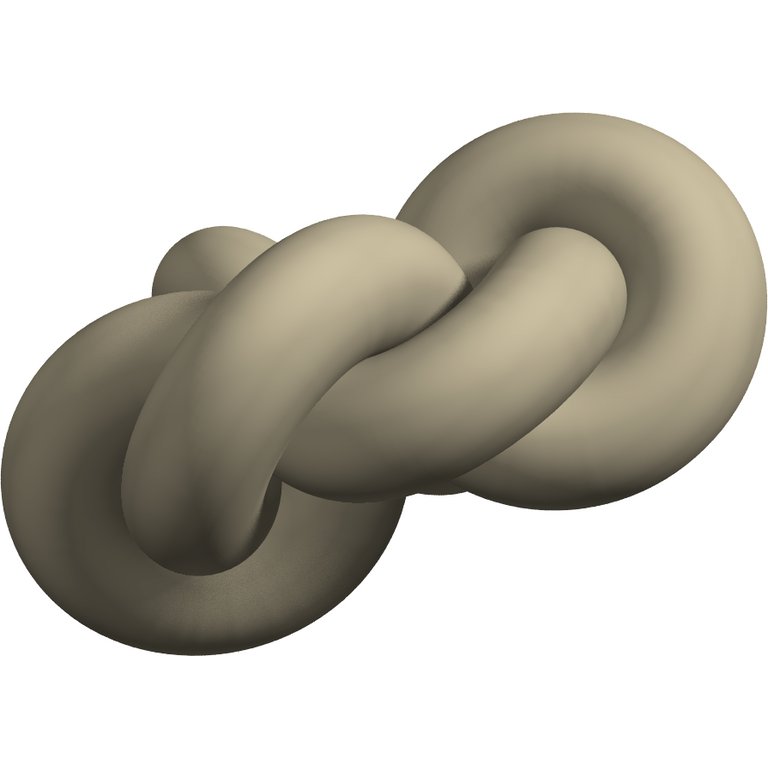} & \end{tabular} 

\vspace{-1.4in}

\hspace{0.4in}\includegraphics[width=1.6in,viewport=170 165 440 570]{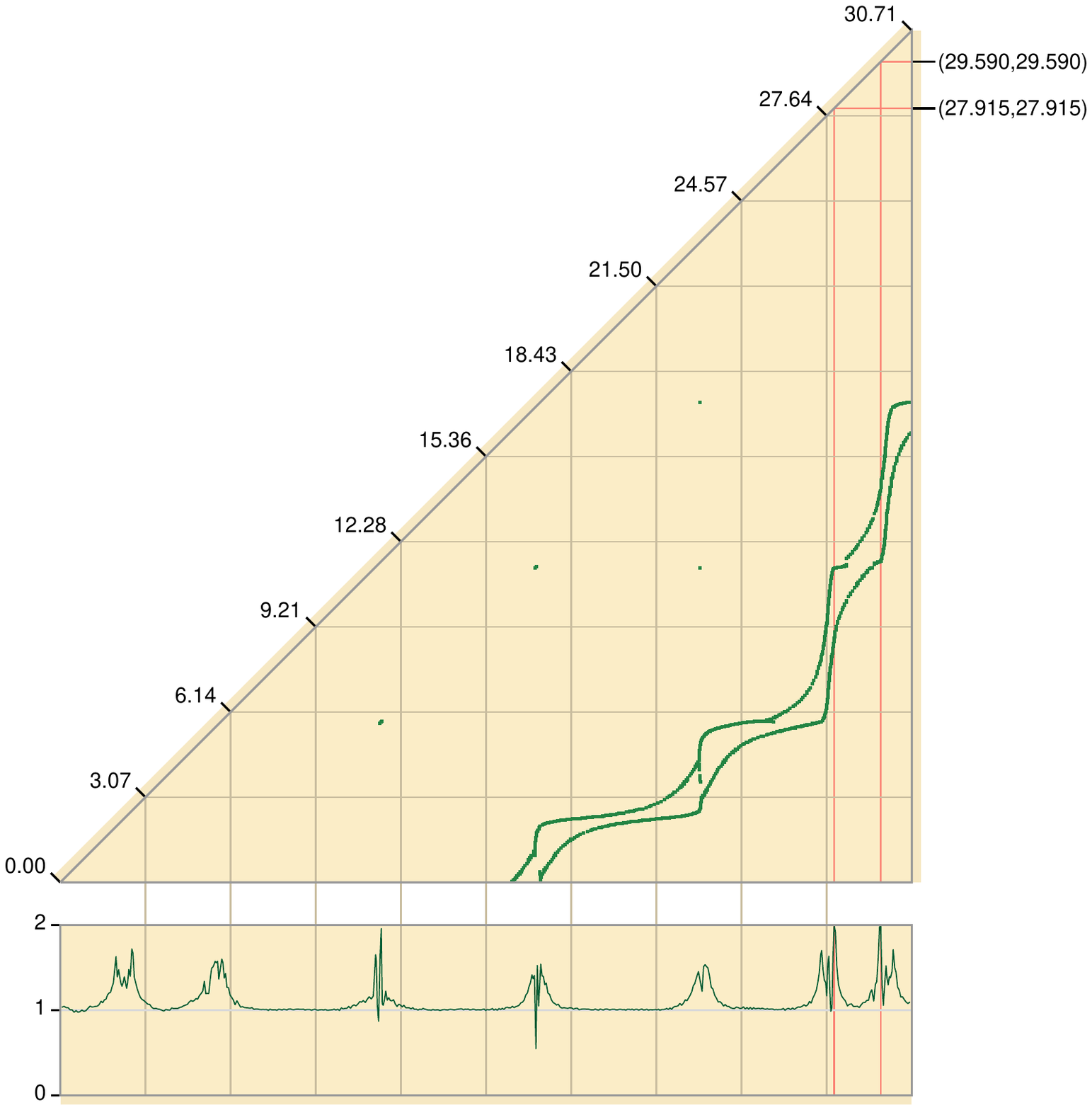} 

\vspace{0.4in}

\begin{tiny}
\hspace{0.1in}
\begin{tabular}{lllll} \toprule
Link & $\PRop$ & $\Rop$ & Verts & Struts \\ \midrule
$7_{1}$ & $61.4234$ & $61.4067$ & $512$ & $788$ \\
\bottomrule
\end{tabular} 
\end{tiny}
\end{minipage}
\hspace{-0.3in}
\begin{minipage}[t]{3in}
\label{FourTwoOnePage}
\begin{tabular}{ll} \includegraphics[width=0.9in]{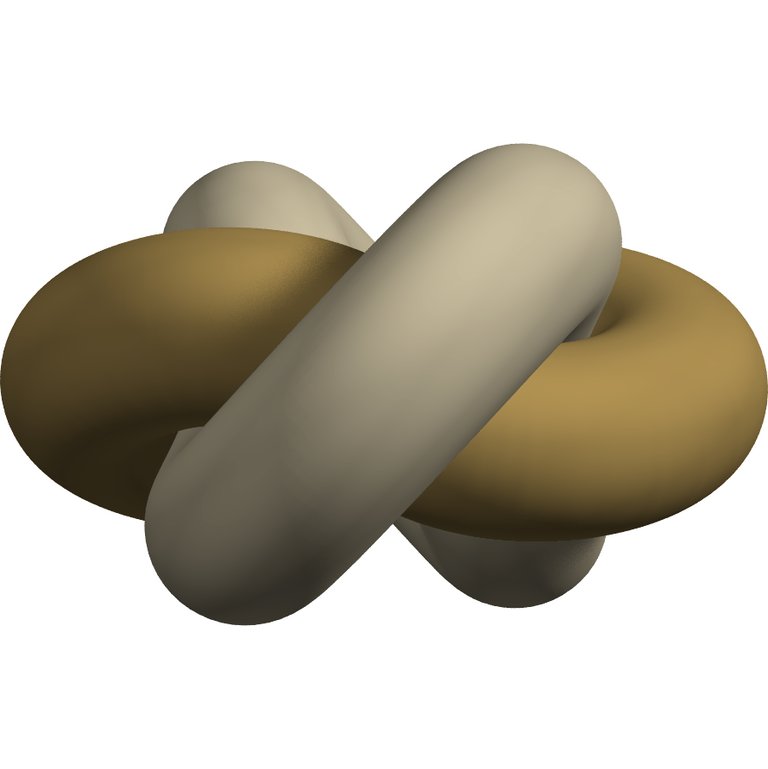} & \includegraphics[width=0.9in]{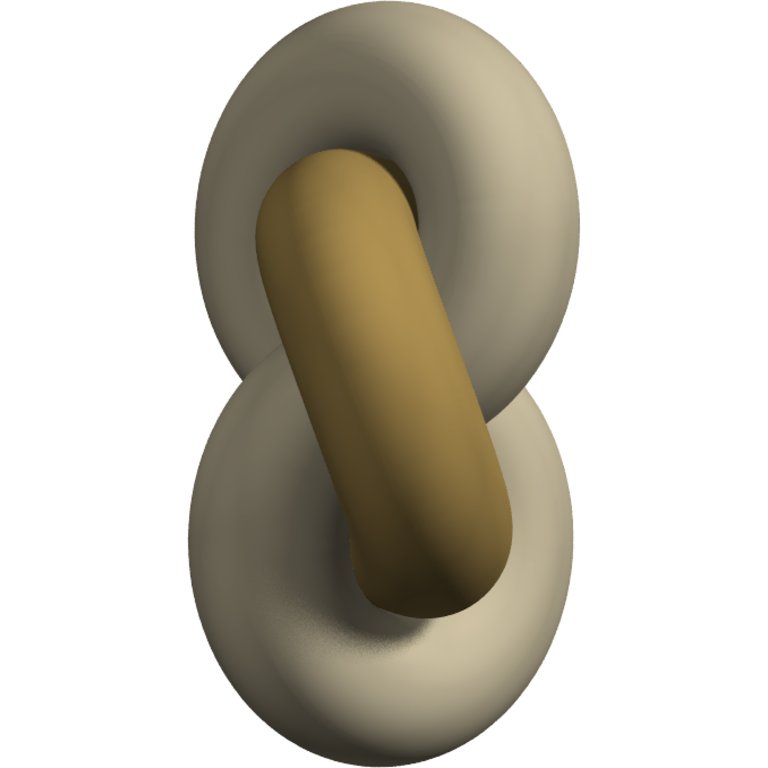} \\
\includegraphics[width=0.9in]{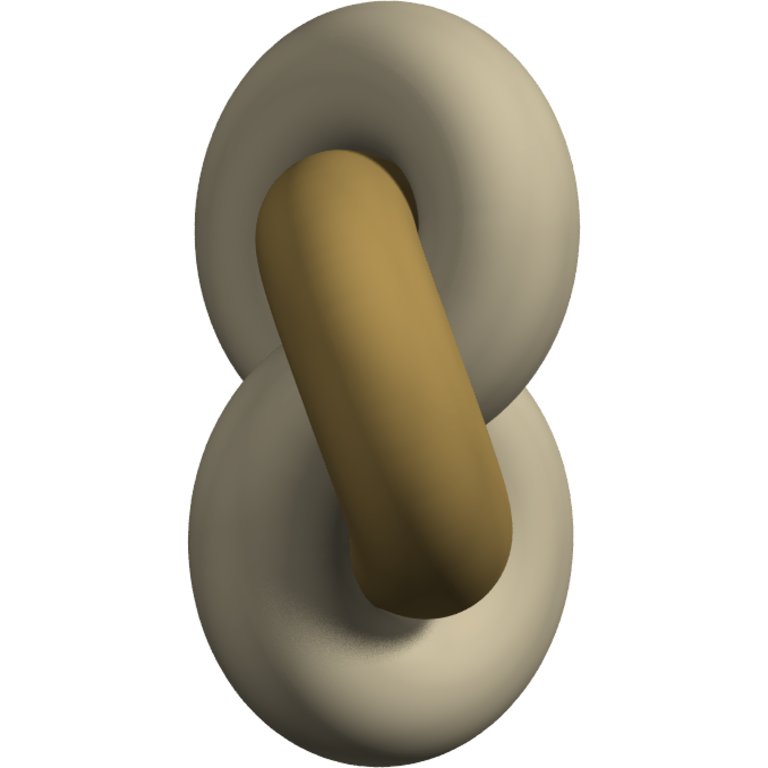} & \end{tabular} 

\vspace{-1.4in}

\hspace{0.4in}\includegraphics[width=1.6in,viewport=170 165 440 570]{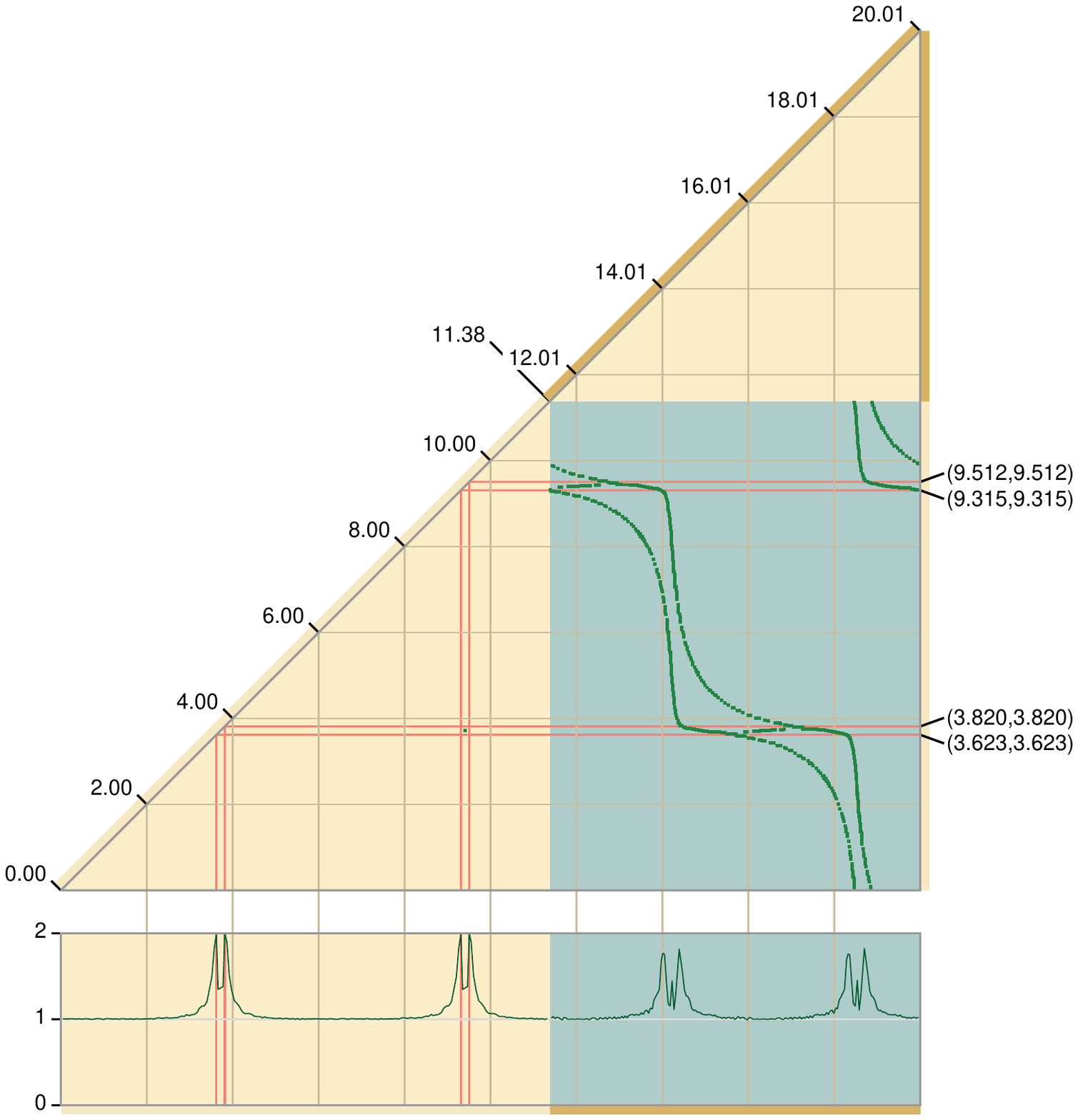} 

\vspace{0.4in}

\begin{tiny}
\hspace{0.1in}
\begin{tabular}{lllll} \toprule
Link & $\PRop$ & $\Rop$ & Verts & Struts \\ \midrule
$4^{2}_{1}$ & $40.0203$ & $40.0122$ & $400$ & $610$ \\
\bottomrule
\end{tabular}
\end{tiny}
\end{minipage}

\caption{This figure shows simplified versions of two pages from the~\emph{Atlas of Tight Links} for the knot $7_1$ and the link $4^2_1$. On each page, the top left pictures show three views of the link. The triangular graphic shows the struts of the link as found by \ridgerunner\ plotted as points $(s,t)$ in arclength coordinates along the link. The graph on the bottom of the page shows the curvature of the curve. The background of each plot changes color to indicate the change from one component to the next. The key along the left-to-right diagonal is given in ropelength units and color-coded with the pictures at upper left to show which component is referred to by the plot.}
\label{atlas}
\end{figure}

\subsection{Discovery of symmetric tight knots}
An interesting feature of the ropelength function is that minimizing ropelength seems to break any symmetry enjoyed by the original configuration of a given knot. For instance, while the minimizing configuration for the $(3,2)$ torus knot $3_1$ appears to be threefold symmetric (as expected), the minimizing configuration for the $(5,2)$ torus knot $5_1$ is not fivefold symmetric. It was therefore somewhat surprising to discover two knots in our data set, $8_{18}$ and $10_{123}$ for which the tight configurations are highly symmetric. These knots are shown in Figure~\ref{symmetric}. Their self-contact sets (which appear on pages $67$ and $358$ of the \emph{Atlas}, and are reproduced in the Appendix of this paper on pages~\pageref{EightEighteenPage} and~\pageref{TenOneHundredAndTwentyThreePage}) are highly suggestive, resembling those of the Borromean rings (page $29$), and appearing to consist of a single element repeated several times. This feature implies that these knots may be better candidates for explicit solution than the seemingly simpler trefoil knot.

\begin{figure}
\includegraphics[height=2.75in]{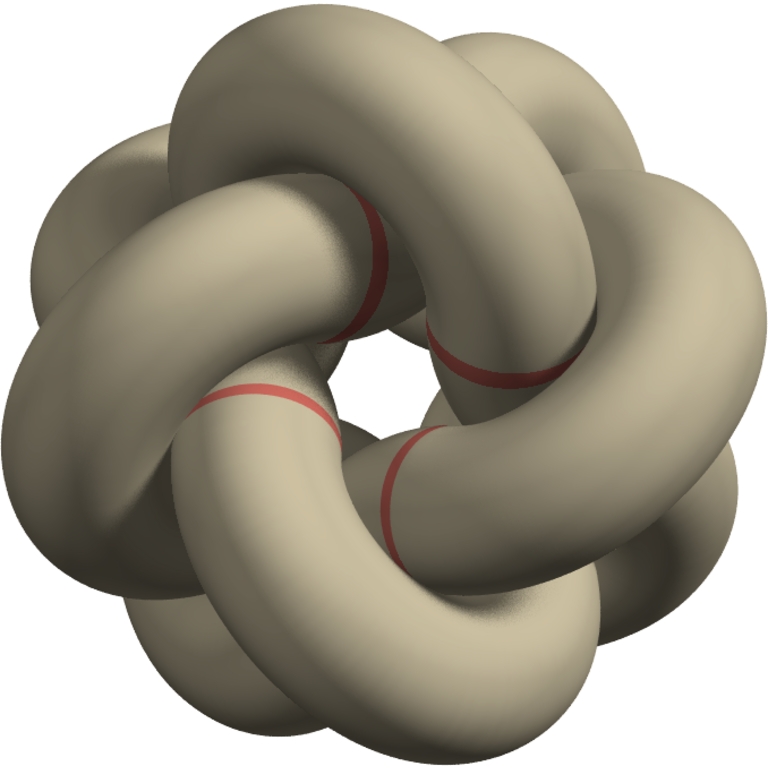}  \hspace{0.5in}
\includegraphics[height=2.75in]{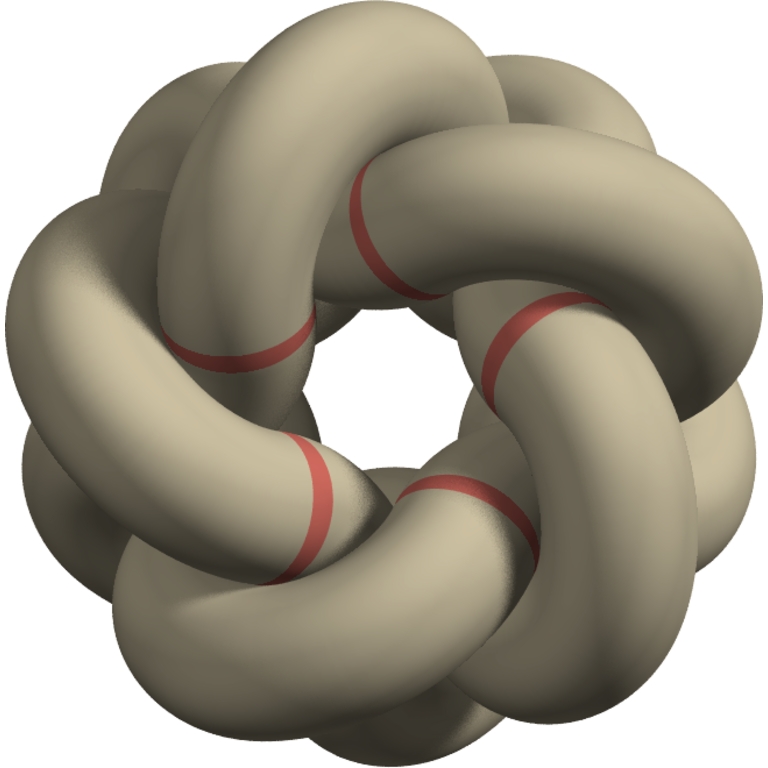}
\caption{Two highly symmetric tight knots are the $8_{18}$ knot shown above left and the $10_{123}$ knot shown above right. Rounding the corners of these curves yields ropelength upper bounds of $\EightEighteenRRUB$ and $\TenOneHundredAndTwentyThreeRRUB$, respectively. Because their strut sets break into a particularly simple form (see pages~\pageref{EightEighteenPage} and~\pageref{TenOneHundredAndTwentyThreePage}), these knots may be better candidates for an explicit solution than seemingly simpler knots such as the trefoil.}
\label{symmetric}
\end{figure}

\section{Future directions}

Several directions for future research suggest themselves from these experiments. First, we note that while we have given finite strut sets for several polygonal knots and observed that they are close to the the 1-dimensional strut sets for the corresponding smooth tight configurations, we have not proved a theorem explaining how our polygonal strut sets converge to the strut sets of a critical polygon. We conjecture that this is part of a larger theorem which would show that if a family of polygonal ropelength critical configurations $\V_n$ converge to a $C^{1,1}$ curve $V$ then $V$ is ropelength critical in the sense of~\cite{cfksw2}, the strut sets of the $\V_n$ converge in Hausdorff distance to the self-contact set of $V$, and the kink sets of the $\V_n$ converge to the portion of $V$ at maximum curvature. 

There are several features of the tight knot data set that we have discovered that seem worthy of further investigation. Carlen, Smutny and Maddocks noted in~\cite{mglob} that curvature constraints seemed to be ``within a rather small tolerance of being active'' at several points on their numerical approximations of the tight trefoil and figure-eight knots. Baranska et al.\ provided numerically smoothed plots of the curvature of their approximately tight trefoil in~\cite{MR2471403} which appear to confirm this observation. 

\begin{figure}[ht]
\begin{overpic}[width=5in,viewport=10 50 350 200]{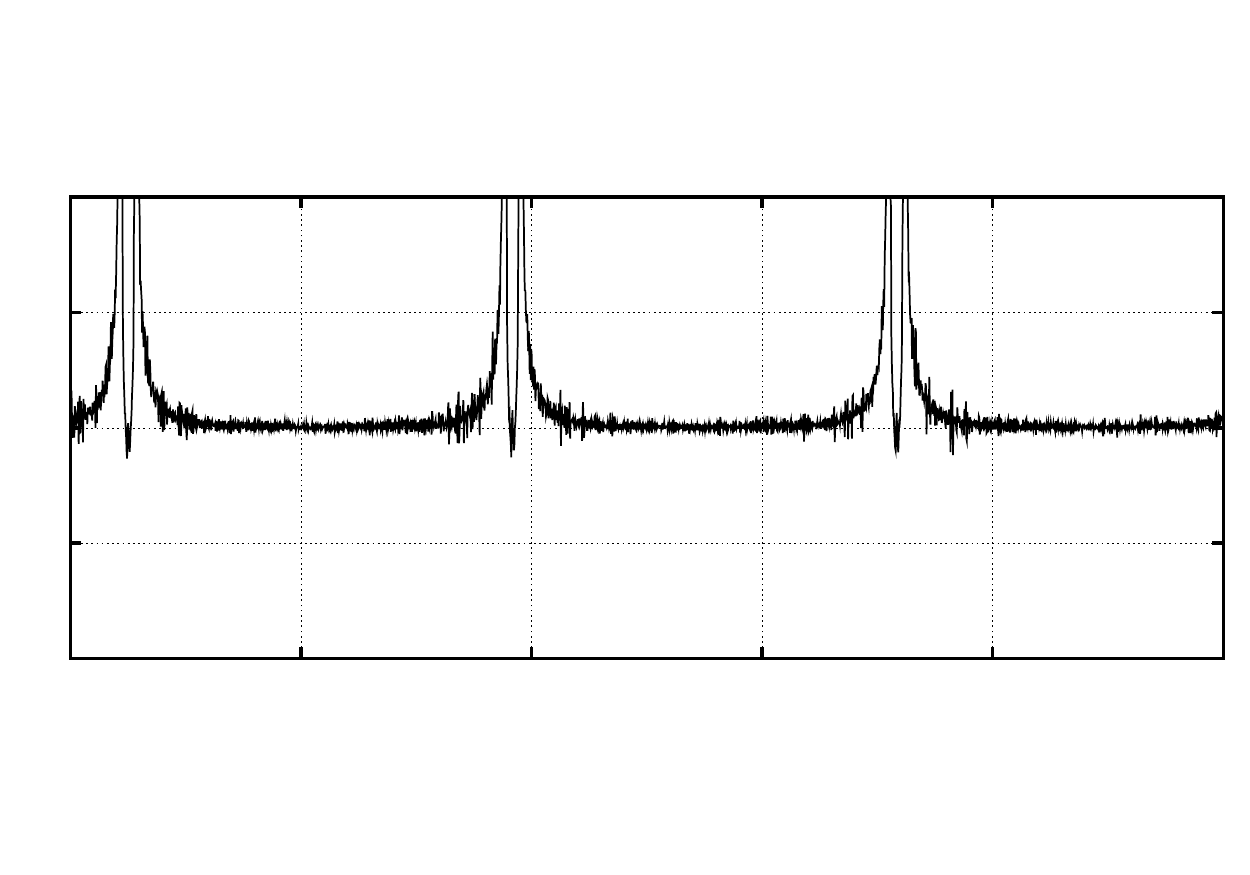}
\put(-2,12){$0.5$}
\put(-2,22){$1.0$}
\put(-2,32){$1.5$}
\put(-2,41.5){$2.0$}
\put(20,0){$3.27$}
\put(40,0){$6.55$}
\put(59,0){$9.82$}
\put(77,0){$13.10$}
\put(97,0){$16.37$}
\end{overpic}
\caption{This plot shows the computed $1/\MinRad$ values as a function of arclength along the polygon for a 2400 edge trefoil with thickness close to $1/2$, residual $0.0018$ and polygonal ropelength $32.743663$ (rounding out the corners as described above gives a smooth ropelength upper bound of $32.74352$ for this configuration). The value at each vertex is plotted above with no numerical smoothing. Though there is some noise in the portions of the plot where curvature is not constrained, the six kinked regions are clearly resolved. A total of $117$ vertices are involved in these regions.}
\label{fig:trefcurvature}
\end{figure}
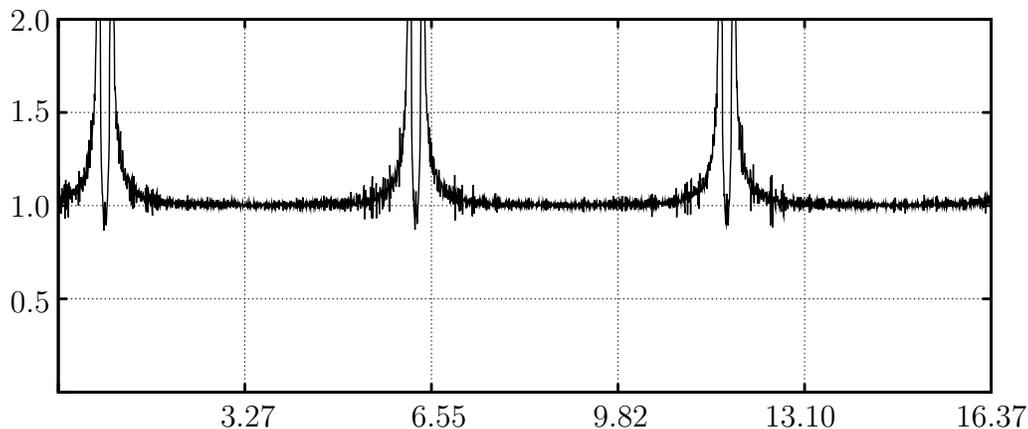

We have noticed the same phenomenon in our data sets. Our computation of the curvature for the trefoil appears in Figure~\ref{fig:trefcurvature}. In the \emph{Atlas of Tight Knots}, we highlight the active curvature constraints found by \ridgerunner\ as part of the minimization process by red lines on the plot of strut sets. These occur in \KinkCount\ of the \FileCount\ knots and links minimized. This provides suggestive numerical evidence that kinks are rather common in tight knots. We intend to provide better evidence for this conjecture in an upcoming publication.

Several authors have proved versions of the theorem that an interval of a tight knot with curvature less than the maximum allowed and no struts must be a straight line segment~\cite{MR2000c:57008,MR2033143,cfksw2,MR2355512}. We see this phenomenon \StraightCount\ times in the~\emph{Atlas}, for instance in the link $6^3_3$ on page $28$ of the~\emph{Atlas} (see also Figure~\ref{sixthreethree}), which appears to have three straight segments of length $2.1$, $1.14$, and $0.56$. We highlight these segments in blue on the plots in the \emph{Atlas}. These segments are almost as common as kinked regions in our data set, suggesting that they are generic features of tight configurations. Gonzalez has conjectured that every composite knot formed from joining a knot to its mirror image has a critical configuration with a pair of straight segments. We do not address this conjecture here since we only consider prime knots and links, but we do intend to compute approximately minimizing composite knots and links in a future publication.

\begin{figure}
\includegraphics[width=2in]{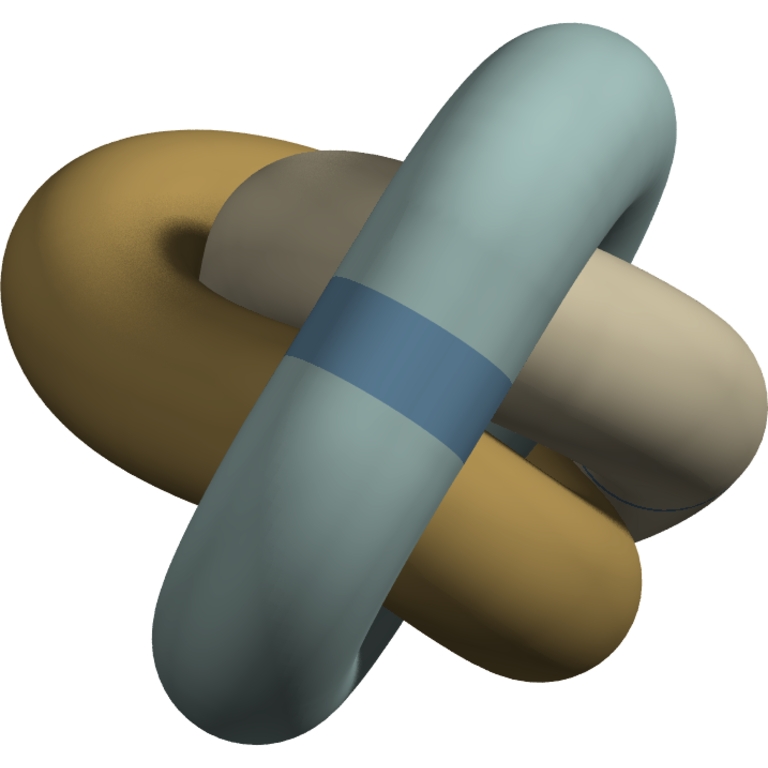} \hspace{0.25in}
\includegraphics[width=2in]{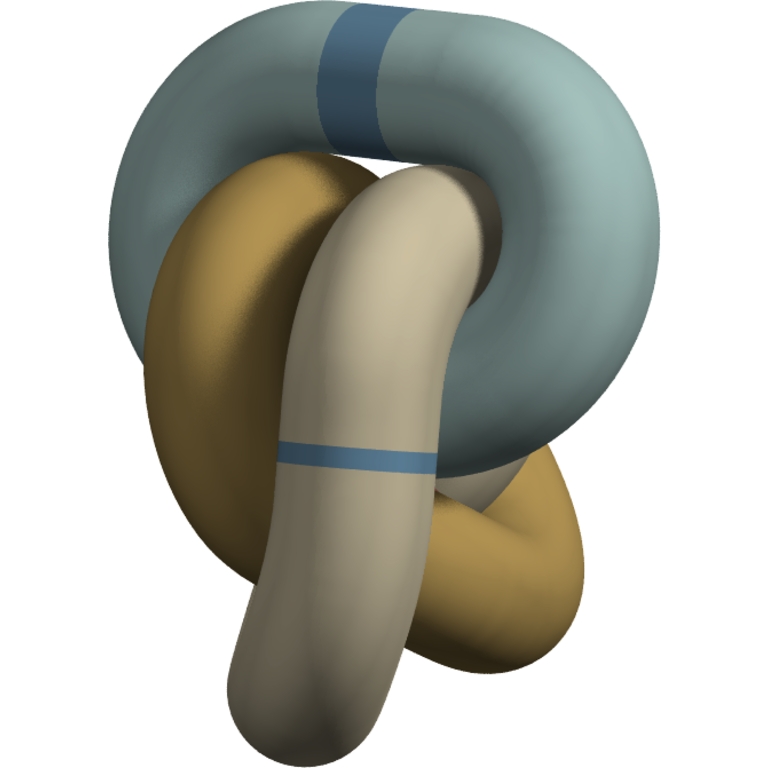} 
\caption{This figure shows two views of our computed tight configuration of the link $6^3_3$ (ropelength upper bound $\SixThreeThreeRRUB$). Straight segments on the blue and white components, which occur when these components lose contact with the other components of the link, are highlighted in darker blue.}
\label{sixthreethree}
\end{figure}

The paper~\cite{cfksw2} (as well as~\cite{MR89g:73037} under very different hypotheses) shows that a pair of arcs in a tight knot coparametrized by a single family of struts and having curvature less than the maximum bound form a standard double helix. As far as we can tell, this phenomenon only occurs a few times in the \emph{Atlas}, for instance in the $6^3_3$ link on page $28$, the $7^2_7$ link on page $43$, the $8_{19}$ knot on page $66$, and possibly in the $8^3_7$ link on page $91$. It would be interesting to look for more critical configurations with double-helix sections.


 We also contemplate further improvements to our numerical knot tightening methods. 
 The constrained gradient descent method presented in this paper is a
 significant improvement over simulated annealing --- in practice, it
 has proved to be an effective minimizer for both knots and links. But
 this is surely not the last word in numerical ropelength
 minimization. Our method is a member of the class of
 ``projected-gradient'' methods introduced by Rosen and Zoutendijk in
 the early 1960's~\cite{MR0135991,MR0122578}. These algorithms are
 subject to a number of well-known numerical problems, such as a
 tendency to ``wobble'' when confronted with a steep-sided valley and
 the problem of ``zigzagging'', which occurs when elements repeatedly
 enter and leave the strut and kink sets on successive minimization or
 error-correction steps. Our implementation seems to suffer from both
 these problems during some difficult minimizations. We have experimented with adding conjugate-gradient features to our existing code to solve these problems, but so far the results seem to yield only a slight improvement.
  
 For these reasons, more modern methods such as sequential quadratic programming (SQP) have become the norm~\cite{MR1867781}. Codes implementing these methods require the user to specify a set of constraint functions in advance. Unfortunately, in our formulation of
 the constraint thickness an $n$-vertex polygon has $O(n^2)$
 self-distance constraints and $O(n)$ turning angle or $\MinRad$
 constraints. For a typical polygon with $10^3$ vertices, this would
 mean a set of $10^6$ constraints --- too many to be practical. However,
 if we know approximately which self-distance constraints will be
 active in the final configuration, we can ignore constraints that
 we expect to be inactive, resulting in a reduced constraint set of
 size $O(n)$. Our approximately minimized polygons provide exactly
 this information. For this reason we imagine an important use
 of our data will be in formulating input problems for a future SQP-based
 knot-minimizer. Our polygons are already serving as input for the
 biarc-based annealer of Carlen, Smutny, and Maddocks~\cite{mglob}.

 While our data set is detailed and suggestive, solving explicitly for
 the structure of ropelength minimizing (smooth) knots and links is
 likely to require even better data. Cantarella et al.~\cite{cfksw2}
 have shown that a critical shape for the simple clasp formed when
 ropes pass over one another at right angles contains tiny straight
 segments of length a few thousandths of the total length of the
 curves. Resolving these features will require converged runs for
 polygonal ropelength minimizers with tens of thousands of vertices,
 an ambitious goal that will keep this area of experimental
 mathematics active for some time to come.

 \section{Acknowledgements}
 
 The authors would like to mention the hard work of Sivan Toledo,
 whose \TAUCS\ library made \tlsqr\ and \tsnnls\ possible. Our code was made much faster by Toledo's carefully written supernodal
 multifrontal cholesky factorization code. We are similarly indebted
 to Portugal, Judice and Vicente for developing the block principal
 pivoting algorithm. Many colleagues provided helpful conversations
 and insights about these and similar problems, including Joe Fu, Rob
 Kusner, John Sullivan, Piotr Pieranski, John Maddocks and Andrzej
 Stasiak.  The authors would also like to acknowledge the support of
 the National Science Foundation through the University of Georgia
 VIGRE grant (DMS-00-89927), DMS-02-04826 (to Cantarella and Fu), and DMS-08-10415 (to Rawdon).
 
\bibliography{rr05siggraph,drl,cantarella,tsnnls} 
\newpage

\appendix
\section{Ropelength Data}
\label{ropdata} 

The pages that follow contain three sets of tables of ropelength data. The first set, Tables~\ref{ByKnotTableone}-\ref{ByKnotTablethree} on pages~\pageref{ByKnotTableone}--\pageref{ByKnotTablethree}, show the polygonal ropelength ($\PRop$) and ropelength upper bounds ($\Rop$) that we have obtained for each of the knot types that we have considered. The knots and links are organized according to their position in Rolfsen's table, with the link $X^y_z$ being the $z$-th example of a prime $X$-crossing link of $y$ components in the table. We have identified the two ``Perko pair'' knots $10_{161}$ and $10_{162}$ and renumbered the subsequent knots accordingly, so there are only 165 ten-crossing knots in our results.

The second set, Tables~\ref{ByRopTableone} and~\ref{ByRopTabletwo} on pages~\pageref{ByRopTableone}--\pageref{ByRopTabletwo}, show the same knot and link types ordered by ropelength upper bound. These tables are to be read down each column from the top left to the bottom right. We can see that this order is quite different from the one in Rolfsen's table with (for instance) the 2-component link $7^2_7$ occurring before any 6 or 7 crossing knot and the $10_{124}$ knot occurring before many 8 and 9 crossing links. 

The third set of tables, Tables~\ref{ResidualTableone}--\ref{ResidualTablethree} on pages~\pageref{ResidualTableone}--\pageref{ResidualTablethree} give the residual of each of our computed configurations. The low residuals show that they are close to critical in the sense of Theorem~\ref{thm:constrainedcritical}. We include this data as measure of the relative quality of each of our minimized configurations.

On pages~\pageref{EightEighteenPage}--\pageref{TenOneHundredAndTwentyThreePage} are reproductions of the pages from the~\emph{Atlas of Tight Knots} for the approximately tight $8_{18}$ and $10_{123}$ knots . On the top left of each page are three views of the tight configurations, with kinked regions highlighted in red. On the top right is a plot of the self-contact map of the configuration. Each of these plots consists of a triangular region with the hypotenuse labeled with arclength values on the knot. A green box is plotted at $(s,t)$ on the plot if there is a strut connecting $L(s)$ and $L(t)$. Below the graph appears a plot of $1/\MinRad$ for the polygon (to the same scale). Kinked regions of maximum curvature are plotted in red on the graph. Each such region has a key on the right-hand side of the plot showing the arclength positions of the start and end of the kink (in order to give a sense of the relative scale of the kinked region). At the bottom of the page is a line of data giving the polygonal ropelength $\PRop$ (as measured by \texttt{octrope}), ropelength upper bound $\Rop$ (from \texttt{roundout\_rl}), filename, number of vertices and struts, maximum and minimum curvature values and number of kinked regions. The last entry shows the total arclength of straight regions in the curves (0 for these two knots, but nonzero for many knots and links in the~\emph{Atlas}).

\footnotesize
\input{KnotTable.tex}
\input{KnotTableByRop.tex}
\input{ResidualTable.tex}

\newpage
\label{EightEighteenPage}
\index{8@8 crossing links!1@with 1 component!50@$8_{18}$}
\begin{tabular}{ll} \includegraphics[width=1.9in]{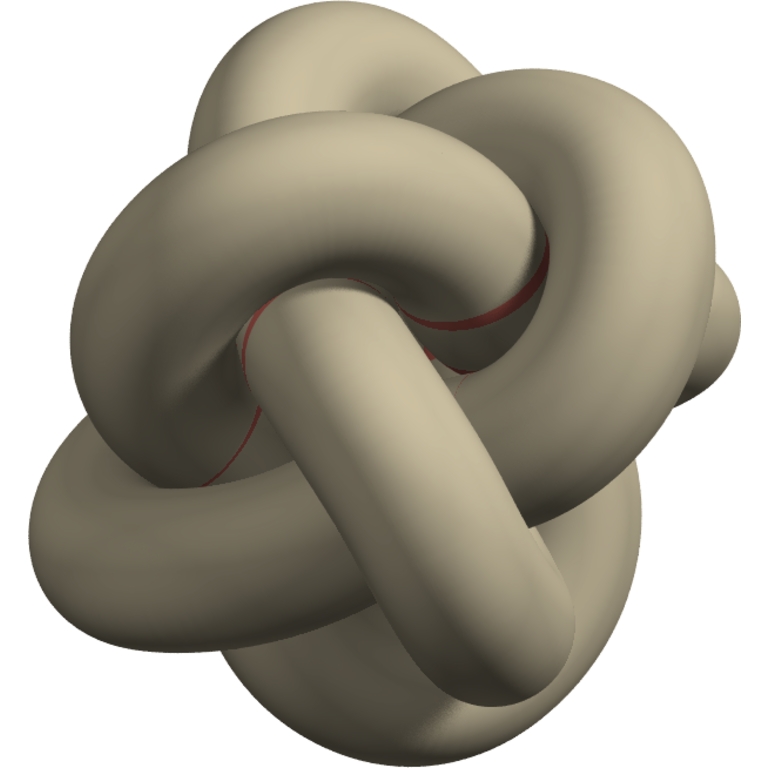} & \includegraphics[width=1.9in]{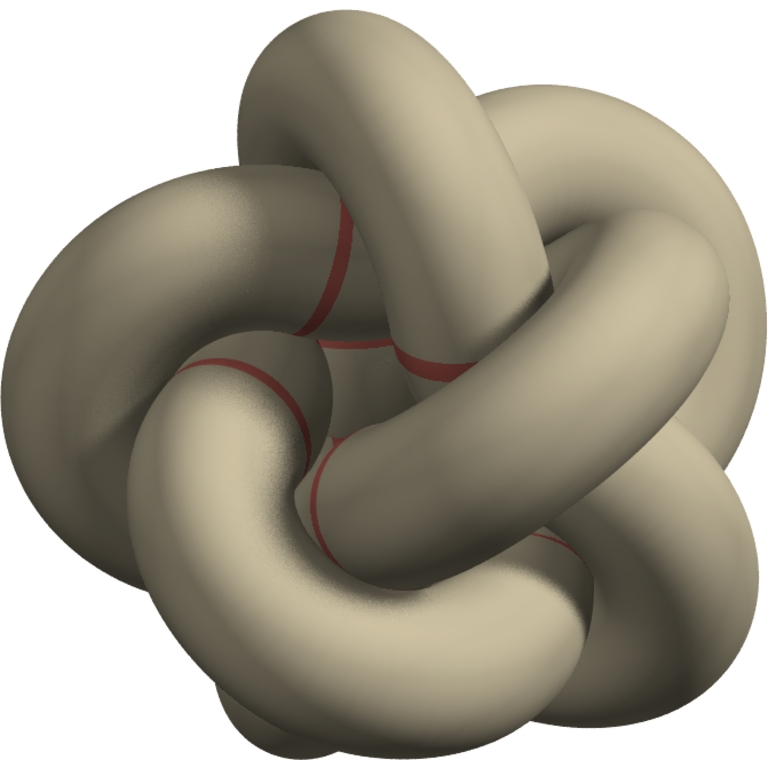} \\
\includegraphics[width=1.9in]{figs/{kl_8_18_hrbanff.a2}.jpg} & \end{tabular} 

\vspace{-2.8in}

\hspace{0.6in}\includegraphics[width=4.3in,viewport=150 145 430 560]{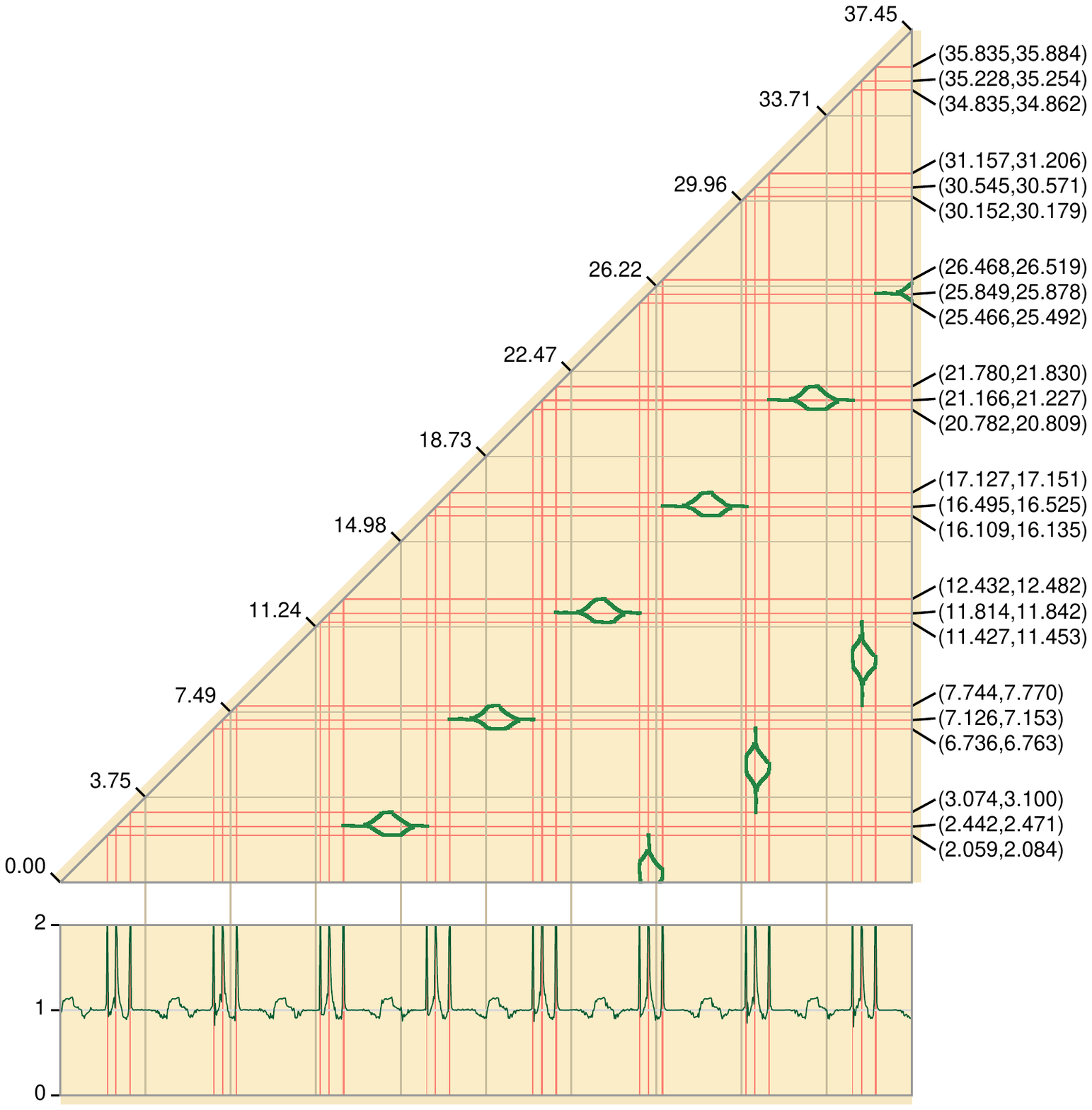} \vfill

\begin{center} \begin{tabular}{lllllllll} \toprule
Link & $\PRop$ & $\Rop$ & Filename & Verts & Struts & $\kappa$ range & Kink & Straight \\ \midrule
$8_{18}$ & $74.9114$ & $74.9063$ & \verb!kl_8_18_hrbanff.vect! & $1199$ & $5591$ & $[0.802748,2.00005]$ & $24$ & $$ \\
\bottomrule
\end{tabular} \end{center}

\newpage
\label{TenOneHundredAndTwentyThreePage}
\index{10@10 crossing links!1@with 1 component!337@$10_{123}$}
\begin{tabular}{ll} \includegraphics[width=1.9in]{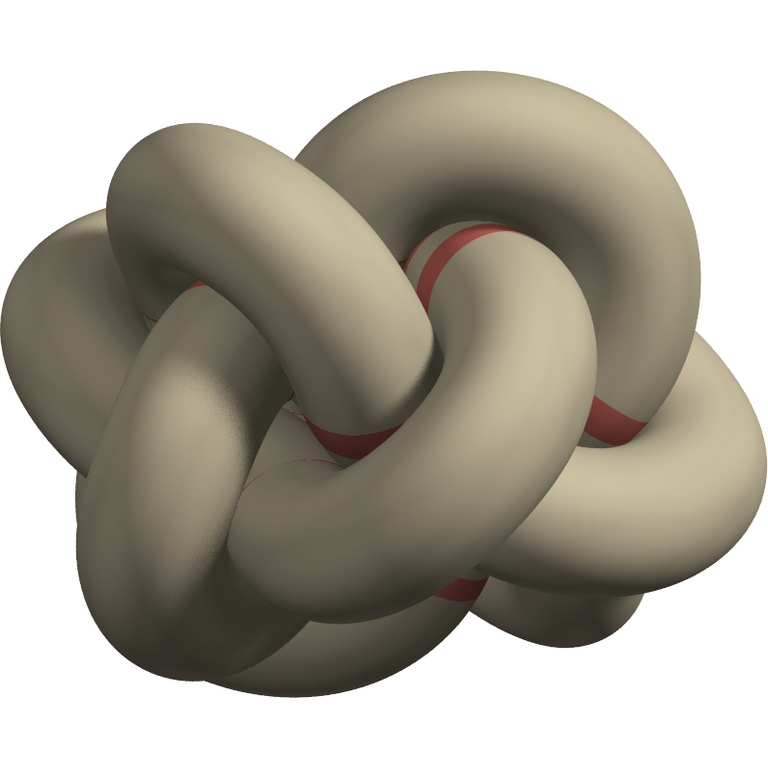} & \includegraphics[width=1.9in]{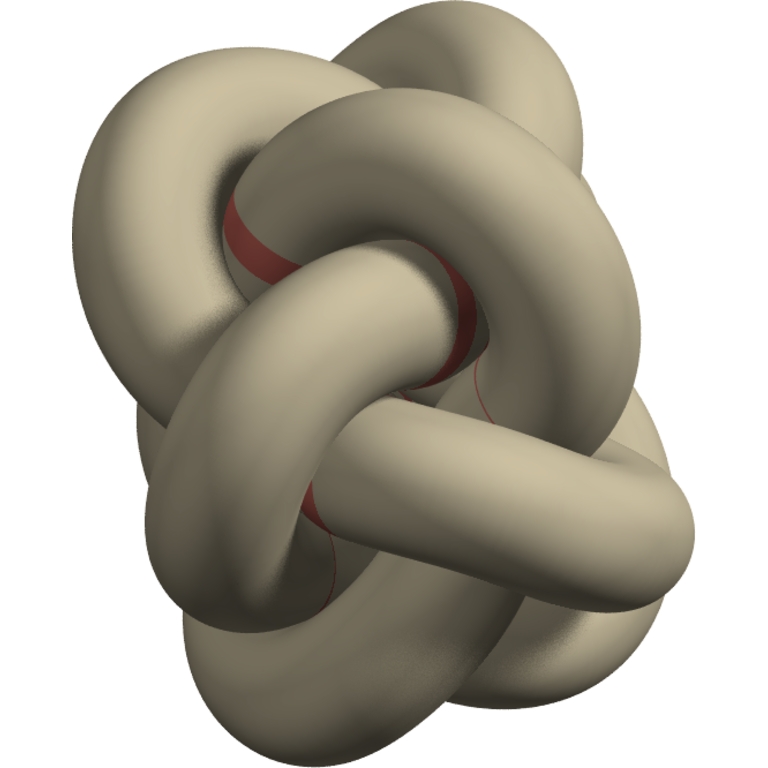} \\
\includegraphics[width=1.9in]{figs/{kl_10_123_handcrafted.a2}.jpg} & \end{tabular} 

\vspace{-2.8in}

\hspace{0.6in}\includegraphics[width=4.3in,viewport=150 145 430 560]{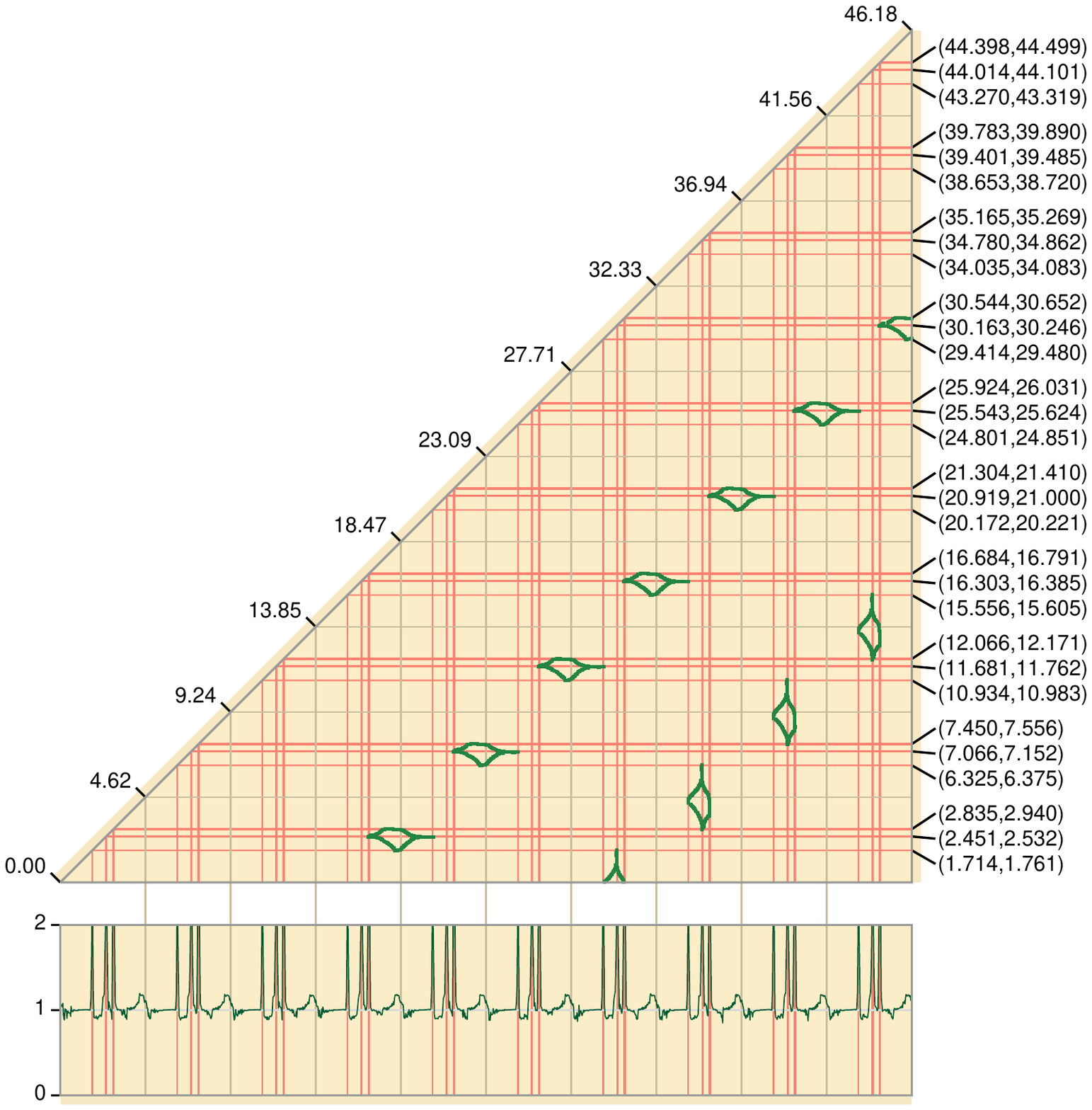} \vfill

\begin{center} \begin{tabular}{lllllllll} \toprule
Link & $\PRop$ & $\Rop$ & Filename & Verts & Struts & $\kappa$ range & Kink & Straight \\ \midrule
$10_{123}$ & $92.3646$ & $92.3565$ & \verb!kl_10_123_handcrafted.vect! & $1498$ & $7189$ & $[0.84917,2.00008]$ & $30$ & $$ \\
\bottomrule
\end{tabular} \end{center}

\end{document}

%% file: latex_knot_prev.tex
\newcommand{\FourOnePrev}{42.1158845}

\newcommand{\NineTwentyPrev}{87.31}

%% file: latex_knot_prevcite.tex
\newcommand{\FourOnePrevCite}{mglob}

\newcommand{\NineTwentyPrevCite}{MR2034393}

%% file: KnotTags.tex
\newcommand{\TenOneHundredAndTwentyThreeRRUB}{92.3565}

\newcommand{\TwoTwoOneRRUB}{25.1334}

\newcommand{\FourOneRRUB}{42.0887}

\newcommand{\SixThreeTwoRRUB}{58.0070}

\newcommand{\SixThreeThreeRRUB}{50.5539}

\newcommand{\EightEighteenRRUB}{74.9063}

\newcommand{\NineTwentyRRUB}{80.2219}

%% file: fileCounts.tex
\newcommand{\FileCount}{379}

\newcommand{\StraightCount}{338}
\newcommand{\KinkCount}{372}
\newcommand{\DoubleOOneCount}{375}

\newcommand{\AverageResidual}{0.00299}
\newcommand{\TripleOOneCount}{202}
\newcommand{\AverageStruts}{919}
\newcommand{\AverageMrlocs}{19}
\newcommand{\AverageEdges}{658}
\newcommand{\AverageEdgelength}{0.061}
\newcommand{\MatrixDims}{938 \times 1974}
\newcommand{\MatrixEnts}{1851612}
\newcommand{\MatrixNNZ}{11199}
\newcommand{\MatrixSparsity}{99.4\%}

\newcommand{\AverageWin}{3.619}
\newcommand{\WonClaimCases}{316}
\newcommand{\LostClaimCases}{36}
\newcommand{\AverageClaimWin}{1.075}

%% file: BestWorstTable.tex
\begin{tabular}[b]{lcccc}
\toprule
$\Cr$ & $\Rop$ & Links \\
\midrule
$3$   & $32.74$ & $3_1$ & & \\
$4$ & $ [40.0122,42.0887] $ & $ 4^{2}_{1} $, $ 4_{1} $ \\ 
$5$ & $ [47.2016,49.7716] $ & $ 5_{1} $, $ 5^{2}_{1} $ \\ 
$6$ & $ [50.5539,58.1013] $ & $ 6^{3}_{3} $, $ 6^{2}_{3} $ \\ 
$7$ & $ [55.5095,66.3147] $ & $ 7^{2}_{7} $, $ 7^{2}_{6} $ \\ 
$8$ & $ [60.5754,75.2592] $ & $ 8^{3}_{7} $, $ 8^{4}_{1} $ \\ 
$9$ & $ [66.0311,83.6092] $ & $ 9^{2}_{49} $, $ 9^{2}_{42} $ \\ 
$10$ & $ [71.0739,92.3565] $ & $ 10_{124} $, $ 10_{123} $ \\ 
\bottomrule
\end{tabular}

%% file: KnotTable.tex
\begin{table}[h]
\caption{\label{ByKnotTableone}Part 1 of Ropelengths of Tight Knots and Links by Knot Type}
\begin{tabular}[t]{lll} \toprule
Link & $\PRop$ & $\Rop$\\ \midrule 
$2^{2}_{1}$ & $25.1415$ & $25.1334$ \\
\addlinespace[0.45em]\cmidrule(r{0.5em}l{0.5em}){1-3}\addlinespace[0.45em]
$3_{1}$ & $32.7437$ & $32.7436$ \\
\addlinespace[0.45em]\cmidrule(r{0.5em}l{0.5em}){1-3}\addlinespace[0.45em]
$4_{1}$ & $42.0971$ & $42.0887$ \\
\addlinespace[0.45em]\cmidrule(r{1.5em}l{1.5em}){1-3}\addlinespace[0.45em]
$4^{2}_{1}$ & $40.0203$ & $40.0122$ \\
\addlinespace[0.45em]\cmidrule(r{0.5em}l{0.5em}){1-3}\addlinespace[0.45em]
$5_{1}$ & $47.2149$ & $47.2016$ \\
$5_{2}$ & $49.4820$ & $49.4701$ \\
\addlinespace[0.45em]\cmidrule(r{1.5em}l{1.5em}){1-3}\addlinespace[0.45em]
$5^{2}_{1}$ & $49.7864$ & $49.7716$ \\
\addlinespace[0.45em]\cmidrule(r{0.5em}l{0.5em}){1-3}\addlinespace[0.45em]
$6_{1}$ & $56.7178$ & $56.7058$ \\
$6_{2}$ & $57.0381$ & $57.0235$ \\
$6_{3}$ & $57.8531$ & $57.8392$ \\
\addlinespace[0.45em]\cmidrule(r{1.5em}l{1.5em}){1-3}\addlinespace[0.45em]
$6^{2}_{1}$ & $54.3919$ & $54.3768$ \\
$6^{2}_{2}$ & $56.7087$ & $56.7000$ \\
$6^{2}_{3}$ & $58.1142$ & $58.1013$ \\
\addlinespace[0.45em]\cmidrule(r{1.5em}l{1.5em}){1-3}\addlinespace[0.45em]
$6^{3}_{1}$ & $57.8286$ & $57.8141$ \\
$6^{3}_{2}$ & $58.0112$ & $58.0070$ \\
$6^{3}_{3}$ & $50.5602$ & $50.5539$ \\
\addlinespace[0.45em]\cmidrule(r{0.5em}l{0.5em}){1-3}\addlinespace[0.45em]
$7_{1}$ & $61.4234$ & $61.4067$ \\
$7_{2}$ & $63.8684$ & $63.8556$ \\
$7_{3}$ & $63.9430$ & $63.9285$ \\
$7_{4}$ & $64.2836$ & $64.2687$ \\
$7_{5}$ & $65.2705$ & $65.2560$ \\
$7_{6}$ & $65.7068$ & $65.6924$ \\
$7_{7}$ & $65.6235$ & $65.6086$ \\
\addlinespace[0.45em]\cmidrule(r{1.5em}l{1.5em}){1-3}\addlinespace[0.45em]
$7^{2}_{1}$ & $64.2484$ & $64.2345$ \\
$7^{2}_{2}$ & $65.0363$ & $65.0204$ \\
$7^{2}_{3}$ & $65.3414$ & $65.3257$ \\
$7^{2}_{4}$ & $65.0759$ & $65.0602$ \\
$7^{2}_{5}$ & $66.2068$ & $66.1915$ \\
$7^{2}_{6}$ & $66.3281$ & $66.3147$ \\
$7^{2}_{7}$ & $55.5177$ & $55.5095$ \\
$7^{2}_{8}$ & $57.7714$ & $57.7631$ \\
\addlinespace[0.45em]\cmidrule(r{1.5em}l{1.5em}){1-3}\addlinespace[0.45em]
$7^{3}_{1}$ & $65.8157$ & $65.8062$ \\
\addlinespace[0.45em]\cmidrule(r{0.5em}l{0.5em}){1-3}\addlinespace[0.45em]
$8_{1}$ & $70.9833$ & $70.9669$ \\
$8_{2}$ & $71.4141$ & $71.3985$ \\
\bottomrule
\end{tabular}
\hspace{0.25in}
\begin{tabular}[t]{lll} \toprule
Link & $\PRop$ & $\Rop$\\ \midrule 
$8_{2}$ & $71.4141$ & $71.3985$ \\
$8_{3}$ & $71.1736$ & $71.1575$ \\
$8_{4}$ & $71.4872$ & $71.4704$ \\
$8_{5}$ & $72.1519$ & $72.1344$ \\
$8_{6}$ & $72.4903$ & $72.4725$ \\
$8_{7}$ & $72.2292$ & $72.2137$ \\
$8_{8}$ & $72.7438$ & $72.7241$ \\
$8_{9}$ & $72.4568$ & $72.4399$ \\
$8_{10}$ & $72.9580$ & $72.9379$ \\
$8_{11}$ & $72.9110$ & $72.8966$ \\
$8_{12}$ & $73.9707$ & $73.9518$ \\
$8_{13}$ & $72.8194$ & $72.8000$ \\
$8_{14}$ & $73.7784$ & $73.7612$ \\
$8_{15}$ & $73.9076$ & $73.8977$ \\
$8_{16}$ & $73.5207$ & $73.5054$ \\
$8_{17}$ & $74.5075$ & $74.4912$ \\
$8_{18}$ & $74.9114$ & $74.9063$ \\
$8_{19}$ & $60.9970$ & $60.9858$ \\
$8_{20}$ & $63.1066$ & $63.0929$ \\
$8_{21}$ & $65.5387$ & $65.5248$ \\
\addlinespace[0.45em]\cmidrule(r{1.5em}l{1.5em}){1-3}\addlinespace[0.45em]
$8^{2}_{1}$ & $68.4208$ & $68.4045$ \\
$8^{2}_{2}$ & $71.0493$ & $71.0311$ \\
$8^{2}_{3}$ & $72.7292$ & $72.7133$ \\
$8^{2}_{4}$ & $72.5995$ & $72.5855$ \\
$8^{2}_{5}$ & $73.9503$ & $73.9331$ \\
$8^{2}_{6}$ & $73.2133$ & $73.1955$ \\
$8^{2}_{7}$ & $74.3917$ & $74.3752$ \\
$8^{2}_{8}$ & $73.7714$ & $73.7540$ \\
$8^{2}_{9}$ & $73.2196$ & $73.2038$ \\
$8^{2}_{10}$ & $73.6729$ & $73.6548$ \\
$8^{2}_{11}$ & $72.9786$ & $72.9608$ \\
$8^{2}_{12}$ & $73.8018$ & $73.7846$ \\
$8^{2}_{13}$ & $74.1522$ & $74.1369$ \\
$8^{2}_{14}$ & $73.6878$ & $73.6695$ \\
$8^{2}_{15}$ & $64.3105$ & $64.2996$ \\
$8^{2}_{16}$ & $66.8148$ & $66.8046$ \\
\addlinespace[0.45em]\cmidrule(r{1.5em}l{1.5em}){1-3}\addlinespace[0.45em]
$8^{3}_{1}$ & $72.2765$ & $72.2603$ \\
$8^{3}_{2}$ & $72.9357$ & $72.9181$ \\
$8^{3}_{3}$ & $74.8824$ & $74.8656$ \\
$8^{3}_{4}$ & $75.0026$ & $74.9866$ \\
$8^{3}_{5}$ & $73.4072$ & $73.3932$ \\
$8^{3}_{6}$ & $74.7320$ & $74.7159$ \\
$8^{3}_{7}$ & $60.5897$ & $60.5754$ \\
$8^{3}_{8}$ & $65.0195$ & $65.0042$ \\
\bottomrule
\end{tabular}
\hspace{0.25in}
\begin{tabular}[t]{lll} \toprule
Link & $\PRop$ & $\Rop$\\ \midrule 
$8^{3}_{8}$ & $65.0195$ & $65.0042$ \\
$8^{3}_{9}$ & $66.7076$ & $66.6936$ \\
$8^{3}_{10}$ & $68.4580$ & $68.4503$ \\
\addlinespace[0.45em]\cmidrule(r{1.5em}l{1.5em}){1-3}\addlinespace[0.45em]
$8^{4}_{1}$ & $75.2748$ & $75.2592$ \\
$8^{4}_{2}$ & $67.4087$ & $67.3937$ \\
$8^{4}_{3}$ & $66.2969$ & $66.2865$ \\
\addlinespace[0.45em]\cmidrule(r{0.5em}l{0.5em}){1-3}\addlinespace[0.45em]
$9_{1}$ & $75.5663$ & $75.5461$ \\
$9_{2}$ & $78.1231$ & $78.1066$ \\
$9_{3}$ & $78.2040$ & $78.1892$ \\
$9_{4}$ & $78.2793$ & $78.2665$ \\
$9_{5}$ & $78.6615$ & $78.6447$ \\
$9_{6}$ & $79.5802$ & $79.5597$ \\
$9_{7}$ & $79.6924$ & $79.6731$ \\
$9_{8}$ & $80.0276$ & $80.0080$ \\
$9_{9}$ & $79.8965$ & $79.8778$ \\
$9_{10}$ & $79.8009$ & $79.7855$ \\
$9_{11}$ & $80.1355$ & $80.1180$ \\
$9_{12}$ & $80.0997$ & $80.0834$ \\
$9_{13}$ & $80.2657$ & $80.2498$ \\
$9_{14}$ & $80.0193$ & $80.0001$ \\
$9_{15}$ & $80.8941$ & $80.8725$ \\
$9_{16}$ & $80.1334$ & $80.1143$ \\
$9_{17}$ & $80.4718$ & $80.4530$ \\
$9_{18}$ & $81.5816$ & $81.5673$ \\
$9_{19}$ & $80.9196$ & $80.9004$ \\
$9_{20}$ & $80.2421$ & $80.2219$ \\
$9_{21}$ & $81.1083$ & $81.0920$ \\
$9_{22}$ & $81.0587$ & $81.0390$ \\
$9_{23}$ & $81.2922$ & $81.2733$ \\
$9_{24}$ & $80.9626$ & $80.9451$ \\
$9_{25}$ & $81.1348$ & $81.1198$ \\
$9_{26}$ & $80.9241$ & $80.9053$ \\
$9_{27}$ & $81.1838$ & $81.1813$ \\
$9_{28}$ & $81.0878$ & $81.1352$ \\
$9_{29}$ & $81.2019$ & $81.1821$ \\
$9_{30}$ & $81.4811$ & $81.4883$ \\
$9_{31}$ & $81.6751$ & $81.6581$ \\
$9_{32}$ & $81.5343$ & $81.5175$ \\
$9_{33}$ & $82.7691$ & $82.7541$ \\
$9_{34}$ & $82.1884$ & $82.1706$ \\
$9_{35}$ & $79.2390$ & $79.2165$ \\
$9_{36}$ & $80.2275$ & $80.2064$ \\
$9_{37}$ & $81.1744$ & $81.1674$ \\
$9_{38}$ & $81.7858$ & $81.7697$ \\
\bottomrule
\end{tabular}
\hspace{0.25in}
\end{table}
\begin{table}[h]
\caption{\label{ByKnotTabletwo}Part 2 of Ropelengths of Tight Knots and Links by Knot Type}
\begin{tabular}[t]{lll} \toprule
Link & $\PRop$ & $\Rop$\\ \midrule 
$9_{38}$ & $81.7858$ & $81.7697$ \\
$9_{39}$ & $81.8439$ & $81.8264$ \\
$9_{40}$ & $81.6652$ & $81.6474$ \\
$9_{41}$ & $81.3687$ & $81.3540$ \\
$9_{42}$ & $69.4867$ & $69.4756$ \\
$9_{43}$ & $71.5050$ & $71.4901$ \\
$9_{44}$ & $71.5587$ & $71.5427$ \\
$9_{45}$ & $74.0861$ & $74.0761$ \\
$9_{46}$ & $68.6330$ & $68.6169$ \\
$9_{47}$ & $74.8935$ & $74.8785$ \\
$9_{48}$ & $74.0317$ & $74.0228$ \\
$9_{49}$ & $73.9403$ & $73.9286$ \\
\addlinespace[0.45em]\cmidrule(r{1.5em}l{1.5em}){1-3}\addlinespace[0.45em]
$9^{2}_{1}$ & $78.6049$ & $78.5862$ \\
$9^{2}_{2}$ & $79.5287$ & $79.5152$ \\
$9^{2}_{3}$ & $79.9495$ & $79.9312$ \\
$9^{2}_{4}$ & $78.6961$ & $78.6764$ \\
$9^{2}_{5}$ & $79.6569$ & $79.6384$ \\
$9^{2}_{6}$ & $80.1200$ & $80.1017$ \\
$9^{2}_{7}$ & $81.1437$ & $81.1261$ \\
$9^{2}_{8}$ & $80.9964$ & $80.9766$ \\
$9^{2}_{9}$ & $80.3174$ & $80.2999$ \\
$9^{2}_{10}$ & $80.3218$ & $80.3036$ \\
$9^{2}_{11}$ & $82.0329$ & $82.0140$ \\
$9^{2}_{12}$ & $81.9602$ & $81.9414$ \\
$9^{2}_{13}$ & $79.3468$ & $79.3280$ \\
$9^{2}_{14}$ & $80.7276$ & $80.7104$ \\
$9^{2}_{15}$ & $80.5659$ & $80.5458$ \\
$9^{2}_{16}$ & $81.3758$ & $81.3565$ \\
$9^{2}_{17}$ & $80.3223$ & $80.3022$ \\
$9^{2}_{18}$ & $81.7563$ & $81.7461$ \\
$9^{2}_{19}$ & $79.4706$ & $79.4491$ \\
$9^{2}_{20}$ & $80.1357$ & $80.1147$ \\
$9^{2}_{21}$ & $80.6010$ & $80.5824$ \\
$9^{2}_{22}$ & $81.0964$ & $81.0794$ \\
$9^{2}_{23}$ & $80.2592$ & $80.2379$ \\
$9^{2}_{24}$ & $81.7913$ & $81.7691$ \\
$9^{2}_{25}$ & $81.7810$ & $81.7630$ \\
$9^{2}_{26}$ & $82.1031$ & $82.0859$ \\
$9^{2}_{27}$ & $81.0288$ & $81.0141$ \\
$9^{2}_{28}$ & $81.3352$ & $81.3222$ \\
$9^{2}_{29}$ & $82.1606$ & $82.1445$ \\
$9^{2}_{30}$ & $82.2155$ & $82.1987$ \\
$9^{2}_{31}$ & $80.5732$ & $80.5561$ \\
$9^{2}_{32}$ & $81.4151$ & $81.3990$ \\
$9^{2}_{33}$ & $82.1790$ & $82.1612$ \\
\bottomrule
\end{tabular}
\hspace{0.25in}
\begin{tabular}[t]{lll} \toprule
Link & $\PRop$ & $\Rop$\\ \midrule 
$9^{2}_{33}$ & $82.1790$ & $82.1612$ \\
$9^{2}_{34}$ & $81.8490$ & $81.8320$ \\
$9^{2}_{35}$ & $81.2508$ & $81.2318$ \\
$9^{2}_{36}$ & $80.7066$ & $80.6866$ \\
$9^{2}_{37}$ & $81.9102$ & $81.8927$ \\
$9^{2}_{38}$ & $82.6750$ & $82.6561$ \\
$9^{2}_{39}$ & $81.8972$ & $81.8758$ \\
$9^{2}_{40}$ & $81.9680$ & $81.9460$ \\
$9^{2}_{41}$ & $83.6038$ & $83.5878$ \\
$9^{2}_{42}$ & $83.6304$ & $83.6092$ \\
$9^{2}_{43}$ & $66.2549$ & $66.2398$ \\
$9^{2}_{44}$ & $72.2072$ & $72.1896$ \\
$9^{2}_{45}$ & $71.0815$ & $71.0726$ \\
$9^{2}_{46}$ & $73.8347$ & $73.8215$ \\
$9^{2}_{47}$ & $69.9130$ & $69.8983$ \\
$9^{2}_{48}$ & $73.6563$ & $73.6426$ \\
$9^{2}_{49}$ & $66.0444$ & $66.0311$ \\
$9^{2}_{50}$ & $69.3353$ & $69.3284$ \\
$9^{2}_{51}$ & $70.5455$ & $70.5299$ \\
$9^{2}_{52}$ & $72.8271$ & $72.8106$ \\
$9^{2}_{53}$ & $68.0154$ & $68.0082$ \\
$9^{2}_{54}$ & $71.0240$ & $71.0089$ \\
$9^{2}_{55}$ & $73.8129$ & $73.7998$ \\
$9^{2}_{56}$ & $72.9013$ & $72.8833$ \\
$9^{2}_{57}$ & $72.2115$ & $72.1922$ \\
$9^{2}_{58}$ & $74.1685$ & $74.1499$ \\
$9^{2}_{59}$ & $72.3285$ & $72.3130$ \\
$9^{2}_{60}$ & $73.5589$ & $73.5442$ \\
$9^{2}_{61}$ & $69.3751$ & $69.3636$ \\
\addlinespace[0.45em]\cmidrule(r{1.5em}l{1.5em}){1-3}\addlinespace[0.45em]
$9^{3}_{1}$ & $81.1522$ & $81.1333$ \\
$9^{3}_{2}$ & $81.7304$ & $81.7190$ \\
$9^{3}_{3}$ & $82.2498$ & $82.2346$ \\
$9^{3}_{4}$ & $82.5202$ & $82.5029$ \\
$9^{3}_{5}$ & $80.2664$ & $80.2456$ \\
$9^{3}_{6}$ & $80.9434$ & $80.9258$ \\
$9^{3}_{7}$ & $82.0540$ & $82.0378$ \\
$9^{3}_{8}$ & $81.1278$ & $81.1107$ \\
$9^{3}_{9}$ & $81.5469$ & $81.5295$ \\
$9^{3}_{10}$ & $82.3146$ & $82.2964$ \\
$9^{3}_{11}$ & $82.0023$ & $81.9867$ \\
$9^{3}_{12}$ & $82.4811$ & $82.4608$ \\
$9^{3}_{13}$ & $72.2098$ & $72.2009$ \\
$9^{3}_{14}$ & $74.4319$ & $74.4205$ \\
$9^{3}_{15}$ & $74.2998$ & $74.2810$ \\
$9^{3}_{16}$ & $75.0113$ & $75.0003$ \\
\bottomrule
\end{tabular}
\hspace{0.25in}
\begin{tabular}[t]{lll} \toprule
Link & $\PRop$ & $\Rop$\\ \midrule 
$9^{3}_{16}$ & $75.0113$ & $75.0003$ \\
$9^{3}_{17}$ & $74.1280$ & $74.1159$ \\
$9^{3}_{18}$ & $72.4529$ & $72.4382$ \\
$9^{3}_{19}$ & $72.6412$ & $72.6275$ \\
$9^{3}_{20}$ & $75.9995$ & $75.9845$ \\
$9^{3}_{21}$ & $74.8967$ & $74.8908$ \\
\addlinespace[0.45em]\cmidrule(r{1.5em}l{1.5em}){1-3}\addlinespace[0.45em]
$9^{4}_{1}$ & $81.6096$ & $81.5927$ \\
\addlinespace[0.45em]\cmidrule(r{0.5em}l{0.5em}){1-3}\addlinespace[0.45em]
$10_{1}$ & $85.1146$ & $85.0947$ \\
$10_{2}$ & $85.6050$ & $85.5850$ \\
$10_{3}$ & $85.4483$ & $85.4278$ \\
$10_{4}$ & $85.8181$ & $85.7974$ \\
$10_{5}$ & $86.4952$ & $86.4741$ \\
$10_{6}$ & $86.8353$ & $86.8125$ \\
$10_{7}$ & $87.2979$ & $87.2775$ \\
$10_{8}$ & $85.8620$ & $85.8428$ \\
$10_{9}$ & $86.8410$ & $86.8222$ \\
$10_{10}$ & $87.2060$ & $87.1870$ \\
$10_{11}$ & $86.9848$ & $86.9630$ \\
$10_{12}$ & $87.1055$ & $87.0824$ \\
$10_{13}$ & $88.9148$ & $88.8989$ \\
$10_{14}$ & $88.3232$ & $88.3023$ \\
$10_{15}$ & $87.4787$ & $87.4606$ \\
$10_{16}$ & $87.4946$ & $87.4684$ \\
$10_{17}$ & $87.0473$ & $87.0277$ \\
$10_{18}$ & $88.4257$ & $88.4092$ \\
$10_{19}$ & $87.5311$ & $87.5099$ \\
$10_{20}$ & $86.8731$ & $86.8514$ \\
$10_{21}$ & $87.0497$ & $87.0343$ \\
$10_{22}$ & $87.2417$ & $87.2182$ \\
$10_{23}$ & $88.7048$ & $88.6901$ \\
$10_{24}$ & $88.4160$ & $88.3963$ \\
$10_{25}$ & $88.7767$ & $88.7587$ \\
$10_{26}$ & $88.4564$ & $88.4328$ \\
$10_{27}$ & $89.8944$ & $89.8795$ \\
$10_{28}$ & $87.5276$ & $87.5061$ \\
$10_{29}$ & $89.2410$ & $89.2238$ \\
$10_{30}$ & $88.3731$ & $88.3558$ \\
$10_{31}$ & $88.2624$ & $88.2401$ \\
$10_{32}$ & $88.6809$ & $88.6597$ \\
$10_{33}$ & $88.2952$ & $88.2744$ \\
$10_{34}$ & $87.0322$ & $87.0101$ \\
$10_{35}$ & $88.0891$ & $88.0697$ \\
$10_{36}$ & $88.0424$ & $88.0233$ \\
$10_{37}$ & $88.1319$ & $88.1153$ \\
\bottomrule
\end{tabular}
\hspace{0.25in}
\end{table}
\begin{table}[h]
\caption{\label{ByKnotTablethree}Part 3 of Ropelengths of Tight Knots and Links by Knot Type}
\begin{tabular}[t]{lll} \toprule
Link & $\PRop$ & $\Rop$\\ \midrule 
$10_{37}$ & $88.1319$ & $88.1153$ \\
$10_{38}$ & $88.3478$ & $88.3257$ \\
$10_{39}$ & $88.3562$ & $88.3323$ \\
$10_{40}$ & $89.2659$ & $89.2464$ \\
$10_{41}$ & $89.0725$ & $89.0553$ \\
$10_{42}$ & $89.9013$ & $89.8857$ \\
$10_{43}$ & $89.3512$ & $89.3366$ \\
$10_{44}$ & $88.8714$ & $88.8515$ \\
$10_{45}$ & $89.4836$ & $89.4621$ \\
$10_{46}$ & $86.4718$ & $86.4487$ \\
$10_{47}$ & $87.3043$ & $87.2821$ \\
$10_{48}$ & $87.3814$ & $87.3643$ \\
$10_{49}$ & $88.2914$ & $88.2705$ \\
$10_{50}$ & $87.3876$ & $87.3716$ \\
$10_{51}$ & $88.3209$ & $88.3002$ \\
$10_{52}$ & $88.0719$ & $88.0565$ \\
$10_{53}$ & $88.8361$ & $88.8180$ \\
$10_{54}$ & $87.5336$ & $87.5127$ \\
$10_{55}$ & $88.3760$ & $88.3699$ \\
$10_{56}$ & $89.0160$ & $88.9973$ \\
$10_{57}$ & $89.6126$ & $89.5946$ \\
$10_{58}$ & $88.9623$ & $88.9445$ \\
$10_{59}$ & $89.2228$ & $89.2090$ \\
$10_{60}$ & $89.3397$ & $89.3190$ \\
$10_{61}$ & $86.4755$ & $86.4561$ \\
$10_{62}$ & $87.5318$ & $87.5071$ \\
$10_{63}$ & $88.4046$ & $88.3861$ \\
$10_{64}$ & $87.4878$ & $87.4742$ \\
$10_{65}$ & $88.3918$ & $88.3725$ \\
$10_{66}$ & $89.0275$ & $89.0047$ \\
$10_{67}$ & $88.4741$ & $88.4534$ \\
$10_{68}$ & $88.1199$ & $88.1013$ \\
$10_{69}$ & $89.0983$ & $89.0778$ \\
$10_{70}$ & $89.2068$ & $89.1846$ \\
$10_{71}$ & $89.0853$ & $89.0699$ \\
$10_{72}$ & $89.1974$ & $89.1779$ \\
$10_{73}$ & $89.5332$ & $89.5130$ \\
$10_{74}$ & $88.1285$ & $88.1077$ \\
$10_{75}$ & $88.9725$ & $88.9524$ \\
$10_{76}$ & $88.3673$ & $88.3479$ \\
$10_{77}$ & $88.5689$ & $88.5471$ \\
$10_{78}$ & $88.5548$ & $88.5322$ \\
$10_{79}$ & $88.9647$ & $88.9488$ \\
$10_{80}$ & $89.1669$ & $89.1556$ \\
\bottomrule
\end{tabular}
\hspace{0.25in}
\begin{tabular}[t]{lll} \toprule
Link & $\PRop$ & $\Rop$\\ \midrule 
$10_{80}$ & $89.1669$ & $89.1556$ \\
$10_{81}$ & $90.0181$ & $90.0007$ \\
$10_{82}$ & $88.7011$ & $88.6801$ \\
$10_{83}$ & $89.5544$ & $89.5314$ \\
$10_{84}$ & $89.6518$ & $89.6788$ \\
$10_{85}$ & $87.8403$ & $87.8164$ \\
$10_{86}$ & $88.7050$ & $88.6851$ \\
$10_{87}$ & $89.1363$ & $89.1173$ \\
$10_{88}$ & $89.5638$ & $89.5461$ \\
$10_{89}$ & $89.4343$ & $89.4178$ \\
$10_{90}$ & $88.9330$ & $88.9115$ \\
$10_{91}$ & $88.9611$ & $88.9435$ \\
$10_{92}$ & $89.6200$ & $89.6011$ \\
$10_{93}$ & $88.3962$ & $88.3773$ \\
$10_{94}$ & $88.8514$ & $88.8306$ \\
$10_{95}$ & $90.0056$ & $89.9848$ \\
$10_{96}$ & $89.5493$ & $89.5284$ \\
$10_{97}$ & $89.4340$ & $89.4163$ \\
$10_{98}$ & $89.7172$ & $89.6969$ \\
$10_{99}$ & $88.8926$ & $88.8734$ \\
$10_{100}$ & $88.7124$ & $88.6927$ \\
$10_{101}$ & $89.7344$ & $89.7210$ \\
$10_{102}$ & $88.7969$ & $88.7734$ \\
$10_{103}$ & $88.7971$ & $88.7914$ \\
$10_{104}$ & $91.7476$ & $91.7280$ \\
$10_{105}$ & $89.8260$ & $89.8055$ \\
$10_{106}$ & $89.1546$ & $89.1319$ \\
$10_{107}$ & $89.7525$ & $89.7356$ \\
$10_{108}$ & $88.5137$ & $88.4932$ \\
$10_{109}$ & $91.1966$ & $91.1789$ \\
$10_{110}$ & $89.6275$ & $89.6114$ \\
$10_{111}$ & $89.6677$ & $89.6438$ \\
$10_{112}$ & $89.5744$ & $89.5529$ \\
$10_{113}$ & $90.2239$ & $90.2141$ \\
$10_{114}$ & $89.3062$ & $89.2856$ \\
$10_{115}$ & $90.4340$ & $90.4176$ \\
$10_{116}$ & $90.2703$ & $90.2583$ \\
$10_{117}$ & $89.5335$ & $89.5245$ \\
$10_{118}$ & $89.5261$ & $89.5094$ \\
$10_{119}$ & $90.1394$ & $90.1226$ \\
$10_{120}$ & $90.1862$ & $90.1674$ \\
$10_{121}$ & $89.9375$ & $89.9240$ \\
$10_{122}$ & $89.8258$ & $89.8094$ \\
$10_{123}$ & $92.3646$ & $92.3565$ \\
\bottomrule
\end{tabular}
\hspace{0.25in}
\begin{tabular}[t]{lll} \toprule
Link & $\PRop$ & $\Rop$\\ \midrule 
$10_{123}$ & $92.3646$ & $92.3565$ \\
$10_{124}$ & $71.0894$ & $71.0739$ \\
$10_{125}$ & $74.9907$ & $74.9778$ \\
$10_{126}$ & $77.6202$ & $77.6026$ \\
$10_{127}$ & $80.0235$ & $80.0124$ \\
$10_{128}$ & $76.4187$ & $76.4026$ \\
$10_{129}$ & $78.5739$ & $78.5553$ \\
$10_{130}$ & $78.8499$ & $78.8356$ \\
$10_{131}$ & $81.2871$ & $81.2678$ \\
$10_{132}$ & $74.7441$ & $74.7330$ \\
$10_{133}$ & $77.1813$ & $77.1631$ \\
$10_{134}$ & $78.6521$ & $78.6377$ \\
$10_{135}$ & $81.2305$ & $81.2157$ \\
$10_{136}$ & $78.0398$ & $78.0276$ \\
$10_{137}$ & $79.6352$ & $79.6185$ \\
$10_{138}$ & $82.5504$ & $82.5320$ \\
$10_{139}$ & $72.9001$ & $72.8944$ \\
$10_{140}$ & $73.8610$ & $73.8477$ \\
$10_{141}$ & $76.9687$ & $76.9543$ \\
$10_{142}$ & $75.8951$ & $75.8754$ \\
$10_{143}$ & $78.2422$ & $78.2307$ \\
$10_{144}$ & $81.4378$ & $81.4275$ \\
$10_{145}$ & $75.9194$ & $75.9076$ \\
$10_{146}$ & $79.7416$ & $79.7322$ \\
$10_{147}$ & $79.1666$ & $79.1571$ \\
$10_{148}$ & $79.0893$ & $79.0742$ \\
$10_{149}$ & $81.0500$ & $81.0318$ \\
$10_{150}$ & $80.1392$ & $80.1219$ \\
$10_{151}$ & $81.8414$ & $81.8207$ \\
$10_{152}$ & $79.1715$ & $79.1556$ \\
$10_{153}$ & $80.4764$ & $80.4648$ \\
$10_{154}$ & $81.5405$ & $81.5218$ \\
$10_{155}$ & $78.0648$ & $78.0503$ \\
$10_{156}$ & $79.5639$ & $79.5443$ \\
$10_{157}$ & $81.4731$ & $81.4568$ \\
$10_{158}$ & $81.6398$ & $81.6220$ \\
$10_{159}$ & $79.8863$ & $79.8692$ \\
$10_{160}$ & $78.1529$ & $78.1472$ \\
$10_{161}$ & $74.5460$ & $74.5302$ \\
$10_{162}$ & $81.0033$ & $80.9838$ \\
$10_{163}$ & $82.6629$ & $82.6548$ \\
$10_{164}$ & $82.1862$ & $82.1698$ \\
$10_{165}$ & $82.8211$ & $82.8040$ \\
\bottomrule
\end{tabular}
\hspace{0.25in}
\end{table}

%% file: KnotTableByRop.tex
\begin{table}[h]
\caption{\label{ByRopTableone}Part 1 of Knot and Link Types sorted by Ropelength}
\begin{tabular}[t]{l} \toprule
Link\\ \midrule 
$2^{2}_{1}$ \\
$3_{1}$ \\
$4^{2}_{1}$ \\
$4_{1}$ \\
$5_{1}$ \\
$5_{2}$ \\
$5^{2}_{1}$ \\
$6^{3}_{3}$ \\
$6^{2}_{1}$ \\
$7^{2}_{7}$ \\
$6^{2}_{2}$ \\
$6_{1}$ \\
$6_{2}$ \\
$7^{2}_{8}$ \\
$6^{3}_{1}$ \\
$6_{3}$ \\
$6^{3}_{2}$ \\
$6^{2}_{3}$ \\
$8^{3}_{7}$ \\
$8_{19}$ \\
$7_{1}$ \\
$8_{20}$ \\
$7_{2}$ \\
$7_{3}$ \\
$7^{2}_{1}$ \\
$7_{4}$ \\
$8^{2}_{15}$ \\
$8^{3}_{8}$ \\
$7^{2}_{2}$ \\
$7^{2}_{4}$ \\
$7_{5}$ \\
$7^{2}_{3}$ \\
$8_{21}$ \\
$7_{7}$ \\
$7_{6}$ \\
$7^{3}_{1}$ \\
$9^{2}_{49}$ \\
$7^{2}_{5}$ \\
$9^{2}_{43}$ \\
$8^{4}_{3}$ \\
$7^{2}_{6}$ \\
$8^{3}_{9}$ \\
$8^{2}_{16}$ \\
$8^{4}_{2}$ \\
$9^{2}_{53}$ \\
$8^{2}_{1}$ \\
\bottomrule
\end{tabular}
\hspace{0.25in}
\begin{tabular}[t]{l} \toprule
Link\\ \midrule 
$8^{2}_{1}$ \\
$8^{3}_{10}$ \\
$9_{46}$ \\
$9^{2}_{50}$ \\
$9^{2}_{61}$ \\
$9_{42}$ \\
$9^{2}_{47}$ \\
$9^{2}_{51}$ \\
$8_{1}$ \\
$9^{2}_{54}$ \\
$8^{2}_{2}$ \\
$9^{2}_{45}$ \\
$10_{124}$ \\
$8_{3}$ \\
$8_{2}$ \\
$8_{4}$ \\
$9_{43}$ \\
$9_{44}$ \\
$8_{5}$ \\
$9^{2}_{44}$ \\
$9^{2}_{57}$ \\
$9^{3}_{13}$ \\
$8_{7}$ \\
$8^{3}_{1}$ \\
$9^{2}_{59}$ \\
$9^{3}_{18}$ \\
$8_{9}$ \\
$8_{6}$ \\
$8^{2}_{4}$ \\
$9^{3}_{19}$ \\
$8^{2}_{3}$ \\
$8_{8}$ \\
$8_{13}$ \\
$9^{2}_{52}$ \\
$9^{2}_{56}$ \\
$10_{139}$ \\
$8_{11}$ \\
$8^{3}_{2}$ \\
$8_{10}$ \\
$8^{2}_{11}$ \\
$8^{2}_{6}$ \\
$8^{2}_{9}$ \\
$8^{3}_{5}$ \\
$8_{16}$ \\
$9^{2}_{60}$ \\
$9^{2}_{48}$ \\
\bottomrule
\end{tabular}
\hspace{0.25in}
\begin{tabular}[t]{l} \toprule
Link\\ \midrule 
$9^{2}_{48}$ \\
$8^{2}_{10}$ \\
$8^{2}_{14}$ \\
$8^{2}_{8}$ \\
$8_{14}$ \\
$8^{2}_{12}$ \\
$9^{2}_{55}$ \\
$9^{2}_{46}$ \\
$10_{140}$ \\
$8_{15}$ \\
$9_{49}$ \\
$8^{2}_{5}$ \\
$8_{12}$ \\
$9_{48}$ \\
$9_{45}$ \\
$9^{3}_{17}$ \\
$8^{2}_{13}$ \\
$9^{2}_{58}$ \\
$9^{3}_{15}$ \\
$8^{2}_{7}$ \\
$9^{3}_{14}$ \\
$8_{17}$ \\
$10_{161}$ \\
$8^{3}_{6}$ \\
$10_{132}$ \\
$8^{3}_{3}$ \\
$9_{47}$ \\
$9^{3}_{21}$ \\
$8_{18}$ \\
$10_{125}$ \\
$8^{3}_{4}$ \\
$9^{3}_{16}$ \\
$8^{4}_{1}$ \\
$9_{1}$ \\
$10_{142}$ \\
$10_{145}$ \\
$9^{3}_{20}$ \\
$10_{128}$ \\
$10_{141}$ \\
$10_{133}$ \\
$10_{126}$ \\
$10_{136}$ \\
$10_{155}$ \\
$9_{2}$ \\
$10_{160}$ \\
$9_{3}$ \\
\bottomrule
\end{tabular}
\hspace{0.25in}
\begin{tabular}[t]{l} \toprule
Link\\ \midrule 
$9_{3}$ \\
$10_{143}$ \\
$9_{4}$ \\
$10_{129}$ \\
$9^{2}_{1}$ \\
$10_{134}$ \\
$9_{5}$ \\
$9^{2}_{4}$ \\
$10_{130}$ \\
$10_{148}$ \\
$10_{152}$ \\
$10_{147}$ \\
$9_{35}$ \\
$9^{2}_{13}$ \\
$9^{2}_{19}$ \\
$9^{2}_{2}$ \\
$10_{156}$ \\
$9_{6}$ \\
$10_{137}$ \\
$9^{2}_{5}$ \\
$9_{7}$ \\
$10_{146}$ \\
$9_{10}$ \\
$10_{159}$ \\
$9_{9}$ \\
$9^{2}_{3}$ \\
$9_{14}$ \\
$9_{8}$ \\
$10_{127}$ \\
$9_{12}$ \\
$9^{2}_{6}$ \\
$9_{16}$ \\
$9^{2}_{20}$ \\
$9_{11}$ \\
$10_{150}$ \\
$9_{36}$ \\
$9_{20}$ \\
$9^{2}_{23}$ \\
$9^{3}_{5}$ \\
$9_{13}$ \\
$9^{2}_{9}$ \\
$9^{2}_{17}$ \\
$9^{2}_{10}$ \\
$9_{17}$ \\
$10_{153}$ \\
$9^{2}_{15}$ \\
\bottomrule
\end{tabular}
\hspace{0.25in}
\begin{tabular}[t]{l} \toprule
Link\\ \midrule 
$9^{2}_{15}$ \\
$9^{2}_{31}$ \\
$9^{2}_{21}$ \\
$9^{2}_{36}$ \\
$9^{2}_{14}$ \\
$9_{15}$ \\
$9_{19}$ \\
$9_{26}$ \\
$9^{3}_{6}$ \\
$9_{24}$ \\
$9^{2}_{8}$ \\
$10_{162}$ \\
$9^{2}_{27}$ \\
$10_{149}$ \\
$9_{22}$ \\
$9^{2}_{22}$ \\
$9_{21}$ \\
$9^{3}_{8}$ \\
$9_{25}$ \\
$9^{2}_{7}$ \\
$9^{3}_{1}$ \\
$9_{28}$ \\
$9_{37}$ \\
$9_{27}$ \\
$9_{29}$ \\
$10_{135}$ \\
$9^{2}_{35}$ \\
$10_{131}$ \\
$9_{23}$ \\
$9^{2}_{28}$ \\
$9_{41}$ \\
$9^{2}_{16}$ \\
$9^{2}_{32}$ \\
$10_{144}$ \\
$10_{157}$ \\
$9_{30}$ \\
$9_{32}$ \\
$10_{154}$ \\
$9^{3}_{9}$ \\
$9_{18}$ \\
$9^{4}_{1}$ \\
$10_{158}$ \\
$9_{40}$ \\
$9_{31}$ \\
$9^{3}_{2}$ \\
$9^{2}_{18}$ \\
\bottomrule
\end{tabular}
\hspace{0.25in}
\begin{tabular}[t]{l} \toprule
Link\\ \midrule 
$9^{2}_{18}$ \\
$9^{2}_{25}$ \\
$9^{2}_{24}$ \\
$9_{38}$ \\
$10_{151}$ \\
$9_{39}$ \\
$9^{2}_{34}$ \\
$9^{2}_{39}$ \\
$9^{2}_{37}$ \\
$9^{2}_{12}$ \\
$9^{2}_{40}$ \\
$9^{3}_{11}$ \\
$9^{2}_{11}$ \\
$9^{3}_{7}$ \\
$9^{2}_{26}$ \\
$9^{2}_{29}$ \\
$9^{2}_{33}$ \\
$10_{164}$ \\
$9_{34}$ \\
$9^{2}_{30}$ \\
$9^{3}_{3}$ \\
$9^{3}_{10}$ \\
$9^{3}_{12}$ \\
$9^{3}_{4}$ \\
$10_{138}$ \\
$10_{163}$ \\
$9^{2}_{38}$ \\
$9_{33}$ \\
$10_{165}$ \\
$9^{2}_{41}$ \\
$9^{2}_{42}$ \\
$10_{1}$ \\
$10_{3}$ \\
$10_{2}$ \\
$10_{4}$ \\
$10_{8}$ \\
$10_{46}$ \\
$10_{61}$ \\
$10_{5}$ \\
$10_{6}$ \\
$10_{9}$ \\
$10_{20}$ \\
$10_{11}$ \\
$10_{34}$ \\
$10_{17}$ \\
$10_{21}$ \\
\bottomrule
\end{tabular}
\hspace{0.25in}
\begin{tabular}[t]{l} \toprule
Link\\ \midrule 
$10_{21}$ \\
$10_{12}$ \\
$10_{10}$ \\
$10_{22}$ \\
$10_{7}$ \\
$10_{47}$ \\
$10_{48}$ \\
$10_{50}$ \\
$10_{15}$ \\
$10_{16}$ \\
$10_{64}$ \\
$10_{28}$ \\
$10_{62}$ \\
$10_{19}$ \\
$10_{54}$ \\
$10_{85}$ \\
$10_{36}$ \\
$10_{52}$ \\
$10_{35}$ \\
$10_{68}$ \\
$10_{74}$ \\
$10_{37}$ \\
$10_{31}$ \\
$10_{49}$ \\
$10_{33}$ \\
$10_{51}$ \\
$10_{14}$ \\
$10_{38}$ \\
$10_{39}$ \\
$10_{76}$ \\
$10_{30}$ \\
$10_{55}$ \\
$10_{65}$ \\
$10_{93}$ \\
$10_{63}$ \\
$10_{24}$ \\
$10_{18}$ \\
$10_{26}$ \\
$10_{67}$ \\
$10_{108}$ \\
$10_{78}$ \\
$10_{77}$ \\
$10_{32}$ \\
$10_{82}$ \\
$10_{86}$ \\
$10_{23}$ \\
\bottomrule
\end{tabular}
\hspace{0.25in}
\begin{tabular}[t]{l} \toprule
Link\\ \midrule 
$10_{23}$ \\
$10_{100}$ \\
$10_{25}$ \\
$10_{102}$ \\
$10_{103}$ \\
$10_{53}$ \\
$10_{94}$ \\
$10_{44}$ \\
$10_{99}$ \\
$10_{13}$ \\
$10_{90}$ \\
$10_{91}$ \\
$10_{58}$ \\
$10_{79}$ \\
$10_{75}$ \\
$10_{56}$ \\
$10_{66}$ \\
$10_{41}$ \\
$10_{71}$ \\
$10_{69}$ \\
$10_{87}$ \\
$10_{106}$ \\
$10_{80}$ \\
$10_{72}$ \\
$10_{70}$ \\
$10_{59}$ \\
$10_{29}$ \\
$10_{40}$ \\
$10_{114}$ \\
$10_{60}$ \\
$10_{43}$ \\
$10_{97}$ \\
$10_{89}$ \\
$10_{45}$ \\
$10_{118}$ \\
$10_{73}$ \\
$10_{117}$ \\
$10_{96}$ \\
$10_{83}$ \\
$10_{88}$ \\
$10_{112}$ \\
$10_{57}$ \\
$10_{92}$ \\
$10_{110}$ \\
$10_{111}$ \\
$10_{84}$ \\
\bottomrule
\end{tabular}
\hspace{0.25in}
\end{table}
\begin{table}[h]
\caption{\label{ByRopTabletwo}Part 2 of Knot and Link Types sorted by Ropelength}
\begin{tabular}[t]{l} \toprule
Link\\ \midrule 
$10_{84}$ \\
$10_{98}$ \\
$10_{101}$ \\
$10_{107}$ \\
\bottomrule
\end{tabular}
\hspace{0.25in}
\begin{tabular}[t]{l} \toprule
Link\\ \midrule 
$10_{107}$ \\
$10_{105}$ \\
$10_{122}$ \\
$10_{27}$ \\
\bottomrule
\end{tabular}
\hspace{0.25in}
\begin{tabular}[t]{l} \toprule
Link\\ \midrule 
$10_{27}$ \\
$10_{42}$ \\
$10_{121}$ \\
$10_{95}$ \\
\bottomrule
\end{tabular}
\hspace{0.25in}
\begin{tabular}[t]{l} \toprule
Link\\ \midrule 
$10_{95}$ \\
$10_{81}$ \\
$10_{119}$ \\
$10_{120}$ \\
\bottomrule
\end{tabular}
\hspace{0.25in}
\begin{tabular}[t]{l} \toprule
Link\\ \midrule 
$10_{120}$ \\
$10_{113}$ \\
$10_{116}$ \\
$10_{115}$ \\
\bottomrule
\end{tabular}
\hspace{0.25in}
\begin{tabular}[t]{l} \toprule
Link\\ \midrule 
$10_{115}$ \\
$10_{109}$ \\
$10_{104}$ \\
$10_{123}$ \\
\bottomrule
\end{tabular}
\hspace{0.25in}
\begin{tabular}[t]{l} \toprule
Link\\ \midrule 
$10_{123}$ \\
\bottomrule
\end{tabular}
\hspace{0.25in}
\end{table}

%% file: ResidualTable.tex
\begin{table}[h]
\caption{\label{ResidualTableone}Part 1 of Residuals of Tight Knots and Links by Knot Type}
\begin{tabular}[t]{ll} \toprule
Link & Residual\\ \midrule 
$2^{2}_{1}$ & $2.45124e-05$ \\
\addlinespace[0.45em]\cmidrule(r{0.5em}l{0.5em}){1-2}\addlinespace[0.45em]
$3_{1}$ & $0.00621792$ \\
\addlinespace[0.45em]\cmidrule(r{0.5em}l{0.5em}){1-2}\addlinespace[0.45em]
$4_{1}$ & $0.000996335$ \\
\addlinespace[0.45em]\cmidrule(r{1.5em}l{1.5em}){1-2}\addlinespace[0.45em]
$4^{2}_{1}$ & $0.000999549$ \\
\addlinespace[0.45em]\cmidrule(r{0.5em}l{0.5em}){1-2}\addlinespace[0.45em]
$5_{1}$ & $0.00981995$ \\
$5_{2}$ & $0.00994775$ \\
\addlinespace[0.45em]\cmidrule(r{1.5em}l{1.5em}){1-2}\addlinespace[0.45em]
$5^{2}_{1}$ & $0.00998078$ \\
\addlinespace[0.45em]\cmidrule(r{0.5em}l{0.5em}){1-2}\addlinespace[0.45em]
$6_{1}$ & $0.000999592$ \\
$6_{2}$ & $0.00897204$ \\
$6_{3}$ & $0.000979541$ \\
\addlinespace[0.45em]\cmidrule(r{1.5em}l{1.5em}){1-2}\addlinespace[0.45em]
$6^{2}_{1}$ & $0.000999952$ \\
$6^{2}_{2}$ & $0.000999833$ \\
$6^{2}_{3}$ & $0.00999004$ \\
\addlinespace[0.45em]\cmidrule(r{1.5em}l{1.5em}){1-2}\addlinespace[0.45em]
$6^{3}_{1}$ & $0.00998537$ \\
$6^{3}_{2}$ & $0.000705159$ \\
$6^{3}_{3}$ & $0.00627026$ \\
\addlinespace[0.45em]\cmidrule(r{0.5em}l{0.5em}){1-2}\addlinespace[0.45em]
$7_{1}$ & $0.00105833$ \\
$7_{2}$ & $0.00998149$ \\
$7_{3}$ & $0.00999358$ \\
$7_{4}$ & $0.00100877$ \\
$7_{5}$ & $0.000999532$ \\
$7_{6}$ & $0.000979869$ \\
$7_{7}$ & $0.00100393$ \\
\addlinespace[0.45em]\cmidrule(r{1.5em}l{1.5em}){1-2}\addlinespace[0.45em]
$7^{2}_{1}$ & $0.000999487$ \\
$7^{2}_{2}$ & $0.00101952$ \\
$7^{2}_{3}$ & $0.000999871$ \\
$7^{2}_{4}$ & $0.00099954$ \\
$7^{2}_{5}$ & $0.000999894$ \\
$7^{2}_{6}$ & $0.00100556$ \\
$7^{2}_{7}$ & $0.00320787$ \\
$7^{2}_{8}$ & $0.0018494$ \\
\addlinespace[0.45em]\cmidrule(r{1.5em}l{1.5em}){1-2}\addlinespace[0.45em]
$7^{3}_{1}$ & $0.000999748$ \\
\addlinespace[0.45em]\cmidrule(r{0.5em}l{0.5em}){1-2}\addlinespace[0.45em]
$8_{1}$ & $0.00898769$ \\
$8_{2}$ & $0.000982684$ \\
\bottomrule
\end{tabular}
\hspace{0.25in}
\begin{tabular}[t]{ll} \toprule
Link & Residual\\ \midrule 
$8_{2}$ & $0.000982684$ \\
$8_{3}$ & $0.00100028$ \\
$8_{4}$ & $0.00100103$ \\
$8_{5}$ & $0.00100033$ \\
$8_{6}$ & $0.000999848$ \\
$8_{7}$ & $0.00101551$ \\
$8_{8}$ & $0.000981272$ \\
$8_{9}$ & $0.000999932$ \\
$8_{10}$ & $0.000978418$ \\
$8_{11}$ & $0.000979921$ \\
$8_{12}$ & $0.00998976$ \\
$8_{13}$ & $0.000993117$ \\
$8_{14}$ & $0.000981486$ \\
$8_{15}$ & $0.0099948$ \\
$8_{16}$ & $0.000981316$ \\
$8_{17}$ & $0.00999085$ \\
$8_{18}$ & $0.000900015$ \\
$8_{19}$ & $0.000998339$ \\
$8_{20}$ & $0.00099998$ \\
$8_{21}$ & $0.000999988$ \\
\addlinespace[0.45em]\cmidrule(r{1.5em}l{1.5em}){1-2}\addlinespace[0.45em]
$8^{2}_{1}$ & $0.00100142$ \\
$8^{2}_{2}$ & $0.000979836$ \\
$8^{2}_{3}$ & $0.000999961$ \\
$8^{2}_{4}$ & $0.00216462$ \\
$8^{2}_{5}$ & $0.00999516$ \\
$8^{2}_{6}$ & $0.00100295$ \\
$8^{2}_{7}$ & $0.000999802$ \\
$8^{2}_{8}$ & $0.000999762$ \\
$8^{2}_{9}$ & $0.000979774$ \\
$8^{2}_{10}$ & $0.000999858$ \\
$8^{2}_{11}$ & $0.00997927$ \\
$8^{2}_{12}$ & $0.000999968$ \\
$8^{2}_{13}$ & $0.0010008$ \\
$8^{2}_{14}$ & $0.00101123$ \\
$8^{2}_{15}$ & $0.00099994$ \\
$8^{2}_{16}$ & $0.000997563$ \\
\addlinespace[0.45em]\cmidrule(r{1.5em}l{1.5em}){1-2}\addlinespace[0.45em]
$8^{3}_{1}$ & $0.00100589$ \\
$8^{3}_{2}$ & $0.000999904$ \\
$8^{3}_{3}$ & $0.00100014$ \\
$8^{3}_{4}$ & $0.00999606$ \\
$8^{3}_{5}$ & $0.000995844$ \\
$8^{3}_{6}$ & $0.00099824$ \\
$8^{3}_{7}$ & $0.00119532$ \\
$8^{3}_{8}$ & $0.00100655$ \\
\bottomrule
\end{tabular}
\hspace{0.25in}
\begin{tabular}[t]{ll} \toprule
Link & Residual\\ \midrule 
$8^{3}_{8}$ & $0.00100655$ \\
$8^{3}_{9}$ & $0.000980533$ \\
$8^{3}_{10}$ & $0.0208108$ \\
\addlinespace[0.45em]\cmidrule(r{1.5em}l{1.5em}){1-2}\addlinespace[0.45em]
$8^{4}_{1}$ & $0.00100006$ \\
$8^{4}_{2}$ & $0.000999682$ \\
$8^{4}_{3}$ & $0.780186$ \\
\addlinespace[0.45em]\cmidrule(r{0.5em}l{0.5em}){1-2}\addlinespace[0.45em]
$9_{1}$ & $0.00802077$ \\
$9_{2}$ & $0.00997484$ \\
$9_{3}$ & $0.00998254$ \\
$9_{4}$ & $8.64059e-05$ \\
$9_{5}$ & $0.00999417$ \\
$9_{6}$ & $0.000980197$ \\
$9_{7}$ & $0.000979897$ \\
$9_{8}$ & $0.00101007$ \\
$9_{9}$ & $0.000999938$ \\
$9_{10}$ & $0.00113523$ \\
$9_{11}$ & $0.000981742$ \\
$9_{12}$ & $0.000979842$ \\
$9_{13}$ & $0.00999582$ \\
$9_{14}$ & $0.000984327$ \\
$9_{15}$ & $0.000979831$ \\
$9_{16}$ & $0.000999818$ \\
$9_{17}$ & $0.00100032$ \\
$9_{18}$ & $0.00992217$ \\
$9_{19}$ & $0.000981217$ \\
$9_{20}$ & $0.00100005$ \\
$9_{21}$ & $0.0010001$ \\
$9_{22}$ & $0.000998846$ \\
$9_{23}$ & $0.000979562$ \\
$9_{24}$ & $0.000999907$ \\
$9_{25}$ & $0.000977105$ \\
$9_{26}$ & $0.00100048$ \\
$9_{27}$ & $0.00999324$ \\
$9_{28}$ & $0.00996501$ \\
$9_{29}$ & $0.000979844$ \\
$9_{30}$ & $0.000979942$ \\
$9_{31}$ & $0.000979062$ \\
$9_{32}$ & $0.000997746$ \\
$9_{33}$ & $0.00100114$ \\
$9_{34}$ & $0.000999697$ \\
$9_{35}$ & $0.000981383$ \\
$9_{36}$ & $0.000978472$ \\
$9_{37}$ & $0.00999228$ \\
$9_{38}$ & $0.000978978$ \\
\bottomrule
\end{tabular}
\hspace{0.25in}
\begin{tabular}[t]{ll} \toprule
Link & Residual\\ \midrule 
$9_{38}$ & $0.000978978$ \\
$9_{39}$ & $0.000999482$ \\
$9_{40}$ & $0.000999343$ \\
$9_{41}$ & $0.00899161$ \\
$9_{42}$ & $0.000999996$ \\
$9_{43}$ & $0.00898749$ \\
$9_{44}$ & $0.000999789$ \\
$9_{45}$ & $0.0099754$ \\
$9_{46}$ & $0.00099973$ \\
$9_{47}$ & $0.000998991$ \\
$9_{48}$ & $0.00998933$ \\
$9_{49}$ & $0.00099957$ \\
\addlinespace[0.45em]\cmidrule(r{1.5em}l{1.5em}){1-2}\addlinespace[0.45em]
$9^{2}_{1}$ & $0.00107787$ \\
$9^{2}_{2}$ & $0.00100115$ \\
$9^{2}_{3}$ & $0.00100055$ \\
$9^{2}_{4}$ & $0.00099991$ \\
$9^{2}_{5}$ & $0.00100118$ \\
$9^{2}_{6}$ & $0.00126944$ \\
$9^{2}_{7}$ & $0.00104121$ \\
$9^{2}_{8}$ & $0.00100133$ \\
$9^{2}_{9}$ & $0.000999724$ \\
$9^{2}_{10}$ & $0.00140283$ \\
$9^{2}_{11}$ & $0.000999221$ \\
$9^{2}_{12}$ & $0.00100137$ \\
$9^{2}_{13}$ & $0.00100112$ \\
$9^{2}_{14}$ & $0.000999788$ \\
$9^{2}_{15}$ & $0.000999236$ \\
$9^{2}_{16}$ & $0.00605$ \\
$9^{2}_{17}$ & $0.00899775$ \\
$9^{2}_{18}$ & $0.000999648$ \\
$9^{2}_{19}$ & $0.00100405$ \\
$9^{2}_{20}$ & $0.000999853$ \\
$9^{2}_{21}$ & $0.00898977$ \\
$9^{2}_{22}$ & $0.00943088$ \\
$9^{2}_{23}$ & $0.000998181$ \\
$9^{2}_{24}$ & $0.000999946$ \\
$9^{2}_{25}$ & $0.0009999$ \\
$9^{2}_{26}$ & $0.00100243$ \\
$9^{2}_{27}$ & $0.00099997$ \\
$9^{2}_{28}$ & $0.000998883$ \\
$9^{2}_{29}$ & $0.00100157$ \\
$9^{2}_{30}$ & $0.00099989$ \\
$9^{2}_{31}$ & $0.000999523$ \\
$9^{2}_{32}$ & $0.00100012$ \\
$9^{2}_{33}$ & $0.000999711$ \\
\bottomrule
\end{tabular}
\hspace{0.25in}
\end{table}
\begin{table}[h]
\caption{\label{ResidualTabletwo}Part 2 of Residuals of Tight Knots and Links by Knot Type}
\begin{tabular}[t]{ll} \toprule
Link & Residual\\ \midrule 
$9^{2}_{33}$ & $0.000999711$ \\
$9^{2}_{34}$ & $0.00100169$ \\
$9^{2}_{35}$ & $0.000999778$ \\
$9^{2}_{36}$ & $0.00100172$ \\
$9^{2}_{37}$ & $0.000999058$ \\
$9^{2}_{38}$ & $0.000999748$ \\
$9^{2}_{39}$ & $0.000999888$ \\
$9^{2}_{40}$ & $0.000999835$ \\
$9^{2}_{41}$ & $0.00100037$ \\
$9^{2}_{42}$ & $0.000998679$ \\
$9^{2}_{43}$ & $0.00100109$ \\
$9^{2}_{44}$ & $0.00100838$ \\
$9^{2}_{45}$ & $0.00997492$ \\
$9^{2}_{46}$ & $0.00100042$ \\
$9^{2}_{47}$ & $0.00999831$ \\
$9^{2}_{48}$ & $0.000999984$ \\
$9^{2}_{49}$ & $0.000999984$ \\
$9^{2}_{50}$ & $0.000999226$ \\
$9^{2}_{51}$ & $0.000999443$ \\
$9^{2}_{52}$ & $0.000999958$ \\
$9^{2}_{53}$ & $0.00996962$ \\
$9^{2}_{54}$ & $0.000999703$ \\
$9^{2}_{55}$ & $0.00100064$ \\
$9^{2}_{56}$ & $0.000979788$ \\
$9^{2}_{57}$ & $0.00255237$ \\
$9^{2}_{58}$ & $0.000999155$ \\
$9^{2}_{59}$ & $0.00108631$ \\
$9^{2}_{60}$ & $0.000999312$ \\
$9^{2}_{61}$ & $0.00100091$ \\
\addlinespace[0.45em]\cmidrule(r{1.5em}l{1.5em}){1-2}\addlinespace[0.45em]
$9^{3}_{1}$ & $0.000999763$ \\
$9^{3}_{2}$ & $0.000999746$ \\
$9^{3}_{3}$ & $0.00100525$ \\
$9^{3}_{4}$ & $0.000999641$ \\
$9^{3}_{5}$ & $0.00100042$ \\
$9^{3}_{6}$ & $0.000999746$ \\
$9^{3}_{7}$ & $0.000999935$ \\
$9^{3}_{8}$ & $0.000999751$ \\
$9^{3}_{9}$ & $0.000996684$ \\
$9^{3}_{10}$ & $0.00099985$ \\
$9^{3}_{11}$ & $0.0010755$ \\
$9^{3}_{12}$ & $0.00100439$ \\
$9^{3}_{13}$ & $0.00980919$ \\
$9^{3}_{14}$ & $0.00900147$ \\
$9^{3}_{15}$ & $0.00112426$ \\
$9^{3}_{16}$ & $0.000999575$ \\
\bottomrule
\end{tabular}
\hspace{0.25in}
\begin{tabular}[t]{ll} \toprule
Link & Residual\\ \midrule 
$9^{3}_{16}$ & $0.000999575$ \\
$9^{3}_{17}$ & $0.247874$ \\
$9^{3}_{18}$ & $0.000999841$ \\
$9^{3}_{19}$ & $0.00101035$ \\
$9^{3}_{20}$ & $0.00100002$ \\
$9^{3}_{21}$ & $0.00100039$ \\
\addlinespace[0.45em]\cmidrule(r{1.5em}l{1.5em}){1-2}\addlinespace[0.45em]
$9^{4}_{1}$ & $0.000979958$ \\
\addlinespace[0.45em]\cmidrule(r{0.5em}l{0.5em}){1-2}\addlinespace[0.45em]
$10_{1}$ & $0.00101691$ \\
$10_{2}$ & $0.00100023$ \\
$10_{3}$ & $0.000991435$ \\
$10_{4}$ & $0.00100846$ \\
$10_{5}$ & $0.00100194$ \\
$10_{6}$ & $0.000979506$ \\
$10_{7}$ & $0.0097283$ \\
$10_{8}$ & $0.000980356$ \\
$10_{9}$ & $0.000979784$ \\
$10_{10}$ & $0.00999688$ \\
$10_{11}$ & $0.00760935$ \\
$10_{12}$ & $0.000991292$ \\
$10_{13}$ & $0.000999947$ \\
$10_{14}$ & $0.0010261$ \\
$10_{15}$ & $0.000979185$ \\
$10_{16}$ & $0.000985699$ \\
$10_{17}$ & $0.00998848$ \\
$10_{18}$ & $0.000979621$ \\
$10_{19}$ & $0.00098045$ \\
$10_{20}$ & $0.000979959$ \\
$10_{21}$ & $0.000999057$ \\
$10_{22}$ & $0.000991413$ \\
$10_{23}$ & $0.00999682$ \\
$10_{24}$ & $0.00166886$ \\
$10_{25}$ & $0.000994731$ \\
$10_{26}$ & $0.00098015$ \\
$10_{27}$ & $0.000999869$ \\
$10_{28}$ & $0.00996703$ \\
$10_{29}$ & $0.00116525$ \\
$10_{30}$ & $0.000999376$ \\
$10_{31}$ & $0.000979897$ \\
$10_{32}$ & $0.000979993$ \\
$10_{33}$ & $0.000979857$ \\
$10_{34}$ & $0.00098555$ \\
$10_{35}$ & $0.000982115$ \\
$10_{36}$ & $0.000979692$ \\
$10_{37}$ & $0.000999835$ \\
\bottomrule
\end{tabular}
\hspace{0.25in}
\begin{tabular}[t]{ll} \toprule
Link & Residual\\ \midrule 
$10_{37}$ & $0.000999835$ \\
$10_{38}$ & $0.000979821$ \\
$10_{39}$ & $0.000986038$ \\
$10_{40}$ & $0.00100863$ \\
$10_{41}$ & $0.00999693$ \\
$10_{42}$ & $0.000999751$ \\
$10_{43}$ & $0.000980157$ \\
$10_{44}$ & $0.00322255$ \\
$10_{45}$ & $0.000982692$ \\
$10_{46}$ & $0.00997656$ \\
$10_{47}$ & $0.000980999$ \\
$10_{48}$ & $0.00999602$ \\
$10_{49}$ & $0.000998073$ \\
$10_{50}$ & $0.000981787$ \\
$10_{51}$ & $0.00098231$ \\
$10_{52}$ & $0.000999419$ \\
$10_{53}$ & $0.00101025$ \\
$10_{54}$ & $0.00999263$ \\
$10_{55}$ & $0.00998728$ \\
$10_{56}$ & $0.00999185$ \\
$10_{57}$ & $0.000999798$ \\
$10_{58}$ & $0.000999966$ \\
$10_{59}$ & $0.00995441$ \\
$10_{60}$ & $0.000980266$ \\
$10_{61}$ & $0.0241498$ \\
$10_{62}$ & $0.00105699$ \\
$10_{63}$ & $0.00998227$ \\
$10_{64}$ & $0.00997603$ \\
$10_{65}$ & $0.00135295$ \\
$10_{66}$ & $0.000999872$ \\
$10_{67}$ & $0.000979823$ \\
$10_{68}$ & $0.00100695$ \\
$10_{69}$ & $0.000999786$ \\
$10_{70}$ & $0.000980057$ \\
$10_{71}$ & $0.00999226$ \\
$10_{72}$ & $0.000999942$ \\
$10_{73}$ & $0.00998888$ \\
$10_{74}$ & $0.000978382$ \\
$10_{75}$ & $0.000981812$ \\
$10_{76}$ & $0.000980892$ \\
$10_{77}$ & $0.00999768$ \\
$10_{78}$ & $0.000981017$ \\
$10_{79}$ & $0.0010001$ \\
$10_{80}$ & $0.000979926$ \\
$10_{81}$ & $0.000981576$ \\
$10_{82}$ & $0.000978946$ \\
\bottomrule
\end{tabular}
\hspace{0.25in}
\begin{tabular}[t]{ll} \toprule
Link & Residual\\ \midrule 
$10_{82}$ & $0.000978946$ \\
$10_{83}$ & $0.00999433$ \\
$10_{84}$ & $0.0099812$ \\
$10_{85}$ & $0.000981325$ \\
$10_{86}$ & $0.000978499$ \\
$10_{87}$ & $0.000979621$ \\
$10_{88}$ & $0.000979845$ \\
$10_{89}$ & $0.0010019$ \\
$10_{90}$ & $0.000980234$ \\
$10_{91}$ & $0.000977397$ \\
$10_{92}$ & $0.00100005$ \\
$10_{93}$ & $0.000979652$ \\
$10_{94}$ & $0.00097991$ \\
$10_{95}$ & $0.000979668$ \\
$10_{96}$ & $0.00018365$ \\
$10_{97}$ & $0.000999872$ \\
$10_{98}$ & $0.00999481$ \\
$10_{99}$ & $0.0099926$ \\
$10_{100}$ & $0.00101003$ \\
$10_{101}$ & $0.00999705$ \\
$10_{102}$ & $0.000979674$ \\
$10_{103}$ & $0.00999479$ \\
$10_{104}$ & $0.00999683$ \\
$10_{105}$ & $0.000979902$ \\
$10_{106}$ & $0.000979055$ \\
$10_{107}$ & $0.000980096$ \\
$10_{108}$ & $0.00127554$ \\
$10_{109}$ & $0.000979798$ \\
$10_{110}$ & $0.000979638$ \\
$10_{111}$ & $0.000979851$ \\
$10_{112}$ & $0.00104599$ \\
$10_{113}$ & $0.00999934$ \\
$10_{114}$ & $0.00100087$ \\
$10_{115}$ & $0.000978725$ \\
$10_{116}$ & $0.00998661$ \\
$10_{117}$ & $0.00998396$ \\
$10_{118}$ & $0.00099987$ \\
$10_{119}$ & $0.000999834$ \\
$10_{120}$ & $0.00100037$ \\
$10_{121}$ & $0.00099989$ \\
$10_{122}$ & $0.000999203$ \\
$10_{123}$ & $0.0016528$ \\
$10_{124}$ & $0.00100133$ \\
$10_{125}$ & $0.00998345$ \\
$10_{126}$ & $0.00999723$ \\
$10_{127}$ & $0.00998882$ \\
\bottomrule
\end{tabular}
\hspace{0.25in}
\end{table}
\begin{table}[h]
\caption{\label{ResidualTablethree}Part 3 of Residuals of Tight Knots and Links by Knot Type}
\begin{tabular}[t]{ll} \toprule
Link & Residual\\ \midrule 
$10_{127}$ & $0.00998882$ \\
$10_{128}$ & $0.000988223$ \\
$10_{129}$ & $0.00902523$ \\
$10_{130}$ & $0.000999987$ \\
$10_{131}$ & $0.00959976$ \\
$10_{132}$ & $0.000980876$ \\
$10_{133}$ & $0.000980018$ \\
$10_{134}$ & $0.00999485$ \\
$10_{135}$ & $0.00100006$ \\
$10_{136}$ & $0.00999149$ \\
$10_{137}$ & $0.000979856$ \\
\bottomrule
\end{tabular}
\hspace{0.25in}
\begin{tabular}[t]{ll} \toprule
Link & Residual\\ \midrule 
$10_{137}$ & $0.000979856$ \\
$10_{138}$ & $0.00899453$ \\
$10_{139}$ & $0.000979731$ \\
$10_{140}$ & $0.0099924$ \\
$10_{141}$ & $0.00100144$ \\
$10_{142}$ & $0.000980204$ \\
$10_{143}$ & $0.00993363$ \\
$10_{144}$ & $0.00995796$ \\
$10_{145}$ & $0.00102699$ \\
$10_{146}$ & $0.00998505$ \\
$10_{147}$ & $0.000999813$ \\
\bottomrule
\end{tabular}
\hspace{0.25in}
\begin{tabular}[t]{ll} \toprule
Link & Residual\\ \midrule 
$10_{147}$ & $0.000999813$ \\
$10_{148}$ & $0.000981385$ \\
$10_{149}$ & $0.00100026$ \\
$10_{150}$ & $0.000979903$ \\
$10_{151}$ & $0.000979813$ \\
$10_{152}$ & $0.00999625$ \\
$10_{153}$ & $0.0091785$ \\
$10_{154}$ & $0.00115132$ \\
$10_{155}$ & $0.00998753$ \\
$10_{156}$ & $0.0009799$ \\
$10_{157}$ & $0.000979535$ \\
\bottomrule
\end{tabular}
\hspace{0.25in}
\begin{tabular}[t]{ll} \toprule
Link & Residual\\ \midrule 
$10_{157}$ & $0.000979535$ \\
$10_{158}$ & $0.000980822$ \\
$10_{159}$ & $0.000979791$ \\
$10_{160}$ & $0.00998455$ \\
$10_{161}$ & $0.00899311$ \\
$10_{162}$ & $0.000985909$ \\
$10_{163}$ & $0.00899697$ \\
$10_{164}$ & $0.000979519$ \\
$10_{165}$ & $0.000979783$ \\
\bottomrule
\end{tabular}
\hspace{0.25in}
\end{table}